\documentclass[11pt]{amsart}
\usepackage{amssymb}
\usepackage{hyperref}
\usepackage{latexsym}
\usepackage[all,cmtip]{xy}
\usepackage{amssymb}
\usepackage{amsmath}
\usepackage{amsthm}
\usepackage{mathrsfs}
\usepackage{amsbsy}
\usepackage{cite}
\usepackage[latin1]{inputenc}

\theoremstyle{plain}
\newtheorem{theorem}{Theorem}[subsection]

\newtheorem{corollary}[theorem]{Corollary}

\newtheorem{lemma}[theorem]{Lemma}

\newtheorem{question}{Question}[section]
\newtheorem{proposition}[theorem]{Proposition}

\theoremstyle{definition}

\newtheorem{remark}[theorem]{Remark}
\newtheorem{definition}[theorem]{Definition}
\newtheorem*{definition-no}{Definition}
\newtheorem{example}[theorem]{Example}

\newtheorem{problem}[question]{Problem}

\begin{document}
\title[Conjugacy and cocycle conjugacy are not Borel]{Conjugacy and cocycle
conjugacy of automorphisms of $\mathcal{O}_{2}$ are not Borel}
\author[Eusebio Gardella]{Eusebio Gardella}
\address{Eusebio Gardella\\
Department of Mathematics\\
Deady Hall, University of Oregon\\
Eugene OR 97403-1222, USA\\
and Fields Institute for Research in Mathematical Sciences\\
222 College Street\\
Toronto ON M5T 3J1, Canada.}
\email{gardella@uoregon.edu}
\urladdr{http://pages.uoregon.edu/gardella/}
\author{Martino Lupini}
\address{Martino Lupini\\
Department of Mathematics and Statistics\\
N520 Ross, 4700 Keele Street\\
Toronto Ontario M3J 1P3, Canada, and Fields Institute for Research in
Mathematical Sciences\\
222 College Street\\
Toronto ON M5T 3J1, Canada.}
\email{mlupini@mathstat.yorku.ca}
\urladdr{http://www.lupini.org/}
\thanks{Eusebio Gardella was supported by the US National Science Foundation
through his thesis advisor's Grant DMS-1101742. Martino Lupini was supported
by the York University Elia Scholars Program. This work was completed when
the authors were attending the Thematic Program on Abstract Harmonic
Analysis, Banach and Operator Algebras at the Fields Institute. The
hospitality of the Fields Institute is gratefully acknowledged.}
\dedicatory{}

\begin{abstract}
The group of automorphisms of the Cuntz algebra $\mathcal{O}_{2}$ is a
Polish group with respect to the topology of pointwise convergence in norm.
Our main result is that the relations of conjugacy and cocycle conjugacy of
automorphisms of $\mathcal{O}_{2}$ are complete analytic sets and, in
particular, not Borel. Moreover, we show that from the point of view of
Borel complexity theory, classifying automorphisms of $\mathcal{O}_{2}$ up
to conjugacy or cocycle conjugacy is strictly more difficult than
classifying up to isomorphism any class of countable structures with Borel
isomorphism relation. In fact the same conclusions hold even if one only
considers automorphisms of $\mathcal{O}_{2}$ of a fixed finite order. In the
course of the proof we will show that the relation of isomorphism of
Kirchberg algebras (with trivial $K_{1}$-group and satisfying the Universal
Coefficient Theorem) is a complete analytic set. Moreover, it is strictly
more difficult to classify Kirchberg algebras (with trivial $K_{1}$-group
and satisfying the Universal Coefficient Theorem) than classifying up to
isomorphism any class of countable structures with Borel isomorphism
relation.
\end{abstract}

\maketitle
\tableofcontents





\section{Introduction}

The \emph{Cuntz algebra }$\mathcal{O}_{2}$ can be described as the universal
unital C*-algebra generated by two isometries $s_{1}$ and $s_{2}$ subject to
the relation 
\begin{equation*}
s_{1}s_{1}^{\ast }+s_{2}s_{2}^{\ast }=1.
\end{equation*}%
It was defined and studied by Cuntz in the groundbreaking paper \cite%
{cuntz_simple_1977}. Since then, a stream of results has made clear the key
role of $\mathcal{O}_{2}$ in the classification theory of C*-algebras; see 
\cite[Chapter 2]{rordam_classification_2002} for a complete account and more
references. This has served as motivation for an intensive study of the
structural properties of $\mathcal{O}_{2}$ and its automorphism group, as in 
\cite{matsumoto_outer_1993, tsui_weakly_1995, conti_endomorphisms_2010,
conti_automorphisms_2011, conti_labeled_2011, conti_weyl_2012,
conti_conjugacy_2013}. In particular, a lot of effort has been put into
trying to classify several important classes of automorphisms; see for
example \cite{izumi_finite_2004, izumi_finite_2004-1}.

If $A$ is a separable C*-algebra, then the group $\text{\textrm{\textrm{Aut}}%
}(A)$ of automorphisms of $A$ is a Polish group with respect to the topology
of pointwise convergence in norm. Two automorphisms $\alpha $ and $\beta $
of $A$ are said to be \emph{conjugate} if there exists an automorphism $%
\gamma $ of $A$ such that 
\begin{equation*}
\gamma \circ \alpha \circ \gamma ^{-1}=\beta .
\end{equation*}%
When $A$ is unital, every unitary $u$ in $A$ defines an automorphism via $%
a\mapsto uau^{\ast }$, and automorphisms of this form are called \emph{inner}%
. The set $\mathrm{\mathrm{Inn}}(A)$ of all inner automorphisms of $A$ is a
normal subgroup of $\mathrm{\mathrm{Aut}}(A)$, and two automorphisms $\alpha 
$ and $\beta $ of $A$ are said to be \emph{cocycle conjugate} if their
images in the quotient $\mathrm{\mathrm{Aut}}(A)/\mathrm{\mathrm{Inn}}(A)$
are conjugate.

Recall that a topological space is said to be \emph{Polish }if it is
separable and its topology is induced by a complete metric. A \emph{Polish
group} is a topological group whose topology is Polish. A \emph{standard
Borel space} is a set endowed with a $\sigma $-algebra which is the $\sigma $%
-algebra of Borel sets for some Polish topology on the space. It is not
difficult to verify that, under the assumption that $A$ is separable, its
automorphism group $\mathrm{\mathrm{Aut}}(A)$ is a Polish group with respect
to the topology of pointwise convergence in norm.

\begin{definition-no}
A subset $B$ of a standard Borel space $X$ is said to be \emph{analytic }if
it is the image of a standard Borel space under a Borel function. \newline
\indent If $B$ and $C$ are analytic subsets of the standard Borel spaces $X$
and $Y$, then $B$ is said to be \emph{Wadge reducible }to $C$ if there is a
Borel map $f\colon X \to Y$ such that $B$ is the inverse image of $C$ under $%
f$; see \cite[Section 2.E]{kechris_classical_1995}. An analytic set which is
moreover a maximal element in the class of analytic sets under Wadge
reducibility is called a \emph{complete analytic set}; more information can
be found in \cite[Section 26.C]{kechris_classical_1995}.
\end{definition-no}

It is a classical result of Souslin from the early beginnings of descriptive
set theory, that there are analytic sets which are not Borel \cite%
{souslin_definition_1917}. In particular --since set that is Wadge reducible
to a Borel set is Borel-- a complete analytic set is not Borel.

The main result of this paper asserts that the relations of conjugacy and
cocycle conjugacy of automorphisms of $\mathcal{O}_{2}$ are complete
analytic sets when regarded as subsets of $\mathrm{\mathrm{Aut}}(\mathcal{O}%
_{2})\times \mathrm{\mathrm{Aut}}(\mathcal{O}_{2})$, and in particular not
Borel.

Informally speaking, a set (or function) is Borel whenever it can be
computed by a countable protocol whose basic bit of information is
membership in open sets. The fact that a set $X$ is not Borel can be
interpreted as the assertion that the problem of membership in $X$ can not
be decided by such a countable protocol, and it is therefore highly
intractable. We can therefore reformulate the main result of this paper as
follows: There does not exist any countable protocol able to determine
whether a given pair of automorphisms of $\mathcal{O}_{2}$ are conjugate or
cocycle conjugate by only looking at any given stage of the computation at
the value of the given automorphisms at some arbitrarily large finite set of
elements of $\mathcal{O}_{2}$ up to some arbitrarily small strictly positive
error.

The fact that conjugacy and cocycle conjugacy of automorphisms of $\mathcal{O%
}_{2}$ are not Borel should be compared with the fact that for any separable
C*-algebra $A$, the relation of unitary equivalence of automorphisms of $A$
is Borel. This is because the relation of coset equivalence modulo the Borel
subgroup $\mathrm{Inn}(A)$ of $\mathrm{\mathrm{Aut}}(A)$. (This does not
necessarily mean that the problem of classifying the automorphisms of $A$ up
to unitary equivalence is more tractable: It is shown in \cite%
{lupini_unitary_2013} that whenever $A$ is simple --or just does not have
continuous trace-- then the automorphisms of $A$ cannot be classified up to
unitary equivalence using countable structures as invariants.) Similarly,
the spectral theorem for unitary operators on the Hilbert space shows that
the relation of conjugacy of unitary operators is Borel; more details can be
found in \cite[Example 55]{foreman_descriptive_2000}. On the other hand, the
main result of \cite{foreman_descriptive_2000} asserts that the relation of
conjugacy for ergodic measure-preserving transformations on the Lebesgue
space is also complete analytic.

We will moreover show that classifying automorphisms of $\mathcal{O}_{2}$ up
to either conjugacy or cocycle conjugacy is strictly more difficult than
classifying any class of countable structures with Borel isomorphism
relation. This statement can be made precise within the framework of
invariant complexity theory. In this context, classification problems are
regarded as equivalence relations on standard Borel spaces. Virtually any
concrete classification problem in mathematics can be regarded --possibly
after a suitable parametrization-- as the problem of classifying the
elements of some standard Borel space up to some equivalence relation. The
key notion of comparison between equivalence relations is the notion of
Borel reduction.

\begin{definition-no}
\label{definition: Borel reduction} Suppose that $E$ and $F$ are equivalence
relation on standard Borel spaces $X$ and $Y$. 
A \emph{Borel reduction} from $E$ to $F$ is a Borel function $f\colon
X\rightarrow Y$ such that%
\begin{equation*}
xEx^{\prime }\text{\quad if and only if\quad }f(x)Ff\left( x^{\prime
}\right) \text{.}
\end{equation*}
\end{definition-no}

A Borel reduction from $E$ to $F$ can be regarded as a way to assign --in a
constructive way-- to the objects of $X$, equivalence classes of $F$ as
complete invariants for $E$.

\begin{definition-no}
The equivalence relation $E$ is said to be \emph{Borel reducible }to $F$, in
symbol $E\leq _{B}F$, if there is a Borel reduction from $E$ to $F$.
\end{definition-no}

In this case, the equivalence relation $F$ can be thought of as being more
complicated than $E$, since any Borel classification of the objects of $Y$
up to $F$ entails -by precomposing with a Borel reduction from $E$ to $F$- a
Borel classification of objects of $X$ up to $E$. It is immediate to check
that if $E$ is Borel reducible to $F$, then $E$ (as a subset of $X\times X$)
is Wadge reducible to $F$ (as a subset of $Y\times Y$). In particular, if $E$
is a complete analytic set and $E\leq _{B}F$, then $F$ is a complete
analytic set. Observe that if $F$ is an equivalence relation on $Y $, and $X$
is an $F $-invariant Borel subset of $Y$, then the restriction of $F$ to $X$
is Borel reducible to $F$.

Using this terminology, we can reformulate the assertion about the
complexity of the relations of conjugacy and cocycle conjugacy of
automorphisms of $\mathcal{O}_{2}$ as follows. If $\mathcal{C}$ is any class
of countable structures such that the corresponding isomorphism relation $%
\cong _{\mathcal{C}}$ is Borel, then $\cong _{\mathcal{C}}$ is Borel
reducible to both conjugacy and cocycle conjugacy of automorphisms of $%
\mathcal{O}_{2}$. Furthermore, if $E$ is any Borel equivalence relation,
then the relations of conjugacy and cocycle conjugacy of automorphisms of $%
\mathcal{O}_{2}$ are \emph{not }Borel reducible to $E$. In particular this
rules out any classification that uses as invariant Borel measures on a
Polish space (up to measure equivalence) or unitary operators on the Hilbert
space (up to conjugacy). In fact, as observed before, the relations of
measure equivalence and, by the spectral theorem, the relation of conjugacy
of unitary operators are Borel; see \cite[Example 55]%
{foreman_descriptive_2000}.

All the results mentioned so far about the complexity of the relation of
conjugacy and cocycle conjugacy of automorphisms of $\mathcal{O}_{2}$ will
be shown to hold even if one only considers automorphisms of a fixed finite
order. Moreover, it will follow from the argument that the same assertions
hold for the relation of isomorphism of Kirchberg algebras (with trivial $%
K_{0}$-group and satisfying the Universal Coefficient Theorem). It also
follows from our constructions and \cite[Theorem 1.11]%
{ellis_classification_2010} that, for every $n\in \mathbb{\mathbb{N}}$, the
relation of isomorphisms of unital AF-algebras with $K_{0}$-group of rank $%
n+1$ is strictly more complicate than the relation of isomorphism of unital
AF-algebras with $K_{0}$-groups of rank $n$.

It should be mentioned that it is a consequence of the main result of \cite%
{kerr_Borel_2014} that the automorphisms of $\mathcal{O}_{2}$ are not
classifiable up to conjugacy by countable structures. This means that there
is no explicit way to assign a countable structure to every automorphism of $%
\mathcal{O}_{2}$, in such a way that two automorphisms are conjugate if and
only if the corresponding structures are isomorphic. More precisely, for no
class $\mathcal{C}$ of countable structures, is the relation of conjugacy of
automorphisms of $\mathcal{O}_{2}$ Borel reducible to the relation of
isomorphisms of elements of $\mathcal{C}$. Moreover the same conclusions
hold for any set of automorphisms of $\mathcal{O}_{2}$ which is not meager
in the topology of pointwise convergence. Similar conclusions hold for
automorphisms of any separable C*-algebra absorbing the Jiang-Su algebra
tensorially.

The strategy of the proof of the main theorem is as follows. Using
techniques from \cite{hjorth_isomorphism_2002, downey_isomorphism_2008}, we
show that for every prime number $p$, the relation of isomorphism of
countable $p$-divisible torsion free abelian groups is a complete analytic
set, and it is strictly more complicated than the relation of isomorphism of
any class of countable structures with Borel isomorphism relation. We then
show that the relation of isomorphism of $p$-divisible abelian groups is
Borel reducible to the relations of conjugacy and cocycle conjugacy of
automorphisms of $\mathcal{O}_{2}$ of order $p$.

This is achieved by showing that there is a Borel way to assign to a
countable abelian group $G$ to assign to a countable abelian group a
Kirchberg algebra $A_{G}$ with trivial $K_{1}$-group, $K_{0}$-group
isomorphic to $G$, and with the class of the unit in $K_{0}$ being the zero
element. Adapting a construction of Izumi from \cite{izumi_finite_2004}, we
define an automorphism $\nu _{p}$ of $\mathcal{O}_{2}$ of order $p$ with the
following property: Tensoring the identity automorphism of $A_{G}$ by $\nu
_{p}$, and identifying $A_{G}\otimes \mathcal{O}_{2}$ with $\mathcal{O}_{2}$
by Kirchberg's absorption theorem, gives a reduction of isomorphism of
Kirchberg algebras with $p$-divisible $K_{0}$-group and with the class of
the unit being the trivial element in $K_{0}$, to conjugacy and cocycle
conjugacy of automorphisms of $\mathcal{O}_{2}$ of order $p$. The proof is
concluded by showing --using results from \cite{farah_turbulence_2014}--
that such reduction is implemented by a Borel map.\newline
\ \newline
\indent The present paper is organized as follows. Section \ref{Section:
parametrising} presents a functorial version of the notion of standard Borel
parametrization of a category as defined in \cite{farah_turbulence_2014}.
Several functorial parametrizations for the category are then presented and
shown to be equivalent. Finally, many standard constructions in C*-algebra
theory are shown to be computable by Borel maps in these parametrizations.
The main result of Section 3 asserts that the reduced crossed product of a
C*-algebra by an action of a countable discrete group can be computed in a
Borel way. The same conclusion holds for crossed products by a corner
endomorphism in the sense of \cite{boyd_faithful_1993}. Section \ref%
{Chapter: Borel selection of AF-algebras} provides a Borel version of the
correspondence between unital AF-algebras and dimension groups established
in \cite{elliott_classification_1976, effros_dimension_1980}. We show that
there is a Borel map that assigns to a dimension group $D$, a unital
AF-algebra $B_{D}$ such that $D$ is isomorphic to the $K_{0}$-group of $%
B_{D} $. 
Moreover, given an endomorphism $\beta $ of $D$, one can select in a Borel
fashion an endomorphism $\rho _{D,\beta }$ of $B_{D}$ whose induced
endomorphism of $K_{0}(B_{D})$ is conjugate to $\beta $. Finally, Section %
\ref{Chapter: cocycle conjugacy of automorphisms} contains the proof of the
main results.\newline
\ \newline
\indent

In the following, all C*-algebras and Hilbert spaces are assumed to be \emph{%
separable}, and all discrete groups are assumed to be \emph{countable}. We
denote by $\omega $ the set of natural numbers \emph{including }$0$. An
element $n\in \omega $ will be identified with the set $\left\{ 0,1,\ldots
,n-1\right\} $ of its predecessors. (In particular $0$ is identified with
the empty set.) We will therefore write $i\in n$ to mean that $i$ is a
natural number and $i<n$.

For $n\geq 1$, we write $\mathbb{Z}_n$ for the cyclic group $\mathbb{Z}/ n%
\mathbb{Z}$.

If $X$ is a Polish space and $D$ is a countable set, we endow the set $X^{D}$
of $D$-indexed sequences of elements of $X$ with the product topology.
Likewise, if $X$ is a standard Borel space, then we give $X^{D}$ the product
Borel structure. In the particular case where $X=2=\left\{ 0,1\right\} $, we
identify $2^{D}$ with the set of subsets of $D$ with its Cantor set
topology, and the corresponding standard Borel structure. In the following
we will often make use --without explicit mention-- of the following basic
principle: Suppose that $X$ is a standard Borel space, $D$ is a countable
set, and $B$ is a Borel subset of $X\times D$ such that for every $x\in X$
there is $y\in D$ such that $\left( x,y\right) \in B$. Then there is a Borel
selector for $B$, this is,\ a function $f$ from $X$ to $D$ such that $\left(
x,f(x)\right) \in B$ for every $x\in X$. To see this one can just fix a well
order $<$ of $D$ and define $f(x)$ to be the $<$-minimum of the set of $y\in
D$ such that $\left( x,y\right) \in B$.

Moreover we will use throughout the paper the fact that a $G_{\delta }$
subspace of a Polish space is Polish in the subspace topology \cite[Theorem
3.11]{kechris_classical_1995}, and that a Borel subspace of a standard Borel
space is standard with the inherited Borel structure \cite[Proposition 12.1]%
{kechris_classical_1995}.

We have tried to make this paper accessible to operator algebraists who are
not familiar with descriptive set theory, as well as set theorists who are
not familiar with C*-algebras.

The authors would like to thank Samuel Coskey and Ilijas Farah for several
helpful conversations.

\section{Parametrizing the category of C*-algebras\label{Section:
parametrising}}

\subsection{Background on C*-algebras and notation}

For a Hilbert space $H$, we denote by $B(H)$ the algebra of bounded
operators on $H$, and by $\mathcal{K}(H)$ the algebra of compact operators
on $H$. The set of $T\in B(H)$ of operator norm at most $1$ is denoted by $%
B_{1}(H)$. The weak operator topology on $B(H)$ is the weakest topology
making the functions%
\begin{eqnarray*}
B(H) &\mapsto &\mathbb{C} \\
x &\mapsto &\left\langle x\xi ,\eta \right\rangle
\end{eqnarray*}%
for $\xi ,\eta \in H$ continuous. Recall that addition and scalar
multiplication are jointly weakly continuous on $B(H)$, while composition of
operators is only separately continuous in each variable; see \cite[I.3.2.1]%
{blackadar_operator_2006}. The unit ball $B_{1}(H)$ of $B(H)$ is compact
when endowed with the weak topology; see \cite[I.3.2.4]%
{blackadar_operator_2006}.

We denote by $U(H)$ the group of unitaries in $H$. It is easily checked that 
$U(H)$ is a $G_{\delta }$ subset of $B(H)$ with respect to the weak
topology. Therefore the weak topology makes $U(H)$ a Polish group by \cite[%
Corollary 9.5]{kechris_classical_1995}. It is well known that on $U(H)$ the
weak topology coincides with several other operator topologies, such as the
weak, strong, $\sigma $-weak, and $\sigma $-strong operator topology; see 
\cite[I.3.2.9]{blackadar_operator_2006}.\newline
\ \newline
\indent A\emph{\ C*-algebra} is a subalgebra of the algebra $B(H)$ of
bounded linear operators on a Hilbert space $H$ that is closed in the norm
topology and contains the adjoint of any of its elements. In particular $%
B(H) $ is itself a (nonseparable) C*-algebra, and $K(H)$ is a separable
C*-algebra.\newline
\indent Equivalently, C*-algebras can be abstractly characterized as those
Banach *-algebras $A$ whose norm satisfies the \emph{C*-identity}%
\begin{equation*}
\left\Vert a^{\ast }a\right\Vert =\left\Vert a\right\Vert ^{2}\text{{}}
\end{equation*}%
for all $a$ in $A$. A C*-algebra is \emph{unital }if it contains a
multiplicative identity (called \emph{unit}) usually denoted by $1$. An
ideal of a C*-algebra $A$ is an ideal of $A$ in the ring-theoretic sense. A
C*-algebra is \emph{simple }if it contains no nontrivial closed ideals \cite[%
II.5.4.1]{blackadar_operator_2006}. If $A$ and $B$ are C*-algebras, a \emph{%
*-homomorphism} from $A$ to $B$ is an algebra homomorphism $\varphi \colon
A\rightarrow B$ satisfying $\varphi (a^{\ast })=\varphi (a)^{\ast }$ for all 
$a$ in $A$. It is a classical result of the theory of C*-algebras \cite[%
II.1.6.6]{blackadar_operator_2006} that if $\varphi \colon A\rightarrow B$
is a *-homomorphism, then $\varphi $ is contractive, this is, $\left\Vert
\varphi (a)\right\Vert \leq \left\Vert a\right\Vert $ for every $a\in A$,
and moreover $\varphi $ is isometric if and only if it is injective. As a
consequence, the range of any *-homomorphism is automatically closed. A 
\emph{*-isomorphism} between $A$ and $B$ is a bijective *-homomorphism. Note
that *-isomorphisms are necessarily isometric. A \emph{representation} of a
C*-algebra $A$ on a Hilbert space $H$ is a *-homomorphism from $A$ to $B(H)$%
. A representation is \emph{faithful }if it is injective or --equivalently--
isometric.

An \emph{automorphism} of a C*-algebra $A$ is a *-isomorphism from $A$ to $A$%
. The set of all automorphisms of $A$, denoted by $\mathrm{\mathrm{Aut}}(A)$%
, is a Polish group under composition when endowed with the topology of
pointwise norm convergence. In this topology, a sequence $(\varphi
_{n})_{n\in \omega }$ in $\mathrm{\mathrm{Aut}}(A)$ converges to an
automorphism $\varphi $ if and only if%
\begin{equation*}
\lim_{n\rightarrow \infty }\left\Vert \varphi _{n}(a)-\varphi (a)\right\Vert
=0
\end{equation*}%
for every $a\in A$. Two automorphisms of $A$ are said to be \emph{conjugate }%
if they are conjugate elements of $\mathrm{\mathrm{Aut}}(A)$. \newline
\indent If $A$ is unital, an element $u$ of $A$ is said to be a \emph{%
unitary }if $uu^{\ast }=u^{\ast }u=1$. The unitary elements of $A$ form a
group under multiplication, denoted by $U(A)$. Any unitary element $u$ of $A$
defines an automorphism $\mathrm{Ad}(u)$ of $A$, which is given by%
\begin{equation*}
\mathrm{Ad}(u)(a)=uau^{\ast }\text{{}}
\end{equation*}%
for all $a$ in $A$. Automorphisms of this form are called \emph{inner}, and
form a normal subgroup $\mathrm{Inn}(A)$ of $\mathrm{\mathrm{Aut}}(A)$. 

\begin{definition}
Let $G$ be a countable discrete group, and let $A$ be a C*-algebra. An \emph{%
action} $\alpha $ of $G$ on $A$ is a group homomorphism $g\mapsto \alpha
_{g} $ from $G$ to the group $\mathrm{\mathrm{Aut}}(A)$ of automorphisms of $%
A$.\newline
\indent Two actions $\alpha $ and $\beta $ of $G$ on $A$ are said to be 
\emph{conjugate} if there is $\gamma \in \mathrm{\mathrm{Aut}}(A)$ such that 
\begin{equation*}
\gamma \circ \alpha _{g}\circ \gamma ^{-1}=\beta _{g}
\end{equation*}%
for every $g\in G$.
\end{definition}

Let $A$ be a unital C*-algebra and let $\alpha $ be an action of $G$ on $A$.
An $\alpha $-\emph{cocycle }is a function $u\colon G\rightarrow U(A)$
satisfying 
\begin{equation*}
u_{gh}=u_{g}\alpha _{g}(u_{h})
\end{equation*}%
for every $g,h\in G$. \newline
\indent If $u$ is an $\alpha $-cocycle, we define the $u$-\emph{perturbation}
of $\alpha $, denoted $\alpha ^{u}\colon G\rightarrow \text{\textrm{\textrm{%
Aut}}}(A)$, by 
\begin{equation*}
\alpha _{g}^{u}=\mathrm{Ad}(u_{g})\circ \alpha _{g}
\end{equation*}%
for $g$ in $G$. \newline
\indent Two actions $\alpha $ and $\beta $ of $G$ on $A$ are said to be 
\emph{cocycle conjugate} if $\beta $ is conjugate to a perturbation of $%
\alpha $ by a cocycle.

It is not hard to check that the relation of cocycle conjugacy is an
equivalence relation for actions. In the case when $G$ is the group of
integers $\mathbb{Z}$, actions of $\mathbb{Z}$ on $A$ naturally correspond
to single automorphisms of $A$. Similarly, if $G$ is the group $\mathbb{Z}%
_{n}$, then actions of $\mathbb{Z}_n$ on $A$ correspond to automorphisms of $%
A$ whose order divides $n$. We show in Lemma \ref{Lemma: whaaat?} below that
the notions of conjugacy and cocycle conjugacy for actions and automorphisms
are respected by this correspondence when $A$ has trivial center. These
observations will be used to infer Corollary \ref{Corollary: not Borel} from
Corollary \ref{Corollary: reduction of iso of p-divisible}.

\begin{lemma}
\label{Lemma: whaaat?} Suppose that $\alpha$ and $\beta $ are automorphisms
of a unital C*-algebra $A$.

\begin{enumerate}
\item Then the following statements are equivalent:

\begin{enumerate}
\item The actions $n\mapsto \alpha ^{n}$ and $n\mapsto \beta ^{n}$ of $%
\mathbb{Z}$ on $A$ are cocycle conjugate;

\item There are an automorphism $\gamma $ of $A$ and a unitary $u$ of $A$
such that $\mathrm{Ad}(u)\circ \alpha =\gamma \circ \beta \circ \gamma ^{-1}$%
.
\end{enumerate}

\item Assume moreover that $\alpha $ and $\beta $ have order $k\geq 2$ and
that $A$ has trivial center (for example, if $A$ is simple). Then the
following statements are equivalent:

\begin{enumerate}
\item The actions $n\mapsto \alpha ^{n}$ and $n\mapsto \beta ^{n}$ of $%
\mathbb{Z}_k$ on $A$ are cocycle conjugate;

\item The actions $n\mapsto \alpha ^{n}$ and $n\mapsto \beta ^{n}$ of $%
\mathbb{Z}$ on $A$ are cocycle conjugate.
\end{enumerate}
\end{enumerate}
\end{lemma}

\begin{proof}
(1). To show that (a) implies (b), simply take the unitary $u=u_{1}$ coming
from the $\alpha $-cocycle $u\colon \mathbb{Z}\rightarrow U(A)$. \newline
\indent Conversely, if $u$ is a unitary in $A$ as in the statement, we
define an $\alpha $-cocycle as follows. Set $u_{0}=1$ and $u_{1}=u$, and for 
$n\geq 2$ define $u_{n}$ inductively by $u_{n}=u_{1}\alpha (u_{n-1})$. Set $%
u_{-1}=\alpha ^{-1}(u_{1}^{\ast })$, and for $n\leq -2$, define $u_{n}$
inductively by $u_{n}=u_{-1}\alpha ^{-1}(u_{n+1})$. It is straightforward to
check that $n\mapsto u_{n}$ is an $\alpha $-cocycle, and that the
automorphism $\gamma $ in the statement implements the conjugacy between $%
\alpha ^{u}$ and $\beta $.\newline
\indent(2). To show that (a) implies (b), it is enough to note that if $%
u\colon \mathbb{Z}_{k}\rightarrow U(A)$ is an $\alpha $-cocycle, when we
regard $\alpha $ as a $\mathbb{Z}_{k}$ action, then the sequence $%
(v_{m})_{m\in \mathbb{N}}$ of unitaries in $A$ given by $v_{m}=u_{n}$ if $%
m=n $ mod $k$, is an $\alpha $-cocycle, when we regard $\alpha $ as a $%
\mathbb{Z} $ action.\newline
\indent Assume that $\alpha $ and $\beta $ are cocycle conjugate as
automorphisms of $A$. Let $(u_{n})_{n\in \mathbb{N}}$ be an $\alpha $%
-cocycle and let $\gamma $ be an automorphism implementing the conjugacy.
Fix $n$ in $\mathbb{N}$, and write $n=km+r$ for uniquely determined $k\in 
\mathbb{Z}$ and $r\in k$. Since $\alpha $ and $\beta $ have order $k$, we
have 
\begin{align*}
\text{$\mathrm{Ad}$}(u_{km+r})\circ \alpha ^{r}& =\text{$\mathrm{Ad}$}%
(u_{km+r})\circ \alpha ^{km+r} \\
& =\gamma \circ \beta ^{km+r}\circ \gamma ^{-1} \\
& =\gamma \circ \beta ^{r}\circ \gamma ^{-1} \\
& =\text{$\mathrm{Ad}$}(u_{r})\circ \alpha ^{r}.
\end{align*}%
In particular, $\text{\textrm{Ad}}(u_{n+mk})=\text{\textrm{Ad}}(u_{n})$, so $%
u_{n+mk}$ and $u_{n}$ differ by a central unitary. Since the center of $A$
is trivial, upon correcting by a scalar, we may assume that $u_{n+mk}=u_{n}$.%
\newline
\indent Thus, the assignment $v\colon \mathbb{Z}_{k}\rightarrow U(A)$ given
by $n\mapsto u_{n}$ is an $\alpha $-cocycle, when we regard $\alpha $ as a $%
\mathbb{Z}_{k}$ action, and $\gamma $ implements an conjugacy between the $%
\mathbb{Z}_{k}$ actions $\alpha ^{v}$ and $\beta $. This finishes the proof.
\end{proof}

More generally, one can define actions of locally compact groups on
C*-algebras, as well as conjugacy and cocycle conjugacy for such actions.
More details can be found in \cite[Section II.1]{blackadar_operator_2006}.

If $A$ and $B$ are C*-algebras, the tensor product of $A$ and $B$ as complex
algebras with involution is denoted by $A\odot B$ and called \emph{algebraic
tensor product }of $A$ and $B$ \cite[II.9.1.1]{blackadar_operator_2006}. A
C*-algebra $A$ is \emph{nuclear} or \emph{amenable} if, for any other
C*-algebra $B$, the algebraic tensor product $A\odot B$ bears a unique
C*-norm, this is, a not necessarily complete norm satisfying%
\begin{equation*}
\left\Vert xy\right\Vert \leq \left\Vert x\right\Vert \left\Vert y\right\Vert
\end{equation*}%
and%
\begin{equation*}
\left\Vert x^{\ast }x\right\Vert =\left\Vert x\right\Vert ^{2}\text{.}
\end{equation*}%
The completion of $A\odot B$ with respect to such norm is called the\emph{\
(C*-algebra) tensor product }of $A$ and $B$ and denoted by $A\otimes B$. All
the C*-algebras considered in Section \ref{Chapter: cocycle conjugacy of
automorphisms} will be nuclear, so their tensor product $A\otimes B$ will be
well defined. It is hard to overestimate the importance of nuclearity in the
theory of C*-algebras. Nuclearity is the analog for C*-algebras of
amenability for groups and Banach algebras. Nuclear C*-algebras admit
several equivalent characterizations; see \cite[IV.3.1.5, IV.3.1.6, IV.3.1.12%
]{blackadar_operator_2006}. Moreover they constitute the main focus of
Elliott classification program \cite[Section 2.2]{rordam_classification_2002}%
.

\subsection{Functorial parametrization\label{Section: Functorial
parametrization}}

Recall that a \emph{(small) semigroupoid} is a quintuple $\left( X,\mathcal{C%
}_{X},s,r,\cdot \right) $, where $X$ and $\mathcal{C}_{X}$ are sets, $s,r$
are functions from $\mathcal{C}_{X}$ to $X$, and $\cdot $ is an associative
partially defined binary operation on $\mathcal{C}_{X}$ with domain%
\begin{equation*}
\left\{ (x,y)\in \mathcal{C}_{X}\times \mathcal{C}_{X}\colon
r(x)=s(y)\right\}
\end{equation*}%
such that $r(x\cdot y)=r(y)$ and $s(x\cdot y)=s(y)$ for all $x$ and $y$ in $%
X $. The elements of $X$ are called objects, the elements of $\mathcal{C}%
_{X} $ morphisms, the map $\cdot $ composition, and the maps $s$ and $r$
source and range map. In the following, a semigroupoid $\left( X,\mathcal{C}%
_{X},s,r,\cdot \right) $ will be denoted simply by $\mathcal{C}_{X}$. Note
that a (small) category is precisely a (small) semigroupoid, where moreover
the identity arrow $\mbox{id}_{x}\in \mathcal{C}_{X}$ is associated with the
element $x$ of $X$. A morphism between semigroupoids $\mathcal{C}_{X}$ and $%
\mathcal{C}_{X^{\prime }}$ is a pair $\left( f,F\right) $ of functions $%
f\colon X\rightarrow X^{\prime }$ and $F\colon \mathcal{C}_{X}\rightarrow 
\mathcal{C}_{X^{\prime }}$ such that

\begin{itemize}
\item $s_{X^{\prime }}\circ F=f\circ s_X$,

\item $r_{X^{\prime }}\circ F=f\circ r_X$, and

\item $F\left( a\cdot b\right) =F(a)\cdot F\left( b\right) $ for every $a$
and $b\in \mathcal{C}_{X}$.
\end{itemize}

In the case of categories, a morphism of semigroupoids is just a functor.

A \emph{standard Borel semigroupoid }is a semigroupoid $\mathcal{C}_{X}$
such that $X$ and $\mathcal{C}_{X}$ are endowed with standard Borel
structures making the source and range functions $s$ and $r$ Borel.

\begin{definition}
Let $\mathcal{D}$ be a category, let $\mathcal{C}_{X}$ be a standard Borel
semigroupoid, and let $\left( f,F\right) $ be a morphism from $\mathcal{C}%
_{X}$ to $\mathcal{D}$. We say that $\left( \mathcal{C}_{X},f,F\right) $ is
a \emph{good parametrization }of $\mathcal{D}$ if

\begin{itemize}
\item $\left( f,F\right) $ is \emph{essentially surjective}, this is, if
every object of $\mathcal{D}$ is isomorphic to an object in the range of $f$,

\item $\left( f,F\right) $ is \emph{full}, this is, if for every $x,y\in X$
the set $\mathrm{Hom}(f(x),f(y))$ is contained in the range of $F$, and

\item the set $\mathrm{Iso}_{X}$ of elements of $\mathcal{C}_{X}$ that are
mapped by $F$ to isomorphisms of $\mathcal{D}$, is Borel.
\end{itemize}
\end{definition}

Observe that if $\left( \mathcal{C}_{X},f,F\right) $ is a \emph{good
parametrization }of $\mathcal{D}$, then $\left( X,f\right) $ is a good
parametrization of $\mathcal{C}$ in the sense of \cite[Definition 2.1]%
{farah_turbulence_2014}.

\begin{definition}
\label{Definition: equivalent parametrizations} Let $\mathcal{D}$ be a
category and let $\left( \mathcal{C}_{X},f,F\right) $ and $\left( \mathcal{C}%
_{X^{\prime }},f^{\prime },F^{\prime }\right) $ be good parametrizations of $%
\mathcal{D}$. A \emph{morphism }from $\left( \mathcal{C}_{X},f,F\right) $ to 
$\left( \mathcal{C}_{X^{\prime }},f^{\prime },F^{\prime }\right) $ is a
triple $\left( g,G,\eta \right) $ of maps $g\colon X\rightarrow X^{\prime }$%
, $G\colon \mathcal{C}_{X}\rightarrow \mathcal{C}_{X^{\prime }}$, and $\eta
\colon X\rightarrow \mathcal{D}$, satisfying the following conditions:

\begin{enumerate}
\item The functions $g$ and $G$ are Borel;

\item $\eta (x)$ is an isomorphism from $f(x)$ to $\left( f^{\prime }\circ
g\right) (x)$ for every $x\in X$;

\item The pair $\left( f^{\prime }\circ g,F^{\prime }\circ G\right) $ is a
semigroupoid morphism $\mathcal{C}_X\to\mathcal{D}$;

\item We have $s_{X^{\prime }}\circ G=g\circ s_X$ and $r_{X^{\prime }}\circ
G=g\circ r_X$;

\item For every $x\in X$, the morphism $\eta (s(x))$ is an isomorphism from $%
F(x)$ to $\left( F^{\prime }\circ G\right) (x)$;

\item For every $a\in \mathcal{C}_{X}$%
\begin{equation*}
F^{\prime }\left( G(a)\right) \circ \eta \left( s(a)\right) =\eta \left(
r(a)\right) \circ F(a)\text{.}
\end{equation*}
\end{enumerate}
\end{definition}

Two good parametrizations $\left( \mathcal{C}_{X},f,F\right) $ and $\left( 
\mathcal{C}_{X^{\prime }},f^{\prime },F^{\prime }\right) $ of $\mathcal{C}$
are said to be \emph{equivalent }if there are isomorphisms from $\left( 
\mathcal{C}_{X},f,F\right) $ to $\left( \mathcal{C}_{X^{\prime }},f^{\prime
},F^{\prime }\right) $ and viceversa. It is not difficult to verify that if $%
\left( \mathcal{C}_{X},f,F\right) $ and $\left( \mathcal{C}_{X^{\prime
}},f^{\prime },F^{\prime }\right) $ are equivalent parametrizations of $%
\mathcal{D}$, then $\left( X,f\right) $ and $\left( X^{\prime },f^{\prime
}\right) $ are weakly equivalent parametrizations of $\mathcal{D}$ in the
sense of \cite[Definition 2.1]{farah_turbulence_2014}.

In the following, a good parametrization $\left( \mathcal{C}_{X},f,F\right) $
of $\mathcal{D}$ will be denoted by $\mathcal{C}_{X}$ for short.

\subsection{The space $\mathcal{C}_{\widehat{\Xi }}$\label{Subsection: XI
hat}}

We follow the notation in \cite[Section 2.2]{farah_turbulence_2014}, and
denote by $\mathbb{Q}(i)$ the field of complex rationals. A $\mathbb{Q}%
\left( i\right) $-$\ast $-algebra is an algebra over the field $\mathbb{Q}%
(i) $ endowed with an involution $x\mapsto x^{\ast }$. We define $\mathcal{U}
$ to be the $\mathbb{Q}(i)$-$\ast $-algebra of noncommutative $\ast $%
-polynomials with coefficients in $\mathbb{Q}(i)$ and without constant term
in the formal variables $X_{k}$ for $k\in \omega $. If $A$ is a C*-algebra, $%
\gamma =(\gamma _{n})_{n\in \omega }$ is a sequence of elements of $A$, and $%
p\in \mathcal{U}$, we define $p(\gamma )$ to be the element of $A$ obtained
by evaluating $p$ in $A$, where for every $k\in \omega $, the formal
variables $X_{k}$ and $X_{k}^{\ast }$ are replaced by the elements $\gamma
_{k}$ and $\gamma _{k}^{\ast }$ of $A$.

We denote by $\widehat{\Xi }$ the set of elements 
\begin{equation*}
A=\left( f,g,h,k,r\right) \in \omega ^{\omega \times \omega }\times \omega ^{%
\mathbb{Q}(i)\times \omega }\times \omega ^{\omega \times \omega }\times
\omega ^{\omega }\times \mathbb{R}^{\omega }
\end{equation*}%
that code on $\omega $ a structure of $\mathbb{Q}(i)$-$\ast $-algebra $A$
endowed with a norm satisfying the C*-identity. The completion $\widehat{A}$
of $\omega $ with respect to such norm is a C*-algebra (denoted by $B(A)$ in 
\cite[Subection 2.4]{farah_turbulence_2014}). It is not hard to check that $%
\widehat{\Xi }$ is a $G_{\delta }$ subspace of $\omega ^{\omega \times
\omega }\times \omega ^{\mathbb{Q}(i)\times \omega }\times \omega ^{\omega
\times \omega }\times \omega ^{\omega }\times \mathbb{R}^{\omega }$, and
hence Polish with the subspace topology. As observed in \cite[Subection 2.4]%
{farah_turbulence_2014}, $\widehat{\Xi }$ can be thought of as a natural
parametrization for \emph{abstract C*-algebras}. We use the notation of \cite%
[Subsection 2.4]{farah_turbulence_2014} to denote the operations on $\omega $
coded by an element $A=\left( f,g,h,k,r\right) $ of $\widehat{\Xi }$. We
denote by $d_{A}$ the metric on $\omega $ coded by $A$, which is given by 
\begin{equation*}
d_{A}\left( n,m\right) =\left\Vert n+_{f}(-1)\cdot _{g}m\right\Vert _{r}
\end{equation*}%
for $n,m\in \omega $. We will also write $n+_{A}m$ for $n+_{f}m$, and
similarly for $g,h,k,r$.

\begin{definition}
Suppose that $A=\left( f,g,h,k,r\right) $ and $A^{\prime }=\left( f^{\prime
},g^{\prime },h^{\prime },k^{\prime },r^{\prime }\right) $ are elements of $%
\widehat{\Xi }$, and that $\Phi =(\Phi_n) _{n\in \omega }\in \left( \omega
^{\omega }\right) ^{\omega }$ is a sequence of functions from $\omega $ to $%
\omega $. We say that $\Phi $ is a \emph{code for a *-homomorphism} from $%
\widehat{A}$ to $\widehat{A}^{\prime }$ if the following conditions hold:

\begin{enumerate}
\item The sequence $\left( \Phi _{n}(k) \right) _{n\in \omega }$ is Cauchy
uniformly in $k\in \omega $ with respect to the metric $d_A$, and in
particular converges to an element $\widehat{\Phi }(k) $ of $\widehat{A}$;

\item The map $k\mapsto \widehat{\Phi }(k)$ is a contractive *-homomorphism
of $\mathbb{Q}(i)$-$\ast $-algebras, and hence it induces a *-homomorphism $%
\widehat{\Phi }$ from $\widehat{A}$ to $\widehat{A}^{\prime }$.
\end{enumerate}

We say that $\Phi $ is a \emph{code for an isomorphism} from $\widehat{A}$
to $\widehat{A}^{\prime }$ if $\Phi $ is a code for a *-homomorphism from $%
\widehat{A}$ to $\widehat{A}^{\prime }$, and $\widehat{\Phi }$ is an
isomorphism. If $\Phi $ and $\Phi ^{\prime }$ are codes for *-homomorphisms
from $\widehat{A}$ to $\widehat{A}^{\prime }$ and from $\widehat{A}^{\prime
} $ to $\widehat{A}^{\prime \prime }$ respectively, we define their
composition $\Phi ^{\prime }\circ \Phi $, which will be a code for a
*-homomorphism from $\widehat{A}$ to $\widehat{A^{\prime \prime }}$, by $%
\left( \Phi ^{\prime }\circ \Phi \right) _{n}=\Phi _{n}^{\prime }\circ \Phi
_{n}$ for $n\in \omega $.
\end{definition}

\begin{remark}
It is easily checked that $\Phi ^{\prime }\circ \Phi \in \left( \omega
^{\omega }\right) ^{\omega }$ is a code for the *-homomorphism $\widehat{%
\Phi }^{\prime }\circ \widehat{\Phi }$ from $\widehat{A}$ to $\widehat{A}%
^{\prime \prime }$.
\end{remark}

One can verify that the set $\mathcal{C}_{\widehat{\Xi }}$ of triples $%
\left( A,A^{\prime },\Phi \right) \in \widehat{\Xi }\times \widehat{\Xi }%
\times \left( \omega ^{\omega }\right) ^{\omega }$ such that $\Phi $ is a
code for a *-homomorphism from $\widehat{A}$ to $\widehat{A}^{\prime }$, is
Borel. We can regard $\mathcal{C}_{\widehat{\Xi }}$ as a standard
semigroupoid having $\widehat{\Xi }$ as set of objects, where the
composition of $\left( A,A^{\prime },\Phi \right) $ and $\left( A^{\prime
},A^{\prime \prime },\Phi ^{\prime }\right) $ is $\left( A^{\prime
},A^{\prime \prime },\Phi ^{\prime }\circ \Phi \right) $, and the source and
range of $\left( A,A^{\prime },\Phi \right) $ are $A$ and $A^{\prime }$
respectively. The semigroupoid morphism $\left( A,A^{\prime },\Phi \right)
\mapsto \left( \widehat{A},\widehat{A}^{\prime },\widehat{\Phi }\right) $
defines a parametrization of the category of C*-algebras with
*-homomorphisms. It is easy to see that this is a good parametrization in
the sense of Definition \ref{Definition: equivalent parametrizations}. In
particular, the set $\mathrm{Iso}_{\widehat{\Xi }}$ of elements $\left(
A,A^{\prime },\Phi \right) $ of $\widehat{\Xi }\times \widehat{\Xi }\times
\left( \omega ^{\omega }\right) ^{\omega }$ such that $\Phi $ is a code for
an isomorphism from $\widehat{A}$ to $\widehat{A}^{\prime }$, is Borel.

\subsection{The space $\mathcal{C}_{\Xi }$\label{Subsection: XI}}

We denote by $\Xi $ the $G_{\delta }$ subset of $\mathbb{R}^{\mathcal{U}}$
consisting of the nonzero functions $\delta \colon \mathcal{U}\rightarrow 
\mathbb{R}$ such that there exists a C*-algebra $A$ and a dense subset $%
\gamma =\left( \gamma _{n}\right) _{n\in \omega }$ of $A$, such that 
\begin{equation*}
\delta (p)=\left\Vert p(\gamma )\right\Vert .
\end{equation*}%
It could be observed that, differently from \cite[Subsection 2.3]%
{farah_turbulence_2014}, we are not considering the function constantly
equal to zero as an element of $\Xi $; this choice is just for convenience
and will play no role in the rest of the discussion. Observe that any
element $\delta $ of $\Xi $ determines a seminorm on the $\mathbb{Q}(i)$-
*-algebra $\mathcal{U}$; therefore one can consider the corresponding
Hausdorff completion of $\mathcal{U}$. Denote by $I_{\delta }$ the ideal of $%
\mathcal{U}$ given by 
\begin{equation*}
I_{\delta }=\{p\in \mathcal{U}\colon \delta (p)=0\}.
\end{equation*}%
Then $\mathcal{U}/I_{\delta }$ is a normed $\mathbb{Q}(i)$-$\ast $-algebra.
Its completion is a C*-algebra, which we shall denote by $\widehat{\delta }$%
. (Notice that what we denote by $\widehat{\delta }$, is denoted by $%
B(\delta )$ in \cite[Subsection 2.3]{farah_turbulence_2014}.)

\begin{definition}
Let $\delta$ and $\delta ^{\prime }$ be elements in $\Xi $, and let $\Phi
=(\Phi_n) _{n\in \omega }\in \left( \mathcal{U}^{\mathcal{U}}\right)
^{\omega }$ be a sequence of functions from $\mathcal{U}$ to $\mathcal{U}$.
We say that $\Phi $ is a \emph{code for a *-homomorphism} from $\widehat{%
\delta }$ to $\widehat{\delta }^{\prime }$, if

\begin{enumerate}
\item for every $p\in \mathcal{U}$, the sequence $\left( \Phi _{n}(p)\right)
_{n\in \omega }$ is Cauchy uniformly in $p\in \mathcal{U}$, with respect to
the pseudometric $\left( q,q^{\prime }\right) \mapsto \delta \left(
q-q^{\prime }\right) $ on $\mathcal{U}$, and in particular converges in $%
\widehat{\delta }$ to an element $\widehat{\Phi }(p)$;

\item $p\mapsto \widehat{\Phi }(p)$ is a morphism of $\mathbb{Q}(i)$-$\ast $%
-algebras such that $\left\Vert \widehat{\Phi }(p)\right\Vert \leq \delta
(p) $, and hence induces a *-homomorphism from $\widehat{\delta }$ to $%
\widehat{\delta }^{\prime }$.
\end{enumerate}
\end{definition}

Writing down explicit formulas defining a code for a *-homomorphism makes it
clear that the set $\mathcal{C}_{\Xi }$ of triples $\left( \delta ,\delta
^{\prime },\Phi \right) \in \Xi \times \Xi \times \left( \mathcal{U}^{%
\mathcal{U}}\right) ^{\omega }$ such that $\Phi $ is a code for a
*-homomorphism from $\widehat{\delta }$ to $\widehat{\delta ^{\prime }}$ is
Borel. Suppose that $\Phi ,\Phi ^{\prime }$ are code for *-homomorphisms
from $\delta $ to $\delta ^{\prime }$ and from $\delta ^{\prime }$ to $%
\delta ^{\prime \prime }$. Similarly as in Subsection \ref{Subsection: XI
hat}, it is easy to check that defining%
\begin{equation*}
\left( \Phi ^{\prime }\circ \Phi \right) _{n}=\Phi _{n}^{\prime }\circ \Phi
_{n}
\end{equation*}%
for $n\in \omega $ gives a code for a *-homomorphism from $\delta $ to $%
\delta ^{\prime \prime }$. This defines a standard Borel semigroupoid
structure on $\mathcal{C}_{\Xi }$, such that the map $\left( \delta ,\delta
^{\prime },\Phi \right) \mapsto \left( \widehat{\delta },\widehat{\delta
^{\prime }},\widehat{\Phi }\right) $ is a good standard Borel
parametrization of the category of C*-algebras.

\subsection{The space $\mathcal{C}_{\Gamma (H)}$\label{Subsection: GAMMA}}

Denote by $B_{1}(H)$ the unit ball of $B(H)$ with respect to the operator
norm. Recall that $B_{1}(H)$ is a compact Hausdorff space when endowed with
the weak operator topology. The standard Borel structure generated by the
weak topology on $B_{1}(H)$ coincide with the Borel structure generated by
several other operator topologies on $B_{1}(H)$, such as the $\sigma $-weak,
strong, $\sigma $-strong, strong-*, and $\sigma $-strong-* operator
topology; see \cite[I.3.1.1]{blackadar_operator_2006}. Denote by $%
B_{1}(H)^{\omega }$ the product of countable many copies of $B_{1}(H)$,
endowed with the product topology, and define $\Gamma (H)$ to be the Polish
space obtained by removing from $B_{1}(H)^{\omega }$ the sequence constantly
equal to $0$. (The space $\Gamma (H)$ is defined similarly in \cite[%
Subection 2.1]{farah_turbulence_2014}; the only difference is that here the
sequence constantly equal to $0$ is excluded for convenience.) Given an
element $\gamma $ in $\Gamma (H)$, denote by $C^{\ast }(\gamma )$ the
C*-subalgebra of $B(H)$ generated by $\{\gamma _{n}\colon n\in \omega \}$.
As explained in \cite[Subsection 2.1 and Remark 2.3]{farah_turbulence_2014},
the space $\Gamma (H)$ can be thought of as a natural parametrization of 
\emph{concrete C*-algebras}.

\begin{definition}
Let $\gamma $ and $\gamma ^{\prime }$ be elements in $\Gamma (H)$, and let $%
\Phi =(\Phi _{n})_{n\in \omega }\in \left( \mathcal{U}^{\mathcal{U}}\right)
^{\omega }$ be a sequence of functions from $\mathcal{U}$ to $\mathcal{U}$.
We say that $\Phi $ is a \emph{code for a *-homomorphism }from $C^{\ast
}(\gamma )$ to $C^{\ast }\left( \gamma ^{\prime }\right) $, if

\begin{enumerate}
\item the sequence $\left( \Phi _{n}(p)(\gamma ^{\prime })\right) _{n\in
\omega }$ of elements of $C^{\ast }(\gamma ^{\prime })$ is Cauchy uniformly
in $p$, and hence converges to an element $\widehat{\Phi }\left( p(\gamma
)\right) $ of $C^{\ast }(\gamma ^{\prime })$;

\item the function $p(\gamma )\mapsto \widehat{\Phi }\left( p(\gamma
)\right) $ extends to a *-homomorphism from $C^{\ast }(\gamma )$ to $C^{\ast
}\left( \gamma ^{\prime }\right) $.
\end{enumerate}
\end{definition}

Again, it is easily checked that the set $\mathcal{C}_{\Gamma (H)}$ of
triples $\left( \gamma ,\gamma ^{\prime },\Phi \right) $ such that $\Phi $
is a code for a *-homomorphism from $C^{\ast }\left( \gamma \right) $ to $%
C^{\ast }\left( \gamma ^{\prime }\right) $, is Borel. Moreover, one can
define a standard Borel semigroupoid structure on $\mathcal{C}_{\Gamma (H)}$%
, in such a way that the map $\left( \gamma ,\gamma ^{\prime },\Phi \right)
\mapsto \left( C^{\ast }(\gamma ),C^{\ast }\left( \gamma ^{\prime }\right) ,%
\widehat{\Phi }\right) $ is a good parametrization of the category of
C*-algebras.

For future reference, we show in Lemma \ref{Lemma: Borel inverse} below that
in the parametrization $\mathcal{C}_{\Gamma (H)}$ one can compute a code for
the inverse of an isomorphism in a Borel way.

\begin{lemma}
\label{Lemma: Borel inverse} There is a Borel map from $\mathrm{Iso}_{\Gamma
(H)}$ to $\left( \mathcal{U}^{\mathcal{U}}\right) ^{\omega }$, assigning to
an element $\left( \gamma ,\gamma ^{\prime },\Phi \right) $ of $\mathrm{Iso}%
_{\Gamma (H)}$ a code $\mathrm{Inv}\left( \gamma ,\gamma ^{\prime },\Phi
\right) $ for an isomorphism from $C^{\ast }\left( \gamma ^{\prime }\right) $
to $C^{\ast }(\gamma )$ such that $\widehat{\mathrm{Inv}(\gamma ,\gamma
^{\prime },\Phi )}=\widehat{\Phi }^{-1}$.
\end{lemma}

\begin{proof}
Observe that the set $\mathcal{E}$ of tuples 
\begin{equation*}
\left( \left( \gamma ,\gamma ^{\prime },\Phi \right) ,p,n,q,N\right) \in 
\mathrm{Iso}_{\Gamma (H)}\times \mathcal{U}\times \omega \times \mathcal{U}%
\times \omega
\end{equation*}%
such that 
\begin{equation*}
\left\Vert q\left( \gamma ^{\prime }\right) -\Phi _{M}(p)\left( \gamma
^{\prime }\right) \right\Vert <\frac{1}{n},
\end{equation*}%
and 
\begin{equation*}
\left\Vert \Phi _{M^{\prime }}(p)\left( \gamma ^{\prime }\right) -\Phi
_{M}(p)\left( \gamma ^{\prime }\right) \right\Vert <\frac{1}{n},
\end{equation*}%
for every $M,M^{\prime }\geq N$ is Borel. Therefore one can find Borel
functions $\left( \xi ,q,n\right) \mapsto p_{\left( \xi ,p,n\right) }$ and $%
\left( \xi ,p,n\right) \mapsto N_{\left( \xi ,p,n\right) }$ from $\mathrm{Iso%
}_{\Gamma (H)}\times \mathcal{U}\times \omega $ to $\mathcal{U}$ and $\omega 
$ respectively such that%
\begin{equation*}
\left( \xi ,q,n,p_{\left( \xi ,q,n\right) },N_{\left( \xi ,q,n\right)
}\right) \in \mathcal{E}
\end{equation*}%
for every $\left( \xi ,q,n\right) \in \mathrm{Iso}_{\Gamma (H)}\times 
\mathcal{U}\times \omega $. Defining now $\mathrm{Inv}(\xi )_{n}(q)=p_{(\xi
,q,n)}$ for every $n\in \omega $ and $q\in \mathcal{U}$ one obtains a Borel
map $\xi \mapsto \mathrm{Inv}\left( \xi \right) $. Moreover,%
\begin{align*}
& \left\Vert \mathrm{Inv}(\xi )_{n}(q)(\gamma )-\widehat{\Phi }^{-1}(q\left(
\gamma ^{\prime }\right) )\right\Vert \\
& \ \ \ \ \ \ \ \ \ \ \ \ \leq \left\Vert p_{(\xi ,q,n)}\left( \gamma
\right) -\widehat{\Phi }^{-1}(\Phi _{N_{(\xi ,q,n)}}(p)\left( \gamma
^{\prime }\right) )\right\Vert +\frac{1}{n} \\
& \ \ \ \ \ \ \ \ \ \ \ \ =\left\Vert \widehat{\Phi }(p_{(\xi ,k,n)}\left(
\gamma \right) )-\Phi _{N_{(\xi ,q,n)}}(p)\left( \gamma ^{\prime }\right)
\right\Vert +\frac{1}{n} \\
& \ \ \ \ \ \ \ \ \ \ \ \leq \frac{1}{2n}\text{.}
\end{align*}%
This shows that $\mathrm{Inv}\left( \xi \right) $ is a code for the inverse
of $\widehat{\Phi }$.
\end{proof}

\subsection{Equivalence of $\mathcal{C}_{\widehat{\Xi }}$, $\mathcal{C}%
_{\Xi} $ and $\mathcal{C}_{\Gamma }$.}

Recall that given an element $\delta $ of $\Xi $, we denote by $I_{\delta }$
the ideal of $\mathcal{U}$ given by 
\begin{equation*}
I_{\delta }=\{p\in \mathcal{U}\colon \delta (p)=0\}.
\end{equation*}

\begin{theorem}
\label{theorem: equivalence of parametrizations of C-algebras} The good
parametrizations $\mathcal{C}_{\widehat{\Xi }}$, $\mathcal{C}_{\Xi }$, and $%
\mathcal{C}_{\Gamma }$, of the category of C*-algebras with *-homomorphisms,
are equivalent in the sense of Definition \ref{Definition: equivalent
parametrizations}.
\end{theorem}

\begin{proof}
We will show first that $\mathcal{C}_{\widehat{\Xi }}$ and $\mathcal{C}_{\Xi
}$ are equivalent.\newline
\indent We start by constructing a morphism from $\mathcal{C}_{\Xi }$ to $%
\mathcal{C}_{\widehat{\Xi }}$ as in Definition \ref{Definition: equivalent
parametrizations} as follows. As in the proof of \cite[Proposition 2.6]%
{farah_turbulence_2014}, for every $n\in \omega $ define a Borel map $%
p_{n}\colon \Xi \rightarrow \mathcal{U}$, denoted $\delta \mapsto
p_{n}^{\delta }$ for $\delta $ in $\Xi $, such that%
\begin{equation*}
\left\{ p_{n}^{\delta }+I_{\delta }\colon n\in \omega \right\}
\end{equation*}%
is an enumeration of $\mathcal{U}/I_{\delta }$ for every $\delta \in \Xi $ .
For $\delta \in \Xi $, define a structure of C*-normed $\mathbb{Q}(i)$-$\ast 
$-algebra $A_{\delta }=\left( f_{\delta },g_{\delta ,}h_{\delta },k_{\delta
},r_{\delta }\right) $ on $\omega $ by:

\begin{itemize}
\item $m+_{f_{\delta }}n=t$ whenever $p_{m}^{\delta }+p_{n}^{\delta
}+I_{\delta }=p_{t}^{\delta }+I_{\delta }$;

\item $q\cdot _{g_{\delta }}m=t$ whenever $q\cdot p_{m}^{\delta }+I_{\delta
}=p_{t}^{\delta }+I_{\delta }$;

\item $m\cdot _{h_{\delta }}n=t$ whenever $q_{m}^{\delta }q_{n}^{\delta
}+I_{\delta }=q_{t}^{\delta }+I_{\delta }$;

\item $m^{\ast k_{\delta }}=t$ whenever $\left( q_{m}^{\delta }\right)
^{\ast }+I_{\delta }=q_{t}^{\delta }+I_{\delta }$;

\item $\left\Vert m\right\Vert _{r_{\delta }}=\delta \left( q_{m}^{\delta
}\right) $.
\end{itemize}

It is clear that the map $\delta \mapsto A_{\delta }$ is Borel. Moreover,
for fixed $\delta $ in $\Xi $, the map $n\mapsto p_{n}^{\delta }+I_{\delta }$
is an isomorphism of normed $\mathbb{Q}(i)$-$\ast $-algebras from $A_{\delta
}$ onto $\mathcal{U}/I_{\delta }$. We denote by $\eta _{\delta }\colon 
\widehat{A}_{\delta }\rightarrow \widehat{\delta }$ the induced isomorphism
of C*-algebras.

Now, if $\xi =\left( \delta ,\delta ^{\prime },\Phi \right) $ belongs to $%
\mathcal{C}_{\Xi }$, define $\Psi _{\xi }\in \left( \mathcal{U}^{\mathcal{U}%
}\right) ^{\omega }$ by%
\begin{equation*}
(\Psi _{\xi })_{n}(m)=k\text{ whenever }\Phi _{n}\left( p_{m}^{\delta
}\right) +I_{\delta }=p_{k}^{\delta }+I_{\delta }\text{,}
\end{equation*}%
for $n,m$ and $k$ in $\omega $. It is not difficult to check that $\Psi
_{\xi }$ is a code for a *-homomorphism from $A_{\delta }$ to $A_{\delta
^{\prime }}$, and that the assignment $\xi \mapsto \Psi _{\xi }$ is Borel.
Thus, the map from $\mathcal{C}_{\Xi }$ to $\mathcal{C}_{\widehat{\Xi }}$
that assigns to the element $\xi =\left( \delta ,\delta ^{\prime },\Phi
\right) $ in $\mathcal{C}_{\Xi }$, the element $\left( A_{\delta },A_{\delta
^{\prime }},\Psi _{\xi }\right) $ of $\mathcal{C}_{\widehat{\Xi }}$, is
Borel. Finally, it is easily verified that the map%
\begin{equation*}
\xi =\left( \delta ,\delta ^{\prime },\Phi \right) \mapsto \left( \widehat{A}%
_{\delta },\widehat{A}_{\delta ^{\prime }},\widehat{\Psi }_{\xi }\right)
\end{equation*}%
is a functor from $\mathcal{C}_{\Xi }$ to the category of C*-algebras.
Moreover, if $\xi =\left( \delta ,\delta ^{\prime },\Phi \right) \in \Xi $,
then it follows from the construction that%
\begin{equation*}
\widehat{\Phi }\circ \eta _{\delta }=\eta _{\delta ^{\prime }}\circ \widehat{%
\Psi }_{\xi }\text{.}
\end{equation*}%
\ \newline
\indent We now proceed to construct morphism from $\mathcal{C}_{\widehat{\Xi 
}}$ to $\mathcal{C}_{\Xi }$. This will conclude the proof that $\mathcal{C}_{%
\widehat{\Xi }}$ and $\mathcal{C}_{\Xi }$ are equivalent parametrizations
according to Definition \ref{Definition: equivalent parametrizations}. 
\newline
\indent For $A\in \widehat{\Xi }$ and $p\in \mathcal{U}$, denote by $p_{A}$
the evaluation of $p$ in the $\mathbb{Q}(i)$-$\ast $-algebra on $\omega $
coded by $A$, where the formal variable $X_{j}$ is replaced by $j$ for every 
$j\in \omega $. Write $A=\left( f,g,h,k,r\right) $, and define an element $%
\delta _{A}$ of $\Xi $ by%
\begin{equation*}
\delta _{A}(p)=\left\Vert p_{A}\right\Vert _{r}\text{,}
\end{equation*}%
for all $p$ in $\mathcal{U}$. It is easily checked that the map $A\mapsto
\delta _{A}$ is a Borel function from $\widehat{\Xi }$ to $\Xi $. For every $%
n\in \omega $, define a Borel map $p_{n}\colon \widehat{\Xi }\rightarrow 
\mathcal{U}$, denoted $A\mapsto p_{n}^{A}$ for $A$ in $\widehat{\Xi }$, such
that%
\begin{equation*}
\left\{ p_{n}^{A}+I_{\delta _{A}}\colon n\in \omega \right\}
\end{equation*}%
is an enumeration of $\mathcal{U}/I_{\delta _{A}}$. The function $n\mapsto
p_{n}^{A}+I_{\delta _{A}}$ induces an isomorphism of normed $\mathbb{Q}(i)$-$%
\ast $-algebras, from $\omega $ with the structure coded by $A$, and $%
\mathcal{U}I_{\delta _{A}}$. One checks that this isomorphism induces a
C*-algebra isomorphism between $\widehat{A}$ and $\widehat{\delta }_{A}$.

For $\xi =\left( A,A^{\prime },\Psi \right) \in \mathcal{C}_{\widehat{\Xi }}$%
, define $\Psi _{\xi }\in \left( \mathcal{U}^{\mathcal{U}}\right) ^{\omega }$
by%
\begin{equation*}
\left( \Psi _{\xi }\right) _{n}(p)=q_{m}^{A}\text{ whenever }p+I_{\delta
_{A}}=p_{k}^{A}+I_{\delta _{A}}\text{ and }\Psi _{n}(k)=m\text{.}
\end{equation*}%
It can easily be checked that

\begin{itemize}
\item $\Psi _{\xi }$ is a code for a *-homomorphism from $\widehat{\delta }%
^{A}$ to $\widehat{\delta }_{A^{\prime }}$,

\item the map $\xi \mapsto \Psi _{\xi }$ is Borel, and

\item $\widehat{\Psi }_{\xi }\circ \eta_A=\eta _{A^{\prime }}\circ \widehat{%
\Psi }_\xi$.
\end{itemize}

This concludes the proof that $\mathcal{C}_{\Xi }$ and $\mathcal{C}_{%
\widehat{\Xi }}$ are equivalent good parametrizations of the category of
C*-algebras.\newline
\ \newline
\indent We proceed to show that $\mathcal{C}_{\Xi }$ and $\mathcal{C}%
_{\Gamma }$ are equivalent parametrizations.\newline
\indent Denote by $\delta \colon \Gamma (H)\rightarrow \Xi $ and $\gamma
\colon \Xi \rightarrow \Gamma (H)$ the Borel maps defined in the proof of 
\cite[Proposition 2.7]{farah_turbulence_2014} witnessing the fact that $\Xi $
and $\Gamma (H)$ are weakly equivalent parametrizations in the sense of \cite%
[Definition 2.1]{farah_turbulence_2014}. It is straightforward to check that
the maps $\Delta \colon \mathcal{C}_{\Gamma (H)}\rightarrow \mathcal{C}_{\Xi
}$ and $\Gamma \colon \mathcal{C}_{\Xi }\rightarrow \mathcal{C}_{\Gamma (H)}$
given by 
\begin{equation*}
\Delta (\gamma ,\gamma ^{\prime },\Phi )=(\delta _{\gamma },\delta _{\gamma
^{\prime }},\Phi )\ \ \ \mbox{ and }\ \ \ \Gamma (\delta ,\delta ^{\prime
},\Psi )=(\gamma _{\delta },\gamma _{\delta ^{\prime }},\Psi )
\end{equation*}%
are morphisms of good parametrizations, witnessing the facts that $\mathcal{C%
}_{\Gamma (H)}$ and $\mathcal{C}_{\Xi }$ are equivalent.
\end{proof}

\subsection{Direct limits of C*-algebras\label{Subsection: inductive limits
of C*-algebras}}

An \emph{inductive system} in the category of C*-algebras is a sequence $%
\left( A_{n},\varphi_n\right) _{n\in\omega }$, where for every $n$ in $%
\omega $, $A_n$ is a C*-algebra, and $\varphi_n\colon A_n\to A_{n+1}$ is a
*-homomorphism. The \emph{inductive limit} of the inductive system $\left(
A_{n},\varphi _{n}\right) _{n\in \omega }$ is the C*-algebra $\varinjlim
\left( A_{n},\varphi _{n}\right) $ defined as in \cite[II.8.2]%
{blackadar_operator_2006}. It is verified in \cite[Subsection 3.2]%
{farah_turbulence_2014} that the inductive limit of an inductive system of
C*-algebras can be computed in a Borel way. We report here, for the sake of
completeness, a different proof.

We will work in the parametrization $\mathcal{C}_{\Xi }$ of the category of
C*-algebras. In view of the equivalence of the parametrizations $\mathcal{C}%
_{\Xi }$, $\mathcal{C}_{\widehat{\Xi }}$, and $\mathcal{C}_{\Gamma (H)}$,
the same result holds if one instead considers either one of the
parametrizations $\mathcal{C}_{\widehat{\Xi }}$ or $\mathcal{C}_{\Gamma (H)}$%
.

Denote by $R_{dir}\left( \Xi \right) $ the set of sequences $\left(
\delta_{n},\Phi _{n}\right) _{n\in \omega }\in \left( \Xi \times \left( 
\mathcal{U}^{\mathcal{U}}\right) ^{\omega }\right) ^{\omega }$ such that $%
\Phi _{n}\mbox{ is a code for a *-homomorphism }\widehat{\delta}%
_{n}\rightarrow \widehat{\delta}_{n+1}$ for every $n\in \omega $. We can
regard $R_{dir}\left( \Xi \right) $ as the standard Borel space
parametrizing inductive systems of C*-algebras. (Though the subscript in $%
R_{dir}$ stands for ``direct system'', we choose the term ``inductive
system'' since we only deal with sequences. Since the notation $R_{dir}$ was
already introduced in \cite{farah_descriptive_2012}, with the same meaning,
we keep it despite calling its elements differently.)

\begin{proposition}
\label{prop: inductive limits} There is a Borel map from $R_{dir}\left( \Xi
\right) $ to $\Xi $ that assigns to an element $\left( \delta _{n},\Phi
_{n}\right) _{n\in \omega }$ of $R_{dir}\left( \Xi \right) $ an element $%
\lambda _{\left( \delta _{n},\Phi _{n}\right) _{n\in \omega }}$ of $\Xi $
such that $\widehat{\lambda }_{_{\left( \delta _{n},\Phi _{n}\right) _{n\in
\omega }}}$ is isomorphic to the inductive limit of the inductive system $%
\left( \widehat{\delta }_{n},\widehat{\Phi }_{n}\right) _{n\in \omega }$.
Moreover, for every $k\in \omega $ there is a Borel map from $R_{dir}\left(
\Xi \right) $ to $\left( \mathcal{U}^{\mathcal{U}}\right) ^{\omega }$ that
assigns to $\left( \delta _{n},\Phi _{n}\right) _{n\in \omega }$ a code $%
I_{k}$ for the canonical *-homomorphism from $\widehat{\delta }_{k}$ to the
inductive limit $\widehat{\lambda }_{_{\left( \delta _{n},\Phi _{n}\right)
_{n\in \omega }}}$.
\end{proposition}

\begin{proof}
Denote for $n\in \omega $ by $\mathcal{U}_{n}$ the $\mathbb{Q}(i)$-$\ast $%
-algebra of $\ast $-polynomials in the pairwise distinct noncommutative
variables $\left( X_{i}^{\left( n\right) }\right) _{i\in \omega }$.
Similarly define $\mathcal{U}_{\infty }$ to be the $\mathbb{Q}(i)$-$\ast $%
-algebra of $\ast $-polynomials in the noncommutative variables $\left(
X_{i}^{\left( n\right) }\right) _{\left( i,n\right) \in \omega \times \omega
}$. We will naturally identify $\mathcal{U}_{n}$ as a $\mathbb{Q}(i)$-$\ast $%
-subalgebra of $\mathcal{U}_{\infty }$, and define $\mathcal{V}_{n}$ to be
the $\mathbb{Q}(i)$-$\ast $-subalgebra of $\mathcal{U}_{\infty }$ generated
by%
\begin{equation*}
\bigcup_{i\in n}\mathcal{U}_{i}
\end{equation*}%
inside $\mathcal{U}_{\infty }$. Fix an element $\left( \delta _{n},\Phi
_{n}\right) _{k\in \omega }$ of $R_{dir}\left( \Xi \right) $. To simplify
the notation we will assume that $\delta _{n}\colon \mathcal{U}%
_{n}\rightarrow \mathbb{R}$ for every $n\in \omega $, and $\Phi _{n}\in
\left( \mathcal{U}_{n+1}^{\mathcal{U}_{n}}\right) ^{\omega }$.
Correspondingly we will define a function $\lambda _{\left( \delta _{n},\Phi
_{n}\right) _{n\in \omega }}\colon \mathcal{U}_{\infty }\rightarrow \mathbb{R%
}$. Fix $n\in n^{\prime }\in \omega $ and $k\in \omega $. Define%
\begin{equation*}
\Phi _{n^{\prime },n,k}\colon \mathcal{V}_{n}\rightarrow \mathcal{U}%
_{n^{\prime }}
\end{equation*}%
to be the function obtained by freely extending the maps%
\begin{equation*}
\left( \Phi _{n^{\prime }-1}\circ \cdots \circ \Phi _{i}\right) _{k}\colon 
\mathcal{U}_{i}\rightarrow \mathcal{U}_{n^{\prime }}
\end{equation*}%
for $i\in n$. Finally define for every $N\in \omega $ and $p\in \mathcal{V}%
_{N}\subset \mathcal{U}_{\infty }$%
\begin{equation*}
\lambda _{_{\left( \delta _{n},\Phi _{n}\right) _{n\in \omega
}}}(p)=\lim_{n^{\prime }>N}\lim_{k\rightarrow \infty }\delta _{n^{\prime
}}\left( \Phi _{n^{\prime },N,k}(p)\right) \text{.}
\end{equation*}%
It is immediate to verify that the definition does not depend on $N$.
Moreover $\lambda _{_{\left( \delta _{n},\Phi _{n}\right) _{n\in \omega
}}}\rightarrow \mathbb{R}$ define a seminorm on $\mathcal{U}_{\infty }$ such
that $\widehat{\lambda }_{\left( \delta _{n},\Phi _{n}\right) _{n\in \omega
}}$ is isomorphic to the direct limit of the inductive system $\left( 
\widehat{\delta }_{n},\widehat{\Phi }_{n}\right) _{n\in \omega }$. If $N\in
\omega $ and $\iota _{N}\colon \mathcal{U}_{N}\rightarrow \mathcal{U}%
_{\infty }$ denotes the inclusion map, and $I_{N}\in \left( \mathcal{U}%
_{\infty }^{\mathcal{U}_{N}}\right) ^{\omega }$ denotes the sequence
constantly equal to $\iota _{N}$, then $I_{N}$ is a code for the canonical
*-homomorphism from $\widehat{\delta }_{k}$ to the direct limit $\widehat{%
\lambda }_{_{\left( \delta _{n},\Phi _{n}\right) _{n\in \omega }}}$.
\end{proof}

\subsection{One sided intertwinings}

\label{Subsection: one-sided intertwinings}

\begin{definition}
Let $\left( A_{n},\varphi _{n}\right) _{n\in \omega }$ and $\left( A^{\prime
}_{n},\varphi^{\prime }_{n}\right) _{n\in \omega }$ be inductive systems of
C*-algebras. A sequence $(\psi_n)_{n\in\omega}$ of homomorphisms $%
\psi_n\colon A_n\to A_{n}^{\prime }$ is said to be a \emph{one sided
intertwining} between $\left( A_{n},\varphi _{n}\right) _{n\in \omega }$ and 
$\left( A^{\prime }_{n},\varphi^{\prime }_{n}\right) _{n\in \omega }$, if
the diagram 
\begin{align*}
\xymatrix{ A_0\ar[r]^{\varphi_0}\ar[d]_{\psi_0} &
A_1\ar[r]^{\varphi_1}\ar[d]_{\psi_1}
&A_2\ar[r]^{\varphi_2}\ar[d]_{\psi_2}&\cdots \\ A'_0\ar[r]_{\varphi'_0} &
A'_1\ar[r]_{\varphi'_1} &A'_2\ar[r]_{\varphi'_2}&\cdots }
\end{align*}
is commutative.
\end{definition}

If $(\psi_n)_{n\in\omega}$ is a one sided intertwining between $\left(
A_{n},\varphi _{n}\right) _{n\in \omega }$ and $\left( A^{\prime
}_{n},\varphi^{\prime }_{n}\right) _{n\in \omega }$, then there is an
inductive limit homomorphism 
\begin{equation*}
\psi=\varinjlim \psi_n\colon \varinjlim (A_n,\varphi_n)\to \varinjlim
(A^{\prime }_n,\varphi^{\prime }_n)
\end{equation*}
that makes the diagram 
\begin{align*}
\xymatrix{ A_0\ar[r]^{\varphi_0}\ar[d]_{\psi_0} &
A_1\ar[r]^{\varphi_1}\ar[d]_{\psi_1}
&A_2\ar[r]^{\varphi_2}\ar[d]_{\psi_2}&\cdots\ar[r] &
\varinjlim(A_n,\varphi_n)\ar[d]^{\psi} \\ A'_0\ar[r]_{\varphi'_0} &
A'_1\ar[r]_{\varphi'_1} &A'_2\ar[r]_{\varphi'_2}&\cdots \ar[r] &
\varinjlim(A'_n,\varphi'_n) }
\end{align*}
commutative.

In this subsection, we verify that the inductive limit homomorphism $%
\varinjlim \psi _{n}$ can be computed in a Borel way. We will work in the
parametrization $\mathcal{C}_{\Xi }$ of C*-algebras. In view of the
equivalence between the parametrizations $\mathcal{C}_{\Xi }$, $\mathcal{C}_{%
\widehat{\Xi }}$ and $\mathcal{C}_{\Gamma (H)}$, the same result holds if
one instead uses $\mathcal{C}_{\widehat{\Xi }}$ or $\mathcal{C}_{\Gamma (H)}$%
.

Define $R_{int}\left( \Xi \right) $ to be the Borel set of all elements 
\begin{equation*}
\left( \left( \delta _{n},\Phi _{n}\right) _{n\in \omega },\left( \delta
_{n}^{\prime },\Phi _{n}^{\prime }\right) _{n\in \omega },\left( \Psi
_{n}\right) _{n\in \omega }\right) \in R_{dir}\left( \Xi \right) \times
R_{dir}\left( \Xi \right) \times \left( \left( \mathcal{U}^{\mathcal{U}%
}\right) ^{\omega }\right) ^{\omega }
\end{equation*}%
such that $\Psi _{n}$ is a code for a *-homomorphism from $\widehat{\delta }%
_{n}$ to $\widehat{\delta }_{n+1}^{\prime }$ satisfying%
\begin{equation*}
\widehat{\Psi }_{n+1}\circ \widehat{\Phi }_{n}=\widehat{\Phi }_{n}^{\prime
}\circ \widehat{\Psi }_{n}
\end{equation*}%
for every $n\in \omega $. In other words, $(\Psi _{n})_{n\in\omega}$ is a
sequence of codes for a one sided intertwining between the inductive systems
coded by $\left( \delta _{n},\Phi _{n}\right) _{n\in \omega }$ and $\left(
\delta^{\prime }_{n},\Phi^{\prime }_{n}\right) _{n\in \omega }$.

\begin{proposition}
\label{prop: one sided intertwining} There is a Borel map from $%
R_{int}\left( \Xi \right) $ to $\left( \mathcal{U}^{\mathcal{U}}\right)
^{\omega }$ assigning to an element 
\begin{equation*}
\left( \left(\delta _{n},\Phi _{n}\right) _{n\in \omega },\left( \delta
_{n}^{\prime },\Phi _{n}^{\prime }\right) _{n\in \omega },\left( \Psi
_{n}\right) _{n\in \omega }\right)
\end{equation*}
of $R_{int}\left( \Xi \right) $, a code $\Lambda $ for the corresponding
inductive limit homomorphism between the inductive limits of the systems
coded by $\left( \delta _{n},\Phi _{n}\right) _{n\in \omega }$ and $\left(
\delta^{\prime }_{n},\Phi^{\prime }_{n}\right) _{n\in \omega }$.
\end{proposition}

\begin{proof}
We will use the same notation as in the proof of Proposition \ref{prop:
inductive limits}. Fix an element $\left( \left( \delta _{n},\Phi
_{n}\right) _{n\in \omega },\left( \delta _{n}^{\prime },\Phi _{n}^{\prime
}\right) _{n\in \omega },\left( \Psi _{n}\right) _{n\in \omega }\right) $ of 
$R_{int}\left( \Xi \right) $. As in the proof of Proposition \ref{prop:
inductive limits} we will assume that for every $n\in \omega $%
\begin{equation*}
\delta _{n}\colon \mathcal{U}_{n}\rightarrow \mathbb{R}\text{,}
\end{equation*}%
\begin{equation*}
\delta _{n}^{\prime }\colon \mathcal{U}_{n}\rightarrow \mathbb{R}\text{,}
\end{equation*}%
\begin{equation*}
\Phi _{n}\in \left( \mathcal{U}_{n+1}^{\mathcal{U}_{n}}\right) ^{\omega }
\end{equation*}%
and%
\begin{equation*}
\Phi _{n}^{\prime }\in \left( \mathcal{U}_{n+1}^{\mathcal{U}_{n}}\right)
^{\omega }\text{.}
\end{equation*}%
Therefore%
\begin{equation*}
\Psi _{n}\in \left( \mathcal{U}_{n}^{\mathcal{U}_{n}}\right) ^{\omega }
\end{equation*}%
for every $n\in \omega $. Similarly the codes $\lambda _{\left( \delta
_{n},\Phi _{n}\right) _{n\in \omega }}$ and $\lambda _{\left( \delta
_{n}^{\prime },\Phi _{n}^{\prime }\right) _{n\in \omega }}$ for the direct
limits of the systems coded by $\left( \delta _{n},\Phi _{n}\right) _{n\in
\omega }$ and $\left( \delta _{n}^{\prime },\Phi _{n}^{\prime }\right)
_{n\in \omega }$ will be supposed to be functions from $\mathcal{U}_{\infty
} $ to $\mathbb{R}$. We will therefore define a code $\Lambda \in \left( 
\mathcal{U}_{\infty }^{\mathcal{U}_{\infty }}\right) ^{\omega }$ for the
*-homomorphism coded by $\left( \Psi _{n}\right) _{n\in \omega }$. Recall
from the proof of Proposition \ref{prop: inductive limits} the definition of 
$\mathcal{V}_{N}$ and $\Phi _{n^{\prime },N,k}\colon \mathcal{V}%
_{N}\rightarrow \mathcal{U}_{n^{\prime }}$ for $N\in n^{\prime }\in \omega $
and $k\in \omega $. Fix functions $\sigma _{0},\sigma _{1},\sigma _{2}\colon
\omega \rightarrow \omega $ such that%
\begin{equation*}
n\mapsto \left( \sigma _{0}\left( n\right) ,\sigma _{1}\left( n\right)
,\sigma _{2}\left( n\right) \right)
\end{equation*}%
is a bijection from $\omega $ to $\omega \times \omega \times \omega $. Fix $%
N\in \omega $ and define for $p\in \mathcal{V}_{\sigma _{0}\left( N\right) }$
\begin{equation*}
\Lambda _{N}(p)=\left( \Psi _{\sigma _{1}\left( N\right) ,\sigma _{2}\left(
N\right) }\circ \Phi _{\sigma _{1}\left( N\right) ,\sigma _{0}\left(
N\right) ,\sigma _{2}\left( N\right) }\right) (p)
\end{equation*}%
and%
\begin{equation*}
\Lambda _{N}(p)=0
\end{equation*}%
for $p\notin \mathcal{V}_{\sigma _{0}\left( N\right) }$. It is not difficult
to check that the sequence $\left( \Lambda _{N}\right) _{N\in \omega }\in
\left( \mathcal{U}_{\infty }^{\mathcal{U}_{\infty }}\right) ^{\omega }$
indeed defines a code for the inductive limit homomorphism defined by the
sequence $\left( \widehat{\Psi }_{n}\right) _{n\in \omega }$.
\end{proof}

\subsection{Direct limits of groups}

We consider as \emph{standard Borel space} of infinite countable groups the
set $\mathcal{G}$ of functions $f\colon \omega \times \omega \rightarrow
\omega $ such that the identity $n\cdot _{f}m=f(n,m)$ for $n,m\in \omega $,
defines a group structure on $\omega $. We consider $\mathcal{G}$ as a Borel
space with respect to the Borel structure inherited from $\omega ^{\omega
\times \omega }$; such Borel structure is standard, since $\mathcal{G}$ is a
Borel subset of $\omega ^{\omega \times \omega }$. In the following, we will
identify a group $G$ and its code as an element of $\omega ^{\omega \times
\omega }$.

It is not difficult to check that most commonly studied classes of groups
correspond to Borel subsets of $\mathcal{G}$. In particular we will denote
by $\mathcal{AG}$ Borel set of abelian groups, and by $\mathcal{AG}_{TF}$
Borel set of torsion free abelian groups.

Let $G$ be a discrete group and let $\alpha $ be an endomorphism of $G$. We
will denote by $G_{\infty }=\varinjlim (G,\alpha )$ the inductive limit of
the inductive system 
\begin{equation*}
\xymatrix{G\ar[r]^\alpha & G\ar[r]^\alpha&\cdots\ar[r]& G_\infty}.
\end{equation*}%
For $n$ in $\omega $, denote by $\varphi _{n}\colon G\rightarrow G_{\infty }$
the canonical group homomorphism obtained by regarding $G$ as the $n$-th
stage of the inductive system above. Denote by $\alpha _{\infty }$ the
unique automorphism of $G_{\infty }$ such that $\alpha _{\infty }\circ
\varphi _{n+1}=\varphi _{n}$ for every $n\in \omega $.

Denote by $End_{\mathcal{G}}$ the set of all pairs $\left( G,\alpha \right)
\in \mathcal{G}\times \omega ^{\omega }$, such that $\alpha $ is an
injective endomorphism of $G$ with respect to the group structure on $\omega 
$ coded by $G$, and note that $End_{\mathcal{G}}$ is Borel. Similarly define 
$DLim_{\mathcal{G}}$ to be the set of pairs $\left( G,\alpha \right) \in
End_{\mathcal{G}}$ such that the direct limit $\varinjlim (G,\alpha )$ is
infinite.

\begin{proposition}
\label{Proposition: inductive limit of groups}The set $DLim_{\mathcal{G}}$
is a Borel subset of $End_{\mathcal{G}}$. Moreover there is a Borel map from 
$DLim_{\mathcal{G}}$ to $End_{\mathcal{G}}$ that assigns to $\left( G,\alpha
\right) \in DLim_{\mathcal{G}}$ the pair $\left( \varinjlim (G,\alpha
),\alpha _{\infty }\right) $.
\end{proposition}

\begin{proof}
Let $\left( G,\alpha \right) $ be an element in $End_{\mathcal{G}}$.
Consider the equivalence relation $\sim _{\alpha }$ on $\omega \times \omega 
$ defined by%
\begin{equation*}
\left( x,i\right) \sim _{\alpha }\left( y,j\right) \ \text{iff there exists }%
k\geq \max \left\{ i,j\right\} \text{ with }\alpha ^{k-i}(x)=\alpha ^{k-j}(y)%
\text{.}
\end{equation*}%
Observe that $\left( G,\alpha \right) \in DLim_{\mathcal{G}}$ iff $\sim
_{\alpha }$ has infinitely many classes. Therefore $DLim_{\mathcal{G}}$ is a
Borel subset of $\mathcal{G}$ by \cite[Theorem 18.10]{kechris_classical_1995}%
. Suppose now that $\left( G,\alpha \right) \in DLim_{\mathcal{G}}$.
Consider the lexicographic order $<_{lex}$ on $\omega \times \omega $, and
define the injective function $\eta _{\alpha }\colon \omega \rightarrow
\omega \times \omega $ recursively on $n$ as follows. Set $\eta _{\alpha
}(0)=(0,0)$, and for $n>0$, define $\eta _{\alpha }(n)$ to be the $<_{lex}$%
-minimum element $\left( m,i\right) $ of $\omega \times \omega $ such that
for every $k\in n$, we have $\eta _{\alpha }(k)\not\sim _{\alpha }\left(
m,i\right) $. (Observe that the set of such elements is nonempty since we
are assuming that $\sim _{\alpha }$ has infinitely many classes.) Define the
group operation on $\omega $ by 
\begin{equation*}
n_{0}\cdot _{G_{\infty }}n_{1}=n
\end{equation*}%
whenever there are $m_{0},m_{1},m,i_{0},i_{1},i,\widetilde{i}\in \omega $
satisfying:

\begin{itemize}
\item $\eta _{\alpha }( n_{0}) =( m_{0},i_{0}) $;

\item $\eta _{\alpha }( n_{1}) =( m_{1},i_{1}) $;

\item $\eta (n) =( m,i) $;

\item $\max \left\{ i_{0},i_{1}\right\} =\widetilde{i}$;

\item $\left(\alpha ^{\widetilde{i}-i_{0}}( m_{0}) \cdot _{G}\alpha ^{%
\widetilde{i}-i_{1}}( m_{1}) ,\widetilde{i} \right) \sim ( m,i) $.
\end{itemize}

Define the function $\alpha_\infty \colon \omega \rightarrow \omega $ by $%
\alpha_\infty (n) =n^{\prime }$ if and only if there are $m,i,m^{\prime
},i^{\prime }\in \omega $ such that:

\begin{itemize}
\item $\eta _{\alpha }(n) =( m,i) $;

\item $\eta _{\alpha }( n^{\prime }) =( m^{\prime },i^{\prime }) $;

\item $( \alpha (m) ,i ) \sim ( m^{\prime },i^{\prime }) $.
\end{itemize}

It is not difficult to check that $G_{\infty }$ is the direct limit $%
\varinjlim (G,\alpha )$, and $\alpha _{\infty }$ is the automorphism of $%
\varinjlim (G,\alpha )$ corresponding to the endomorphism $\alpha $ of $G$.
Moreover the function $\left( G,\alpha \right) \mapsto \left( G_{\infty
},\alpha _{\infty }\right) $ is Borel by construction.
\end{proof}

\subsection{Borel version of the Nielsen-Schreier theorem}

The celebrated nielsen-Schreier theorem \cite{nielsen_om_1921,
schreier_Untergruppen_1927} asserts that a subgroup of a countable discrete
free group is free. In this subsection we will prove a \emph{Borel version }%
of such theorem, to be used in the proof of Lemma \ref{Lemma: exact sequence}%
. This will be obtained by analyzing Schreier's proof of the theorem, as
presented in \cite[Chapter 2]{johnson_presentations_1997}.

Denote by $F$ the (countable) set of \emph{reduced }words in the
indeterminates $x_{n}$ for $n\in \omega $ ordered lexicographically. We can
identify the free group on countable many generators with $F$ with the
operation of \emph{reduced concatenation }of words. It is immediate to check
that the set $\mathcal{S}(F)$ of $H\in 2^{\omega }$ such that $H$ is a
subgroup of $F$ is Borel.

\begin{lemma}
\label{Lemma: NS}There is a Borel function $H\mapsto B_{H}$ from $\mathcal{S}%
(F)$ to $2^{F}$ such that $L_{H}$ is a set of free generators for $H$ for
every $H\in \mathcal{S}(F)$.
\end{lemma}

\begin{proof}
Suppose that $H\in \mathcal{S}(F)$. If $a\in F$ denote by $\phi _{H}\left(
a\right) $ the $<$-minimal element of the coset $Ha$, where $<$ is the
lexicographic order of $F$. Observe that $\phi _{H}\left( a\right) \leq b$
iff there is $b^{\prime }\leq b$ such that $b^{\prime }a^{-1}\in H$. This
shows that the map 
\begin{eqnarray*}
\mathcal{S}(F) &\rightarrow &F^{F} \\
H &\mapsto &\phi _{H}
\end{eqnarray*}%
is Borel. Define $B_{H}$ to be the set containing%
\begin{equation*}
\phi _{H}\left( a\right) x_{n}\phi _{H}\left( \phi _{H}\left( a\right)
x_{n}\right) ^{-1}
\end{equation*}%
for $n\in \omega $ and $a\in F$ such that $\phi _{H}\left( a\right)
x_{n}\neq \phi _{H}\left( c\right) $ for every $c\in F$. It is clear that
the map $H\mapsto B_{H}$ is Borel. Moreover it can be shown as in \cite[%
Chapter 2, Lemmas 3,4,5]{johnson_presentations_1997} that $B_{H}$ is a free
set of generators of $H$.
\end{proof}

Suppose now that $\mathbb{F}_{\omega }$ is an element of $\mathcal{G}$
representing the group of countably many generators, and $\mathcal{S}(F)$ is
the Borel set of $H\in 2^{\omega }$ such that $H$ is a subgroup of $\mathbb{F%
}_{\infty }$. Proposition can be seen as just a reformulation of Lemma

\begin{proposition}
\label{Proposition: NS}There is a Borel map $H\mapsto B_{H}$ from $\mathcal{S%
}(F)$ to $2^{\omega }$ that assigns to $H\in \mathcal{S}(F)$ a free set of
generators of $H$.
\end{proposition}

\subsection{An exact sequence}

The following lemma asserts that the construction of \cite[Proposition 3.5]%
{rordam_classification_1995} can be made in a Borel way.

\begin{lemma}
\label{Lemma: exact sequence} There is a Borel function from $\mathcal{AG}$
to $\mathcal{AG}_{TF}\times \omega ^{\omega }$ that assigns to an abelian
group $G$, a pair $(H,\alpha )$, where $H$ is a torsion free abelian group,
and $\alpha $ is an automorphism of $H$ such that 
\begin{equation*}
H\left/ (id_{H}-\alpha )\left[ H\right] \right. \cong G.
\end{equation*}
\end{lemma}

\begin{proof}
Denote by $\mathbb{F}_{\omega \times \omega }$ the free group with
generators $x_{n,m}$ for $\left( n,m\right) \in \omega \times \omega $,
suitably coded as an element of the standard Borel spaces of discrete groups 
$\mathcal{AG}$. Given an element $G\in \mathcal{AG}$, denote by $N_{G}$ the
subset of $\omega $ coding the kernel of the homomorphism from $\mathbb{F}%
_{\omega \times \omega }$ to $G$ obtained by sending $x_{n,m}$ to $n$ if $%
m=0 $, and to zero otherwise. In view of Proposition \ref{Proposition: NS}
one can find a Borel map%
\begin{eqnarray*}
\mathcal{AG} &\rightarrow &\omega ^{\omega } \\
G &\mapsto &x^{G}
\end{eqnarray*}%
such that $x^{G}=\left( x_{n}^{G}\right) _{n\in \omega }$ is an enumeration
of a free set of generators of $N_{G}$. Define an injective endomorphism $%
\delta _{G}$ of $\mathbb{F}_{\omega \times \omega }$ by%
\begin{equation*}
\delta _{G}\left( x_{n,m}\right) =%
\begin{cases}
x_{n,m+1} & \text{ if }m\neq -1\text{,} \\ 
x_{n}^{G} & \text{ otherwise.}%
\end{cases}%
\end{equation*}

Let $\beta _{G}\colon \mathbb{F}_{\omega \times \omega }\rightarrow \mathbb{F%
}_{\omega \times \omega }$ be $\beta _{G}=\mbox{id}_{\mathbb{F}_{\omega
\times \omega }}-\delta _{G}$. By construction, the map $G\mapsto \beta _{G}$
is Borel. From now on we fix a group $G$, and abbreviate $\beta _{G}$ to
just $\beta $. \newline
\indent By Proposition \ref{Proposition: inductive limit of groups}, the
inductive limit group $G_{\infty }=\varinjlim (G,\beta )$ and the
automorphism $\beta _{\infty }=\varinjlim \beta $ can be constructed in a
Borel way from $G$ and $\beta $. We take $H=G_{\infty }$ and $\alpha =\beta
_{\infty }$. It can now be verified, as in the proof of \cite[Proposition 3.5%
]{rordam_classification_1995}, that $G$ is isomorphic to the quotient of $H$
by the image of $\mbox{id}_{H}-\alpha $. Moreover, it follows that the map $%
G\mapsto (H)$ is Borel. This finishes the proof.
\end{proof}

\section{Computing reduced crossed products\label{Chapter: computing reduced
crossed products}}

The goal of this section is to show that the reduced crossed product of a
C*-algebra by an action of a discrete group can be computed in a Borel way.
We begin by recalling the construction of the reduced crossed product.

\subsection{The reduced crossed product\label{Subsection: the reduced
crossed product}}

Suppose that $A$ is a C*-algebra, and $\alpha $ is an action of a countable
discrete group $G$ on $A$. We recall here the construction of the reduced
crossed product $A\rtimes _{\alpha ,r}G$ of $A$ by $\alpha $. Denote by $A%
\left[ G\right] $ the \emph{skew group algebra}. This is the complex $\ast $%
-algebra%
\begin{equation*}
A[G]=\left\{ \sum_{g\in G}a_{g}g\colon a_{g}\in A,a_{g}=0%
\mbox{ for all but
finitely many $g\in G$}\right\} \text{.}
\end{equation*}%
The product on $A\left[ G\right] $ is defined by twisted convolution%
\begin{equation*}
\left( \sum_{g\in G}a_{g}g\right) \left( \sum_{h\in G}b_{h}h\right)
=\sum_{g,h\in G}a_{g}\alpha _{g}\left( b_{h}\right) gh\text{.}
\end{equation*}%
The involution in $A\left[ G\right] $ is given by%
\begin{equation*}
\left( \sum_{g\in G}b_{g}g\right) ^{\ast }=\sum_{g\in G}\alpha _{g}\left(
b_{g^{-1}}^{\ast }\right) g\text{.}
\end{equation*}%
Suppose that $H$ is a Hilbert space. A \emph{covariant representation} of $%
\alpha $ on $H$ is a pair $\left( \pi ,v\right) $ where

\begin{enumerate}
\item $\pi $ is a representation of $A$ on $H$, and

\item $v$ is a unitary representation of $G$ on $H$ such that%
\begin{equation*}
v(g)\pi (a)v(g)^{\ast }=\pi \left( \alpha _{g}(a)\right)
\end{equation*}%
for every $g\in G$ and $a\in A$.
\end{enumerate}

The \emph{integrated form} of the covariant representation $\left( \pi
,v\right) $ is the *--homomorphism $\pi \rtimes v$ from the twisted group
algebra $A\left[ G\right] $ to $B(H)$ defined by%
\begin{equation*}
\left( \pi \rtimes u\right) \left( \sum_{g\in G}a_{g}g\right) =\sum_{g\in
G}\pi \left( a_{g}\right) u_{g}\text{.}
\end{equation*}%
Let now $\rho $ be a representation of $A$ on $H_{0}$. The \emph{regular
covariant representation} of $\alpha $ associated with $\rho $, is the
covariant representation $(\pi _{\rho ,\alpha },v_{\rho ,\alpha })$ of $%
\alpha $ on 
\begin{equation*}
H=\ell ^{2}\left( G,H_{0}\right) \text{.}
\end{equation*}%
defined as follows: For $a\in A,g,h\in G$ and $\xi \in H$, set 
\begin{equation*}
(\pi _{\rho ,\alpha }(a)\xi )(g)=\rho (\alpha _{g^{-1}}(a))\xi (g).
\end{equation*}%
and%
\begin{equation*}
(v_{\rho ,\alpha })_{g}(\xi )(h)=\xi (g^{-1}h)\text{.}
\end{equation*}%
Observe that, if $\lambda \colon G\rightarrow \ell ^{2}(G)$ denotes the left
regular representation of $G$, then the unitary operator $\left( v_{\rho
,a}\right) _{g}$ on $\ell ^{2}\left( G,H\right) $ can be identified with 
\begin{equation*}
\lambda _{g}\otimes \mbox{id}_{H_{0}}
\end{equation*}%
under the natural identification of $\ell ^{2}\left( G,H_{0}\right) $ with $%
\ell ^{2}\left( G\right) \otimes H_{0}$. The integrated form of $(\pi _{\rho
,\alpha },v_{\rho ,\alpha })$ will be called the \emph{regular representation%
} of $A[G]$ associated with $\rho $.

\begin{definition}
Suppose that $\alpha $ is an action of a discrete group $G$ on $A$. For $%
a\in A[G]$, set 
\begin{equation*}
\Vert a\Vert _{r}=\sup \left\{ \Vert (\pi \rtimes v)(a)\Vert \colon (\pi ,v)%
\mbox{ is a regular
covariant representation}\right\} ,\text{\label{Equation: reduced crossed
norm}}
\end{equation*}%
and define the \emph{reduced crossed product} of $A$ by $\alpha $, denoted $%
A\rtimes _{\alpha ,r}G$, to be the completion of $A[G]$ with respect to the
norm $\Vert \cdot \Vert _{r}$.
\end{definition}

It is shown in \cite[Theorem 7.7.5]{pedersen_c-algebras_1979} that if $\rho $
is a faithful representation of $A$ on $H_{0}$, then for any action $\alpha $
of a discrete group on $A$, the integrated form of the regular covariant
representation of $(\pi _{\rho ,\alpha },v_{\rho ,\alpha })$ on $\ell
^{2}(G,H_{0})$ associated with $\rho $ induces a faithful representation of
the reduced crossed product $A\rtimes _{\alpha ,r}G$ on $\ell ^{2}\left(
G,H_{0}\right) $. Equivalently, 
\begin{equation*}
\left\Vert x\right\Vert _{r}=\left\Vert \left( \pi _{\rho ,\alpha }\rtimes
v_{\rho ,\alpha }\right) (x)\right\Vert
\end{equation*}%
for every $x\in A\left[ G\right] $.

One can also consider the completion of $A\left[ G\right] $ with respect to
the C*-norm obtained as in \ref{Equation: reduced crossed norm} but
considering the supremum over \emph{all }covariant representations of $%
\alpha $. One thus obtains the so called \emph{full crossed product }$%
A\rtimes _{\alpha }G$.

Both full and reduced crossed products are C*-algebras encoding information
about the action $\alpha $. For many purposes reduced crossed products are
far better behaved and easier to understand than full ones. It is a standard
fact in the theory of crossed products that if $G$ is amenable, then full
and reduced crossed products agree. See \cite{williams_crossed_2007} for
more details. In the following we will consider exclusively reduced crossed
products\newline

Similar notions can be defined for actions of locally compact group on
C*-algebras. More details can be found in \cite[Section II.10]%
{blackadar_operator_2006}.

\subsection{Parametrizing actions of discrete groups on C*-algebras}

We proceed to construct a standard Borel parametrization of the space of all
actions of discrete groups on C*-algebras. For convenience, we will work
using the parametrization $\Gamma (H)$ of C*-algebras. In view of the weak
equivalence of $\Xi $, $\widehat{\Xi }$, and $\Gamma (H)$, similar
statements will hold for the parametrizations $\Xi $ and $\widehat{\Xi }$.

\begin{definition}
Let $\gamma $ be an element of $\Gamma (H)$, and $G$ be an element of $%
\mathcal{G}$. Suppose that $\Phi =\left( \Phi _{m,n}\right) _{\left(
m,n\right) \in \omega \times \omega }\in \left( \mathcal{U}^{\mathcal{U}%
}\right) ^{\omega \times \omega }$ is an $\left( \omega \times \omega
\right) $-sequence of functions from $\mathcal{U}$ to $\mathcal{U}$. We say
that $\Phi $ is a \emph{code for an action of }$G$ \emph{on} $C^{\ast
}\left( \gamma \right) $, if the following conditions hold:

\begin{enumerate}
\item for every $m\in \omega ,$ the sequence $\left( \Phi _{m,n}\right)
_{n\in \omega }\in \left( \mathcal{U}^{\mathcal{U}}\right) ^{\omega }$ is a
code for an automorphism $\widehat{\Phi }_{m}$ of $C^{\ast }\left( \gamma
\right) $,

\item $\Phi _{0,n}(m)=m$ for every $n,m\in \omega $, and

\item the function $m\mapsto \widehat{\Phi }_{m}$ is an action of $G$ on $%
C^{\ast }(\gamma )$, this is,%
\begin{equation*}
\widehat{\Phi }_{m}\circ \widehat{\Phi }_{k}=\widehat{\Phi }_{n}
\end{equation*}%
whenever $\left( m,k,n\right) \in G$.
\end{enumerate}
\end{definition}

It is easy to verify that any action of $G$ on $C^{\ast }\left( \gamma
\right) $ can be coded in a similar fashion. Moreover, the set $\text{Act}%
_{\Gamma (H)}$ of triples $\left( G,\gamma ,\Phi \right) \in \mathcal{G}%
\times \Gamma (H)\times \left( \mathcal{U}^{\omega }\right) ^{\omega \times
\omega }$ such that $\Phi $ is a code for an action of $G$ on $C^{\ast
}(\gamma )$, is a Borel subset of $\mathcal{G}\times \Gamma (H)\times \left( 
\mathcal{U}^{\omega }\right) ^{\omega \times \omega }$. We will regard $%
\text{Act}_{\Gamma (H)} $ as the parametrizing standard Borel space of
actions of discrete groups on C*-algebras.

\subsection{Computing the reduced crossed product}

We are now ready to prove that the reduced crossed product of a C*-algebra
by an action of a countable discrete group can be computed in a Borel way.

\begin{proposition}
\label{Proposition: computing crossed product action} Let $H$ be a separable
Hilbert space. Then there is a Borel map $\left( G,\gamma ,\Phi \right)
\mapsto \delta _{(G,\gamma ,\Phi )}$ from $\text{Act}_{\Gamma (H)}$ to $%
\Gamma (H)$ such that $C^{\ast }\left( \delta _{(G,\gamma ,\Phi )}\right)
\cong C^{\ast }(\gamma )\rtimes _{\widehat{\Phi }}^{r}G$. In other words,
there is a Borel way to compute the code of the reduced crossed product of
separable C*-algebras by countable discrete groups.
\end{proposition}

\begin{proof}
Denote by $\left\{ e_{k}\colon k\in \omega \right\} $ the canonical basis of 
$\ell _{2}$. Let $\left( G,\gamma ,\Phi \right) $ be an element of $\text{Act%
}_{\Gamma (H)}$. Define the element $\delta _{(G,A,\Phi )}$ of $\Gamma (H)$
as follows. Given $m$ in $\omega $, denote by $m^{\prime }\in \omega $ the
inverse of $m$ in $G$. Now set 
\begin{equation*}
(\delta _{(G,\gamma ,\Phi )})_{n}(\xi \otimes m)=%
\begin{cases}
\lim_{k\rightarrow +\infty }\gamma _{\Phi _{(m^{\prime },k)}(r)} & \text{ if 
}n=2r,\mbox{ where }\left( n,m,k\right) \in G\text{,} \\ 
\xi \otimes k & \text{ otherwise,}%
\end{cases}%
\end{equation*}%
for all $\xi $ in $H$ and all $m$ in $\omega $. The fact that $C^{\ast
}\left( \delta _{(G,\gamma ,\Phi )}\right) \cong C^{\ast }\left( \gamma
\right) \rtimes _{\widehat{\Phi },r}G$ follows from \cite[Theorem 7.7.5]%
{pedersen_c-algebras_1979}. Moreover, the map $\left( G,\gamma ,\Phi \right)
\mapsto \delta _{(G,\gamma ,\Phi )}$ is Borel by construction and by \cite[%
Lemma 3.4]{farah_turbulence_2014}.
\end{proof}

Proposition \ref{Proposition: computing crossed product action} above
answers half of \cite[Problem 9.5(ii)]{farah_turbulence_2014}. It is not
clear how to treat the case of full crossed products, even in the special
case when the algebra is $\mathbb{C}$.

\subsection{Crossed products by a single automorphism\label{Subsection:
crossed product by automorphism}}

Any automorphism $\alpha $ of a C*-algebra $A$ naturally induces an action
of the group of integers $\mathbb{Z}$ on $A$ by $n\mapsto \alpha^n
=\alpha\circ\cdots\circ\alpha$.

In this subsection, we want to show that the crossed product of a C*-algebra
by a single automorphism, when regarded as an action of $\mathbb{Z}$, can be
computed in a Borel way. In view of the equivalence of the good
parametrizations $\mathcal{C}_{\Xi }$, $\mathcal{C}_{\widehat{\Xi }}$, and $%
\mathcal{C}_{\Gamma (H)}$, we can work in any of these. For convenience, we
consider the parametrization $\mathcal{C}_{\Gamma (H)}$.

Let us denote by $\mathrm{\mathrm{Aut}}_{\Gamma (H)}$ the set of pairs $%
\left( \gamma ,\Phi \right) $ in $\Gamma (H)\times \left( \mathcal{U}^{%
\mathcal{U}}\right) ^{\omega }$ such that $\Phi $ is a code for an
automorphism of $C^{\ast }\left( \gamma \right) .$ It is immediate to check
that such set is Borel. We can regard $\mathrm{\mathrm{Aut}}_{\Gamma (H)}$
as the standard Borel space of automorphisms of C*-algebras.

\begin{lemma}
\label{Lemma: coding automorphisms} There is a Borel map from $\mathrm{%
\mathrm{Aut}}_{\Gamma (H)}$ to $\text{Act}_{\Gamma (H)}$ that assigns to an
element $\left( \gamma ,\Phi \right) $ in $\mathrm{\mathrm{Aut}}_{\Gamma
(H)} $, a code for the action of $\mathbb{Z}$ on $C^{\ast }(\gamma )$
associated with the automorphism coded by $\Phi $.
\end{lemma}

\begin{proof}
In the parametrization $\mathcal{G}$ of discrete groups described before,
the group of integers $\mathbb{Z}$ is coded, for example, by the element $f_{%
\mathbb{Z}}$ of $\omega ^{\omega \times \omega }$ given by%
\begin{align*}
f_{\mathbb{Z}}(2n,2m)& =2(n+m) \\
f_{\mathbb{Z}}(2n-1,2m-1)& =2(n+m)-1 \\
f_{\mathbb{Z}}(2n-1,m)& =f_{\mathbb{Z}}(m,2n-1)=2(n-m)-1 \\
f_{\mathbb{Z}}(2m-1,n)& =f_{\mathbb{Z}}(n,2m-1)=2(n-m) \\
f_{\mathbb{Z}}(k,0)& =f_{\mathbb{Z}}(0,k)=k
\end{align*}%
for $n,m,k\in \omega $ with $n,m\geq 1$. Recall that by Lemma \ref{Lemma:
Borel inverse} there is a Borel map $\xi \mapsto \mathrm{Inv}\left( \xi
\right) $ from $\mathrm{Iso}_{\Gamma (H)}$ to $\left( \mathcal{U}^{\mathcal{U%
}}\right) ^{\omega }$ such that if $\xi =\left( \gamma ,\gamma ^{\prime
},\Phi \right) $, then $\mathrm{Inv}\left( \xi \right) $ is a code for the
inverse of the *-isomorphism coded by $\Phi $. Suppose now that $\left(
\gamma ,\Phi \right) \in \mathrm{\mathrm{Aut}}_{\Gamma (H)}$. We want to
define a code $\Psi $ for the action of $\mathbb{Z}$ on $C^{\ast }(\gamma )$
induced by $\widehat{\Phi }$. For $n,m\in \omega $ with $m\geq 1$ define%
\begin{equation*}
\Psi _{0,n}(k)=k\text{,}
\end{equation*}%
\begin{equation*}
\Psi _{2m,n}=\overset{m\text{ times}}{\overbrace{\Phi _{n}\circ \Phi
_{n}\cdots \circ \Phi _{n}}}\text{,}
\end{equation*}%
and%
\begin{equation*}
\Psi _{2m+1,n}=\mathrm{Inv}\left( \gamma ,\gamma ,\Psi _{m}\right) \text{.}
\end{equation*}%
Observe that $\left( f_{\mathbb{Z}},A,\Psi \right) $ is a code for the
action of $\mathbb{Z}$ associated with the automorphism $\widehat{\Phi }$ of 
$C^{\ast }(\gamma )$. It is not difficult to verify that the map assigning $%
\left( f_{\mathbb{Z}},A,\Psi \right) $ to $\left( A,\Phi \right) $ is Borel.
We omit the details.
\end{proof}

\begin{corollary}
\label{cor: crossed product automorphism} Given a C*-algebra $A$ and an
automorphism $\alpha$ of $A$, there is a Borel way to compute the crossed
product $A\rtimes_\alpha \mathbb{Z}$.
\end{corollary}

\begin{proof}
Note that the group of integers $\mathbb{Z}$ is amenable, so full and
reduced crossed products coincide. The result now follows immediately from
Lemma \ref{Lemma: coding automorphisms} together with Proposition \ref%
{Proposition: computing crossed product action}.
\end{proof}

\subsection{Crossed product by an endomorphism\label{Subsection: crossed
product by enromorphism}}

We now turn to crossed products by injective, corner endomorphisms, as
introduced by Paschke in \cite{paschke_crossed_1980}, building on previous
work of Cuntz in \cite{cuntz_simple_1977}. Although there are more general
theories for such crossed products allowing arbitrary endomorphisms of
C*-algebras (see, for example, \cite{Exel_new_2003}), the endomorphisms
considered by Paschke will suffice for our purposes. We begin by presenting
the precise definition of a corner endomorphism. Throughout this subsection,
we fix a unital C*-algebra $A$.

\begin{definition}
An endomorphism $\rho \colon A\rightarrow A$ is said to be a \emph{corner}
endomorphism if $\rho (A)$ is a corner of $A$, this is, if there exists a
projection $p$ in $A$ such that $\rho (A)=pAp$.
\end{definition}

Since $A$ is unital, if $\rho \colon A\rightarrow A$ is a corner
endomorphism and $\rho (A)=pAp$ for some projection $p$ in $A$, then we must
have $p=\rho (1)$. Let us observe for future reference that the set $%
CorEnd_{\Gamma }$ of pairs $\left( \gamma ,\Phi \right) \in \Gamma \times
\left( \mathcal{U}^{\mathcal{U}}\right) ^{\omega }$ such that $C^{\ast
}(\gamma )$ is unital and $\Phi $ is a code for an injective corner
endomorphism of $C^{\ast }(\gamma )$ is Borel. By \cite[Lemma 3.14]%
{farah_turbulence_2014} the set $\Gamma _{u}$ of $\gamma \in \Gamma $ such
that $C^{\ast }(\gamma )$ is unital is Borel. Moreover, there is a Borel map 
$un\colon \Gamma _{u}\rightarrow B_{1}(H)$ such that $un(\gamma )$ is the
unit of $C^{\ast }(\gamma )$ for every $\gamma \in C^{\ast }(\gamma )2$. If
now $\gamma \in \Gamma _{u}$ and $\Phi \in \left( \mathcal{U}^{\mathcal{U}%
}\right) ^{\omega }$, then $\Phi $ is a code for an injective corner
endomorphism of $C^{\ast }(\gamma )$ if and only if $\Phi $ is a code for an
endomorphism of $A$ (which is a Borel condition, as observed in\ Subsection %
\ref{Subsection: GAMMA}), and for every $p\in \mathcal{U}$ and $n\in \omega $
there is $m_{0}\in \omega $ and $q\in \mathcal{U}$ such that for every $%
m\geq m_{0}$%
\begin{equation*}
\left\Vert \Phi _{m}(p)(\gamma )\right\Vert \geq \left\Vert p(\gamma
)\right\Vert -\frac{1}{n}
\end{equation*}%
and%
\begin{equation*}
\left\Vert un(\gamma )p(\gamma )un\left( \gamma \right) -\Phi _{m}\left(
q\right) (\gamma )\right\Vert \leq \frac{1}{n}\text{.}
\end{equation*}

Let $\rho $ be an injective corner endomorphism of $A$. The crossed product $%
A\rtimes _{\rho }\mathbb{N}$ of $A$ by $\rho $ is implicitly defined in \cite%
{paschke_crossed_1980} as the universal C*-algebra generated by a unital
copy of $A$ together with an isometry $S$, subject to the relation 
\begin{equation*}
SaS^{\ast }=\rho (a)
\end{equation*}%
for all $a$ in $A$. Suppose that $s$ is an isometry of $A$. Notice that the
endomorphisms $a\mapsto sas^{\ast }$ is injective and its range is the
corner $(ss^{\ast })A(ss^{\ast })$ of $A$.

Instead of using this construction, which involves universal C*-algebras on
generators and relations, we will use the construction of the endomorphism
crossed product described by Stacey in \cite{stacey_crossed_1993}. Stacey's
picture has the advantage that, given what we have proved so far, it will be
relatively easy to conclude that crossed products by injective corner
endomorphisms can be computed in a Borel way.

We proceed to describe Stacey's construction.

\begin{definition}
Let $\rho\colon A\to A$ be an injective corner endomorphism. Consider the
inductive system $(A,\alpha)_{n\in\omega}$ (the same algebra and same
connecting maps throughout the sequence), and denote by $A_\infty$ its
inductive limit, and by $\iota_{n,\infty}\colon A\to A_\infty$ the canonical
map into the inductive limit. There is a commutative diagram 
\begin{align*}
\xymatrix{ A\ar[rr]^\alpha \ar[d]_\alpha && A\ar[rr]^\alpha \ar[d]_\alpha
&&\cdots \ar[rr] && A_\infty\ar@{-->}[d]^{\alpha_\infty}\\ A\ar[rr]_\alpha
&& A\ar[rr]_\alpha &&\cdots \ar[rr] && A_\infty.}
\end{align*}
It is not hard to check that the corresponding endomorphism $\alpha_\infty
\colon A_\infty\to A_\infty$ of the inductive limit is an automorphism.
Denote by $e$ the projection of $A_\infty$ corresponding to the unit of $A$.
The \emph{(endomorphism) crossed product} of $A$ by $\rho$ is the corner $%
e(A_\infty \rtimes_{\alpha_\infty}\mathbb{Z})e$ of the (automorphism)
crossed product $A_\infty \rtimes_{\alpha_\infty}\mathbb{Z}$.
\end{definition}

As mentioned before, this construction of the crossed product of a
C*-algebra by an endomorphism makes it apparent that it can be computed in a
Borel way. In fact, we have verified in Proposition \ref{prop: one sided
intertwining}, that the limit of a one sided intertwining can be computed in
a Borel way, and in Corollary \ref{cor: crossed product automorphism}, that
the crossed product of a C*-algebra by an automorphism can be computed in a
Borel way. Moreover, it is shown in \cite[Lemma 3.14]{farah_turbulence_2014}%
, that one can select in a Borel way the unit of a unital C*-algebra. The
only missing ingredient in the construction is taking a corner by a
projection, which is shown to be Borel in the following lemma. We will work,
for convenience, in the parametrization $\Gamma (H)$ of C*-algebras.

\begin{lemma}
\label{lemma: taking corners is Borel} The set $\Gamma _{proj}(H)$ of pairs $%
(\gamma ,e)$ in $\Gamma (H)\times B(H)$ such that $e$ is a nonzero
projection in $C^{\ast }(\gamma )$, is Borel. Moreover, there is a Borel map 
$(\gamma ,e)\mapsto c_{\gamma ,e}$ from $\Gamma _{proj}(H)$ to $\Gamma (H)$
such that $C^{\ast }(c_{\gamma ,e})$ is the corner $eC^{\ast }(\gamma )e$ of 
$C^{\ast }(\gamma )$.
\end{lemma}

\begin{proof}
Enumerate a dense subset $\left\{ \xi _{n}\colon n\in \omega \right\} $ of $%
H $, and let $(\gamma ,e)$ be an element in $\Gamma (H)\times B(H)$. Observe
that $(\gamma ,e)$ belongs to $\Gamma _{proj}(H)$ if and only if the
following conditions hold:

\begin{enumerate}
\item The element $e$ is a projection, this is, for every $n,k,t\in \omega $%
, 
\begin{equation*}
\left\langle \left( e-e^{\ast }\right) \xi _{k},\xi _{t}\right\rangle <\frac{%
1}{n+1}\ \ \ \mbox{ and }\ \ \ \left\langle \left( e^{2}-e\right) \xi
_{k},\xi _{t}\right\rangle <\frac{1}{n+1}\text{;}
\end{equation*}

\item The element $e$ is non-zero, this is, there are $k,n,t\in \omega $
such that%
\begin{equation*}
\left\langle e\xi _{k},\xi _{t}\right\rangle >\frac{1}{n+1};
\end{equation*}

\item The element $e$ is in $C^{\ast }(\gamma )$, this is, for every $n\in
\omega $ there is $p\in \mathcal{U}$ such that 
\begin{equation*}
\left\langle \left( p(\gamma )-e\right) \xi _{m},\xi _{t}\right\rangle <%
\frac{1}{n+1},
\end{equation*}%
for every $m,t\in \omega $.
\end{enumerate}

This shows that $\Gamma _{proj}(H)$ is a Borel subset of $\Gamma (H)\times
B(H)$. Observe now that by setting 
\begin{equation*}
\left( c_{\gamma ,e}\right) _{n}=e\gamma _{n}e
\end{equation*}%
for every $n\in \omega $, one obtains an element $c_{\gamma ,e}$ of $\Gamma
(H)$ such that $C^{\ast }\left( c_{\gamma ,e}\right) =eC^{\ast }(\gamma )e$.
It is immediate to check that the map $\left( \gamma ,e\right) \mapsto
c_{\gamma ,e}$ is Borel.
\end{proof}

We have thus proved the following.

\begin{corollary}
\label{Corollary: Borel crossed product endomorphism} Given a unital
C*-algebra $A$ and an injective corner endomorphism $\rho $ of $A$, there is
a Borel way to compute a code for the crossed product of $A$ by $\rho $.%
\newline
\indent More precisely, there is a Borel map 
\begin{eqnarray*}
CorEnd_{\Gamma } &\rightarrow &\Gamma \\
\left( \gamma ,\Phi \right) &\mapsto &\delta _{\gamma ,\Phi }
\end{eqnarray*}%
such that $C^{\ast }\left( \delta _{\gamma ,\Phi }\right) \cong C^{\ast
}(\gamma )\rtimes _{\widehat{\Phi }}\mathbb{N}$, where, as before, $%
CorEnd_{\Gamma }$ is the Borel space of pairs $\left( \gamma ,\Phi \right) $
in $\Gamma \times \left( \mathcal{U}^{\mathcal{U}}\right) ^{\omega }$ such
that $C^{\ast }(\gamma )$ is unital and $\Phi $ is a code for an injective
corner endomorphism of $C^{\ast }(\gamma )$ .
\end{corollary}

\begin{proof}
Combine Lemma \ref{lemma: taking corners is Borel} with Proposition and
Corollary \ref{cor: crossed product automorphism}.
\end{proof}


\section{Borel selection of AF-algebras\label{Chapter: Borel selection of
AF-algebras}}

\subsection{Bratteli diagrams}

We refer the reader to \cite{ellis_classification_2010} for the standard
definition of a Bratteli diagram. We will identify Bratteli diagrams with
elements 
\begin{equation*}
\left( l,\left( w_{n}\right) _{n\in \omega },\left( m_{n}\right) _{n\in
\omega }\right) \in \omega ^{\omega }\times \left( \omega ^{\omega }\right)
^{\omega }\times \left( \omega ^{\omega \times \omega }\right) ^{\omega }
\end{equation*}%
such that for every $i,j,n,m\in \omega $, the following conditions hold:

\begin{enumerate}
\item $l\left( 0\right) =1$;

\item $w_0(0) =1$;

\item $w_n(i) >0$ if and only if $i\in l(n) $;

\item $m_n(i,j) =0$ whenever $i\geq l(n) $ or $j\geq l(n+1) $

\item With $k_n=i$, 
\begin{equation*}
w_n(i)=\sum_{k_{j}\in l(j),1\leq j<n} \prod_{t=0}^{n-1}m_t(k_t,k_{t+1}).
\end{equation*}
\end{enumerate}

We denote by $\mathcal{BD}$ the Borel set of all elements $\left(
l,w,m\right) $ in $\omega ^{\omega }\times \left( \omega ^{\omega }\right)
^{\omega }\times \left( \omega ^{\omega \times \omega }\right) ^{\omega }$
that satisfy conditions 1--5 above. An element $\left( l,w,m\right) $ of $%
\mathcal{BD}$ codes the Bratteli diagram with $l(n)$ vertices at the $n$-th
level of weight $w_{n}\left( 0\right) ,\ldots ,w\left( l(n)-1\right) $ and
with $m_{n}(i,j)$ arrows from the $i$-th vertex at the $n$-th level to the $%
j $-th vertex and the $(n+1)$-st level for $n\in \omega $, $i\in l(n)$, and $%
j\in l(n+1)$. We call the elements of $\mathcal{BD}$ simply
\textquotedblleft Bratteli diagrams\textquotedblright .

\subsection{Dimension groups}

\begin{definition}
An \emph{ordered abelian group} is a pair $(G,G^+)$, where $G$ is an abelian
group and $G^+$ is a subset of $G$ satisfying

\begin{enumerate}
\item $G^++G^+\subseteq G^+$;

\item $0\in G^+$;

\item $G^+\cap \left( -G^+\right) =\left\{ 0\right\} $;

\item $G^+-G^+=G$.
\end{enumerate}

We call $G^{+}$ the \emph{positive cone} of $G$. It defines an order on $G$
by declaring that $x\leq y$ whenever $y-x\in G^{+}$. An element $u$ of $%
G^{+} $ is said to be an \emph{order unit} for $(G,G^{+})$, if for every $x$
in $G$, there exists a positive integer $n$ such that 
\begin{equation*}
-nu\leq x\leq nu.
\end{equation*}%
\indent An ordered abelian group $\left( G,G^{+}\right) $ is said to be 
\emph{unperforated} if whenever a positive integer $n$ and $a\in A$ satisfy $%
na\geq 0$, then $a\geq 0$. Equivalently, $G^{+}$ is divisible. \newline
\indent An ordered abelian group is said to have the \emph{Riesz
interpolation property }if for every $x_{0},x_{1},y_{0},y_{1}\in G$ such
that $x_{i}\leq y_{j}$ for $i,j\in 2$, there is $z\in G$ such that%
\begin{equation*}
x_{i}\leq z\leq y_{j}
\end{equation*}%
for $i,j\in 2$.
\end{definition}

\begin{definition}
A \emph{dimension group} is an unperforated ordered abelian group $\left(
G,G^+,u\right) $ with the Riesz interpolation property and a distinguished
order unit $u$. \newline
\indent Let $(G,G^+,u)$ and $(H,H^+,w)$ be dimension groups, and let $%
\phi\colon G\to H$ be a group homomorphism.

\begin{enumerate}
\item We say that $\phi$ is \emph{positive} if $\phi(G^+)\subseteq H^+$, and

\item We say that $\phi $ \emph{preserves the unit} if $\phi (u)=w$.
\end{enumerate}
\end{definition}

Notice that positivity for a homomorphism between ordered groups is
equivalent to preservation of the order.

\begin{example}
If $l\in \omega $ and $w_{0},\ldots ,w_{l-1}\in \mathbb{N}$, then $\mathbb{Z}%
^{l}$ with $\mathbb{N}^{l}$ as its positive cone, and $\left( w_{0},\ldots
,w_{l-1}\right) $ as order unit, is a dimension group. We denote by $%
e_{0}^{(l)},\ldots ,e_{l-1}^{(l)}$ the canonical basis of $\mathbb{Z}^{l}$.
\end{example}

We refer the reader to \cite[Section 1.4]{rordam_classification_2002} for a
more complete exposition on dimension groups.

A dimension group can be coded in a natural way as an element of $\omega
^{\omega \times \omega }\times 2^{\omega }\times \omega $. The set $\mathcal{%
DG}$ of codes for dimension groups is a Borel subset of $\omega ^{\omega
\times \omega }\times 2^{\omega }\times \omega $, which can be regarded as
the standard Borel space of dimension groups.

One can associate to a Bratteli diagram $\left( l,w,m\right) $ the dimension
group $G_{(l,w,m)}$ obtained as follows. For $n$ in $\omega$, denote by 
\begin{equation*}
\varphi_n\colon \mathbb{Z}^{l(n)}\to\mathbb{Z}^{l(n+1)}
\end{equation*}
the homomorphism given on the canonical bases of $\mathbb{Z}^{l(n)}$ by 
\begin{equation*}
\varphi _{n}\left( e^{(l(n)}_{k}\right) =\sum_{i\in l(n+1) }m_{n}(i,j)
e^{(l(n+1)}_{j},
\end{equation*}
for all $k$ in $l(n)$. Then $G_{(l,w,m)}$ is defined as the inductive limit
of the inductive system%
\begin{equation*}
\left( \mathbb{Z}^{l(n) },\left( w_{n}\left( 0\right) ,\ldots ,w_{n}\left(
l(n) -1\right) \right) ,\varphi _{n}\right) _{n\in \omega }.
\end{equation*}

Theorem 2.2 in \cite{effros_dimension_1980} asserts that any dimension group
is in fact isomorphic to one of the form $G_{(l,w,m)}$ for some Bratteli
diagram $\left( l,w,m\right) $. The key ingredient in the proof of \cite[%
Theorem 2.2]{effros_dimension_1980} is a Lemma due to Shen, see \cite[Lemma
2.1]{effros_dimension_1980} and also \cite[Theorem 3.1]%
{shen_classification_1979}. We reproduce here the statement of the Lemma,
for convenience of the reader.

\begin{lemma}
\label{Lemma: Shen's condition}Suppose that $\left( G,G^+,u\right) $ is a
dimension group. If $n\in \omega $ and $\theta \colon n\rightarrow G$ is any
function, then there are $N\in \omega $, and functions $\Phi \colon
N\rightarrow G^+$ and $g\colon n\times N\rightarrow \omega $, satisfying the
following conditions:

\begin{enumerate}
\item For all $i\in n$, 
\begin{equation*}
\theta (i) =\sum_{j\in N}g(i,j) \Phi (j).
\end{equation*}

\item Whenever $\left( k_{i}\right) _{i\in n}\in \mathbb{Z}^{n}$ is such
that $\sum_{i\in n}k_{i}\theta (i)=0$, then 
\begin{equation*}
\sum_{i\in l^{G}(n)}k_{i}g(i,j)=0
\end{equation*}%
for every $j\in N$.
\end{enumerate}
\end{lemma}

It is immediate to note that the set of tuples $\left( \left(
G,G^{+},u\right) ,n,\theta ,N,\Phi ,g\right) $ satisfying 1 and 2 of Lemma %
\ref{Lemma: Shen's condition} is Borel. It follows that in Lemma \ref{Lemma:
Shen's condition} the number $N$ and the maps $\Phi $ and $g$ can be
computed from $\left( G,G^{+},u\right) $, $n$, and $\theta $ is a Borel way.
This will be used to show that if we start with a dimension group $%
(G,G^{+},u)$, the choice of the Bratteli diagram $(l^{G},w^{G},m^{G})$
satisfying $G_{l^{G},w^{G},m^{G}}\cong G$ as dimension groups with order
units, which is guaranteed to be possible by \cite[Theorem 2.2]%
{effros_dimension_1980}, can be made in a Borel way. See Proposition \ref%
{Proposition Borel version of EHS} below. A Borel version of \cite[Theorem
2.2]{effros_dimension_1980} is also proved in \cite[Theorem 5.3]%
{ellis_classification_2010}. We present here a proof, for the convenience of
the reader, and to introduce ideas and notations to be used in the proof of
Proposition \ref{Proposition: Borel endomorphism dimension group}.

\begin{proposition}
\label{Proposition Borel version of EHS}There is a Borel function that
associates to a dimension group $G=\left( G,G^+,u\right) \in \mathcal{DG}$ a
Bratteli diagram $\left( l^{G},w^{G},m^{G}\right) \in \mathcal{BD}$ such
that the dimension group associated with $\left( l^{G},w^{G},m^{G}\right) $
is isomorphic to $G$.
\end{proposition}

\begin{proof}
It is enough to construct in a Borel way a Bratteli diagram $\left(
l^{G},w^{G},m^{G}\right) $ and maps $\theta _{n}^{G}\colon
l^{G}(n)\rightarrow G$ satisfying the following conditions:

\begin{enumerate}
\item \label{Enumerate item: homomorphism} For all $i\in l^{G}(n)$, 
\begin{equation*}
\theta _{n}^{G}(i)=\sum_{j\in l^{G}(n+1)}m_{n}^{G}(i,j)\theta _{n}^{G}(j);
\end{equation*}

\item \label{Enumerate item: 1:1}For any $k_{0},\ldots ,k_{l^{G}(n)-1}\in 
\mathbb{Z}$ such that%
\begin{equation*}
\sum_{i\in l^{G}(n)}k_{i}\theta _{n}^{G}(i)=0,
\end{equation*}%
we have that 
\begin{equation*}
\sum_{i\in l^{G}(n)}k_{i}m_{n}^{G}(i,j)=0,
\end{equation*}%
for every $j\in l^{G}(n+1)$;

\item \label{Enumerate item: onto}For every $x\in G^+$, there are $n\in
\omega $ and $i\in l^{G}(n) $, such that $\theta _{n}^{G}(i) =x$.
\end{enumerate}

It is not difficult to verify that conditions (\ref{Enumerate item:
homomorphism}), (\ref{Enumerate item: 1:1}) and (\ref{Enumerate item: onto})
ensure that the dimension group coded by the Bratteli diagram $\left(
l^{G},w^{G},m^{G}\right) $ is isomorphic to $G$, via the isomorphism coded
by $\left( \theta _{n}^{G}\right) _{n\in \omega }$.

We define $\theta _{n}^{G}$, $l^{G}(n) $, $w_{n}^{G}$ and $m_{n}^{g}$ by
recursion on $n$. Define $l^{G}(0) =1$ and $\theta ^{G}(0) =u$. Suppose that 
$l^{G}(k) $, $w_{k}^{G}$, $m_{k}^{G}$, and $\theta _{k}^{G}$ have been
defined for $k\leq n$. Define $\theta ^{\prime }\colon l^{G}(n)
+1=\{0,\ldots,l^{G}(n)\} \rightarrow G$ by 
\begin{equation*}
\theta ^{\prime }( i) = 
\begin{cases}
\theta(i), & \text{if } i\in l^G(n) \text{,} \\ 
n, & \text{if } i=l^G(n)\mbox{ and } n\in G^+\text{,} \\ 
u, & \text{otherwise.}%
\end{cases}%
\end{equation*}%
Suppose that the positive integer $N$ and the functions $\Phi \colon
N\rightarrow G^+$ and $g\colon n\times N\rightarrow \omega $ are obtained
from $l^{G}(n) $ and $\theta ^{\prime }$ via Lemma \ref{Lemma: Shen's
condition}. Define now: 
\begin{align*}
l^{G}(n+1) &=N \\
m_{n}(i,j) &= 
\begin{cases}
g(i,j), & \text{if } i\in l^{G}(n) \mbox{ and } j\in l^{G}(n+1), \\ 
0, & \text{otherwise.}%
\end{cases}
\\
w_{n+1}(j) &=\sum_{i\in l^{G}(n) }w_{n}(i,n) m_{n}(i,j).
\end{align*}
It is left as an exercise to check that with these choices, conditions (1),
(2) and (3) are satisfied. This finishes the proof.
\end{proof}

\subsection{Approximately finite dimensional C*-algebras\label{Subsection :
AF algebras}}

A unital C*-algebra $A$ is said to be \emph{approximately finite dimensional}%
, or \emph{AF-algebra} if it is isomorphic to a direct limit of a direct
system of finite dimensional C*-algebras with unital connecting maps.

It is a standard result in the theory of C*-algebras, that any finite
dimensional C*-algebra is isomorphic to a direct sum of matrix algebras over
the complex numbers \cite[Theorem III.l.l]{davidson_c*-algebras_1996}. A
fundamental result due to Bratteli (building on previous work of Glimm)
asserts that unital AF-algebras are precisely the unital C*-algebras that
can be locally approximated by finite dimensional C*-algebras.

\begin{theorem}[Bratteli-Glimm \protect\cite{bratteli_inductive_1972,
glimm_certain_1960}]
\label{Theorem: GB} Let $A$ be a separable C*-algebra. Then the following
are equivalent:

\begin{enumerate}
\item $A$ is a unital AF-algebra;

\item For every finite subset $F$ of $A$ and every $\varepsilon >0$, there
exists a finite dimensional C*-subalgebra $B$ of $A$, such that for every $%
a\in F$ there is $b\in B$ such that $\left\Vert a-b\right\Vert <\varepsilon $%
.
\end{enumerate}
\end{theorem}

A modern presentation of the proof of Theorem \ref{Theorem: GB} can be found
in \cite[Proposition 1.2.2]{rordam_classification_2002}.

A distinguished class of unital AF-algebras are the so called unital \emph{%
UHF-algebras}. These are the unital AF-algebras that are isomorphic to a
direct limit of \emph{full matrix algebras}. Of particular importance are
the UHF-algebras \emph{of infinite type}. These can be described as follows:
Fix a strictly positive integer $n$. Denote by $\mathbb{M}_{n^{\infty }}$
the C*-algebra obtained as a limit of the inductive system 
\begin{equation*}
\mathbb{M}_{n}\rightarrow \mathbb{M}_{n^{2}}\rightarrow \mathbb{M}%
_{n^{3}}\rightarrow \cdots
\end{equation*}%
where the inclusion of $\mathbb{M}_{n^{k}}$ into $\mathbb{M}_{n^{k+1}}$ is
given by the diagonal embedding $a\mapsto \mbox{diag}(a,\ldots ,a)$. The
UHF-algebras of infinite type are precisely those ones of the form $\mathbb{M%
}_{n^{\infty }}$ for some $n\in \mathbb{N}$.

A celebrated theorem of Elliott asserts that unital AF-algebras are
classified up to isomorphism by their ordered $K_{0}$-group. The $K_{0}$%
-group is an ordered abelian group with a distinguished order unit that can
be associated to any unital C*-algebra, and bears information about the
projections of the given C*-algebra. The definition of the $K_{0}$-group and
its basic properties can be found in \cite[Chapter 3]%
{rordam_introduction_2000}. We will denote by $K_{0}(A) $ the $K_{0}$-group
of the unital C*-algebra $A$ with positive cone $K_{0}( A) ^{+}$ and
distinguished order unit $\left[ 1_{A}\right] $ corresponding to the class
of the unit of $A$. Dimension groups can be characterized within the class
of ordered abelian groups with a distinguished order unit as the $K_{0}$%
-groups of AF-algebras.

\begin{theorem}[Elliott \protect\cite{elliott_classification_1976}]
Let $A$ and $B$ be unital AF-algebras.

\begin{enumerate}
\item For every positive morphism 
\begin{equation*}
\phi \colon (K_{0}(A),K_{0}(A)^{+})\rightarrow (K_{0}(B),K_{0}(B)^{+}),
\end{equation*}%
such that $\phi ([1_{A}])\leq \lbrack 1_{B}]$, there exists a homomorphism $%
\rho \colon A\rightarrow B$ such that $K_{0}(\rho )=\phi $.

\item $A$ and $B$ are isomorphic if and only if $%
(K_{0}(A),K_{0}(A)^{+},[1_{A}])$ and $(K_{0}(B),K_{0}(B)^{+},[1_{B}])$ are
isomorphic as dimension groups with order units.
\end{enumerate}

In (1), the range of $\rho $ is a corner of $B$ if and only if the set 
\begin{equation*}
\{x=[p]-[q]\in K_{0}(B)^{+}\colon p,q\in B\mbox{ and }x\leq \phi ([1_{A}])\}
\end{equation*}%
is contained in $\phi (K_{0}(A)^{+}).$
\end{theorem}

Let $\left( l,w,m\right) $ be a Bratteli diagram. We will describe how to
canonically associate to it a unital AF-algebra, which we will denote by $%
A_{(l,w,m)}$. For each $n$ in $\omega $, define a finite dimensional
C*-algebra $F_{n}$ by 
\begin{equation*}
F_{n}=\bigoplus_{i\in l(n)}\mathbb{M}_{w_{n}(i)}.
\end{equation*}%
Denote by $\varphi _{n}\colon F_{n}\rightarrow F_{n+1}$ the unital injective
*-homomorphism determined as follows. For every $i\in l(n)$ and $j\in l(n+1)$%
, the restriction of $\varphi _{n}$ to the $i$-th direct summand of $F_{n}$
and the $j$-th direct summand of $F_{n+1}$ is a diagonal embedding of $%
m_{n}(i,j)$ copies of $\mathbb{M}_{w_{n}(i)}$ in $\mathbb{M}_{w_{n+1}(j)}$.
Then $A_{(l,w,m)}$ is the inductive limit of the inductive system $\left(
F_{n},\varphi _{n}\right) _{n\in \omega }$.

The $K_{0}$-group of $A_{(l,w,m)}$ is isomorphic to the dimension group $%
G_{l,w,m}$ associated with $\left( l,w,m\right) $.

The main result of \cite{bratteli_inductive_1972} asserts that any unital
AF-algebra is isomorphic to the C*-algebra associated with a Bratteli
diagram. We show below that the code for such an AF-algebra can be computed
in a Borel way.

\begin{proposition}
\label{Proposition: Bratteli and AF} Given a Bratteli diagram, there is a
Borel way to compute the code for its associated unital AF-algebra.
\end{proposition}

\begin{proof}
By Proposition \ref{prop: inductive limits}, the inductive limit of an
inductive system of C*-algebras can be computed in a Borel way. It is
therefore enough to show that there is a Borel map that assigns to each
Bratteli diagram, a code for the corresponding inductive system of
C*-algebras. We will work, for convenience, with the parametrization $\Gamma
(H)$ of C*-algebras.

Let $\left\{ \xi _{n}^{k}\colon \left( k,n\right) \in \omega \times \omega
\right\} $ be an orthonormal basis of $H$. For $n,m,k\in \omega $, denote by 
$e_{n,m}^{(k)}$ the rank $1$ operator in $B(H)$ sending $\xi _{n}^{k}$ to $%
\xi _{m}^{k}$. For convenience, we will also identify $\Gamma (H)$ with the
space of nonzero functions from $\omega \times \omega \times \omega $ to the
unit ball of $B(H)$. For $n\in \omega $, define $\gamma ^{(n)}\in \Gamma (H)$
by%
\begin{equation*}
\gamma _{i,j,k}^{(n)}=%
\begin{cases}
e_{i,j}^{(k)}, & \text{if }k\in l(n)\text{ and }i,j\in w_{n}(k)\text{,} \\ 
0, & \text{otherwise.}%
\end{cases}%
\end{equation*}%
Denote by $A_{i,j,k}^{n}$ the set of triples in $\omega \times \omega \times
\omega $ of the form 
\begin{equation*}
\left( \sum_{k^{\prime }\in k}w_{n}(k^{\prime })m_{n}(k^{\prime
},t,n)+dw_{n}(k)+i,\sum_{k^{\prime }\in k}w_{n}(k^{\prime })m_{n}(k^{\prime
},t)+dw_{n}(n),t\right)
\end{equation*}%
such that $d\in m_{n}(k,t)$, $i,j\in w_{n}(k)$ and $t\in l(n+1)$. It is
clear that $C^{\ast }\left( \gamma ^{(n)}\right) $ is a finite dimensional
C*-algebra isomorphic to 
\begin{equation*}
\bigoplus_{i\in l(n)}\mathbb{M}_{w\left( i,n\right) }.
\end{equation*}%
For $n$ in $\omega $, let $\Phi ^{(n)}\colon \mathcal{U}\rightarrow \mathcal{%
U}$ be the unique morphism of $\mathbb{Q}(i)$-algebras satisfying 
\begin{equation*}
\Phi ^{(n)}(X_{ijk})=\sum_{\left( a,b,t\right) \in A_{ijk}^{n}}X_{a,b,t},
\end{equation*}%
for $i,j,k$ in $\omega $. Then $\Phi ^{(n)}$ is a code for the unital
injective *-homomorphism from $C^{\ast }\left( \gamma ^{(n)}\right) $ to $%
C^{\ast }\left( \gamma ^{(n+1)}\right) $ given by the diagonal embedding of $%
m_{n}(k,t,n)$ copies of $\mathbb{M}_{w_{n}(k)}$ in $\mathbb{M}_{w_{n}(t)}$
for every $k\in l(n)$ and $t\in l(n+1)$. By construction, the map $\mathcal{%
BD}\rightarrow R_{dir}\left( \Gamma (H)\right) $ that assigns to every
Bratteli diagram $\left( l,w,m\right) $, the code $\left( \gamma ^{(n)},\Phi
^{(n)}\right) _{n\in \omega }$, is Borel. This finishes the proof.
\end{proof}

Proposition \ref{Proposition Borel version of EHS} together with Proposition %
\ref{Proposition: Bratteli and AF} imply the following corollary.

\begin{corollary}
\label{Corollary: dimension groups and AF}There is a Borel map that assigns
to a dimension group $D$ a unital AF-algebra $A_{D}$ whose ordered $K_{0}$%
-group is isomorphic to $D$ as dimension group.
\end{corollary}

Since by \cite[Proposition 3.4]{farah_descriptive_2012} the $K_{0}$-group of
a C*-algebra can be computed in a Borel way, one can conclude that if $%
\mathcal{A}$ is any Borel set of dimension groups, then the relation of
isomorphisms restricted to $\mathcal{A}$ is Borel bireducible with the
relation of isomorphism unital AF-algebras whose $K_{0}$-group is isomorphic
to an element of $\mathcal{A}$.

Fix $n\in \mathbb{N}$. A dimension group has \emph{rank }$n$ if $n$ is the
largest size of a linearly independent subset. Let us denote by $\cong
_{n}^{+}$ the relation of isomorphisms of dimension groups of rank $n$, and
by $\cong _{n}^{AF}$ the relation of isomorphisms of AF-algebras whose
dimension group has rank. By the previous discussion, the relations $\cong
_{n}^{+}$ and $\cong _{n}^{AF}$ are Borel bireducible. Moreover \cite[%
Theorem 1.11]{ellis_classification_2010} asserts that 
\begin{equation*}
\cong _{n}^{+}<_{B}\cong _{n+1}^{+}
\end{equation*}%
for every $n\in \mathbb{N}$. This means that $\cong _{n}^{+}$ is Borel
reducible to $\cong _{n+1}^{+}$, but $\cong _{n+1}^{+}$ is \emph{not }Borel
reducible to $\cong _{n}^{+}$. It follows that the same conclusions hold for
the relations $\cong _{n}^{AF}$: For every $n\in \mathbb{N}$%
\begin{equation*}
\cong _{n}^{AF}<_{B}\cong _{n+1}^{AF}\text{.}
\end{equation*}%
This amounts at saying that it is strictly more difficult to classify
AF-algebras with $K_{0}$-group of rank $n+1$ than classifying AF-algebras
with $K_{0}$-group of rank $n$.

\subsection{Enomorphisms of Bratteli diagrams\label{Subsection: Endomorphism
of Bratteli diagrams}}

\begin{definition}
Let $T=\left( l,w,m\right) $ be a Bratteli diagram. We say that an element $%
q=\left( q_{n}\right) _{n\in \omega }\in \left( \omega ^{\omega \times
\omega }\right) ^{\omega }$ is an \emph{endomorphism} of $T$, if for every $%
n\in \omega $, $i\in l(n) $ and $t^{\prime }\in l\left( n+1\right) $, the
following identity holds%
\begin{equation}
\sum_{t\in l(n+1) }m_{n}( i,t) q_{n+1}( t,t^{\prime }) =\sum_{t\in l( n+1)
}q_{n}( i,t) m_{n+1}( t,t^{\prime }) \text{.\label{Equation: BD endomorphism}%
}
\end{equation}
\end{definition}

The set $End_{\mathcal{BD}}$ of pairs $\left( T,q\right) \in \mathcal{BD}%
\times \left( \omega ^{\omega \times \omega }\right) ^{\omega }$ such that $%
T $ is a Bratteli diagram and $q$ is an endomorphism of $T$, is Borel.

We proceed to describe how an endomorphism of a Bratteli diagram, in the
sense of the definition above, gives rise to an endomorphism of the unital
AF-algebra associated with it. Let $\left( F_{n},\varphi _{n}\right) _{n\in
\omega }$ be the inductive system of finite dimensional C*-algebras
associated with $T$, and denote by $A_{T}$ its inductive limit. By
repeatedly applying \cite[Lemma III.2.1]{davidson_c*-algebras_1996}, one can
define unital *-homomorphisms $\psi _{n}\colon F_{n}\rightarrow F_{n+1}$ for 
$n$ in $\omega $, satisfying the following conditions:

\begin{enumerate}
\item $\psi _{n}$ is unitarily equivalent to the *-homomorphism from $F_{n}$
to $F_{n+1}$ such that for every $i\in l(n)$ and $j\in l\left( n+1\right) $
the restriction of $\psi _{n}$ to the $i$-th direct summand of $F_{n}$ and
the $j$-th direct summand of $F_{n+1}$ is a diagonal embedding of $%
q_{n}(i,j) $ copies of $\mathbb{M}_{w_{n}(i)}$ in $\mathbb{M}_{w_{n+1}(j)}$;

\item $\psi _{n}\circ \varphi _{n-1}=\varphi _{n}\circ \psi _{n-1}$ whenever 
$n\geq 1$.
\end{enumerate}

(Notice in particular that $\psi _{0}$ is determined solely by condition
(1).) One thus obtains a one sided intertwining $(\psi_n)_{n\in\omega}$ from 
$(F_n,\varphi_n)_{n\in\omega}$ to itself. We denote by $\psi _{T,q}\colon
A_{T}\rightarrow A_{T}$ the corresponding inductive limit endomorphism.

\begin{proposition}
\label{Proposition: Borel endomorphism AF algebra} Given a Bratteli diagram $%
T$ and an endomorphism $q$ of $T$, there is a Borel way to compute a code
for the endomorphism $\psi _{T,q}$ of $A_{T}$ associated with $q$.
\end{proposition}

\begin{proof}
By Proposition \ref{prop: one sided intertwining}, a code for the limit
homomorphism between two inductive limits of C*-algebras can be computed in
a Borel way. It is therefore enough to show that there is a Borel function
from $End_{\mathcal{BD}}$ to $R_{int}\left( \Gamma (H)\right) $ assigning to
an element $\left( T,q\right) $ of $End_{\mathcal{BD}}$, a code for the
corresponding one sided intertwining system. \newline
\indent Let $((l,w,m),q)$ be an element in $End_{\mathcal{BD}}$, and let $%
\left\{ \xi _{m}^{n}\colon \left( n,m\right) \in \omega \times \omega
\right\} $ be an orthonormal basis of $H$. Denote by $\left( \gamma
^{(n)},\Phi ^{(n)}\right) _{n\in \omega }$ the element of $R_{dir}\left(
\Gamma (H)\right) $ associated with the Bratteli diagram $\left(
l,w,m\right) $ as in the proof of Proposition \ref{Proposition: Bratteli and
AF}. For $n$ in $\omega $, define%
\begin{equation*}
u^{(n)}=\sum_{k\in l(n)}\sum_{i\in w_{n}\left( k\right) }e_{ii}^{(k)},
\end{equation*}%
which is an element of $B(H)$. Observe that $u^{(n)}$ is the unit of $%
C^{\ast }\left( \gamma ^{(n)}\right) $. We define the sequence $\left( \Psi
^{(n)}\right) _{n\in \omega }$ in $\left( \mathcal{U}^{\mathcal{U}}\right)
^{\omega }$ as follows. Let $A_{i,j,k}^{n}$ denote the set of triples%
\begin{equation*}
\left( \sum_{k^{\prime }\in k}w_{n}\left( k^{\prime }\right) q_{n}\left(
k^{\prime },t,n\right) +dw_{n}(k)+i,\sum_{k^{\prime }\in k}w_{n}\left(
k^{\prime }\right) q_{n}\left( k^{\prime },t\right) +dw_{n}(k),t\right)
\end{equation*}%
such that $d$ belongs to $q_{n}\left( k,t\right) $, $i$ and $j$ belong to $%
w_{n}(k)$, and $t$ belongs to $l(n+1)$. Let $\psi ^{(n)}\colon \mathcal{U}%
\rightarrow \mathcal{U}$ be the unique homomorphism of $\mathbb{Q}(i)$%
-algebras satisfying 
\begin{equation*}
\psi ^{(n)}\left( X_{ijk}\right) =\sum_{(a,b,t)\in A_{ijk}^{n}}X_{a,b,t}
\end{equation*}%
for $i,j$ and $k$ in $\omega $. For $p\in \mathcal{U}$, set%
\begin{equation*}
\Psi _{0}^{(n)}(p)=\psi ^{\left( n\right) }(p)\text{.}
\end{equation*}%
%
%
%
%
%
%
%
%
%
%
%
%
%
%
%
%
%
%
%
%
%
%
%
%
%
%
%
%
%
%
%
%
%
%
%
%
%
%
%
%
%
%
%
%
%
%
%
%
%
%
%
%
%
%
%
%
%
%
%
%
%
%
%
%
By construction, the elements 
\begin{equation*}
\left( \Psi _{0}^{(n)}\circ \Phi _{k}^{\left( n-1\right) }\right) (p)\left(
\gamma ^{\left( n-1\right) }\right) \ \ \ \mbox{ and }\ \ \ \left( \Phi
_{k}\circ \Psi _{0}^{\left( n-1\right) }\right) (p)\left( \gamma ^{\left(
n-1\right) }\right)
\end{equation*}%
are unitarily equivalent in $C^{\ast }\left( \gamma ^{(n+1)}\right) $ for
every $k\in \omega $. Using that $\gamma ^{(n)}$ is a unitary for all $n$ in 
$\omega $, choose elements $p_{k}^{(n)}$ in $\mathcal{U}$, for $k$ in $%
\omega $, satisfying the following conditions:

\begin{align*}
\left\Vert p_{k}^{(n) }\left( \gamma ^{(n+1) }\right) p_{k}^{(n) }\left(
\gamma ^{(n+1) }\right) ^{\ast }-1\right\Vert &<\frac{1}{k+1} \\
\left\Vert p_{k}^{(n) }\left( \gamma ^{(n+1) }\right) ^{\ast }p_{k}^{(n)
}\left( \gamma ^{(n+1) }\right) -1\right\Vert &<\frac{1}{k+1} \\
\left\Vert p_{k}^{(n) }\left( \gamma ^{(n+1) }\right) -p_{m}^{(n) }\left(
\gamma ^{(n+1) }\right) \right\Vert &<\frac{1}{\min \left\{ k,m\right\} +1}
\end{align*}
and 
\begin{align*}
& \left\Vert p_{k}^{(n) }\left( \gamma ^{(n+1) }\right) \left( \left( \Psi
_{0}^{n}\circ \Phi _{k}^{n-1}\right) (p) \left( \gamma ^{\left( n-1\right)
}\right) \right) p_{k}^{\left( n\right) }\left( \gamma ^{(n+1) }\right)
^{\ast }\right. \\
& \ \ \ \ \ \ \ \ \ \ \ \left.-\left( \Phi _{k}\circ \Psi _{0}^{\left(
n-1\right) }\right) (p) \left( \gamma ^{\left( n-1\right) }\right)
\right\Vert <\frac{1}{k+1}.
\end{align*}

Finally, define%
\begin{equation*}
\Psi _{k}^{(n)}(p)=p_{k}^{(n)}\Psi _{0}^{(n)}(p)\left( p_{k}^{(n)}\right)
^{\ast }\text{{}}
\end{equation*}%
for all $p$ in $\mathcal{U}$. It is clear that for fixed $n$ in $\omega $,
the sequence $\Psi ^{(n)}=(\Psi _{k}^{(n)})_{k\in \omega }$ is a code for a
*-homomorphism $\widehat{\psi ^{(n)}}\colon \widehat{\gamma }%
^{(n)}\rightarrow \widehat{\gamma }^{(n+1)}$ that moreover satisfies 
\begin{equation*}
\widehat{\Psi }^{(n)}\circ \widehat{\Phi }^{\left( n-1\right) }=\widehat{%
\Phi }^{(n)}\circ \widehat{\Psi }^{\left( n-1\right) }\text{.}
\end{equation*}%
Thus, 
\begin{equation*}
\left( \left( \gamma ^{(n)},\Phi ^{(n)}\right) _{n\in \omega },\left( \gamma
^{(n)},\Phi ^{(n)}\right) _{n\in \omega },\left( \Psi ^{(n)}\right) _{n\in
\omega }\right)
\end{equation*}%
is an element in $R_{int}\left( \Gamma (H)\right) $. It is clear that this
is a code for the one sided intertwining system associated with $((l,m,w),q)$%
, and that it can be computed in a Borel fashion.
\end{proof}

\subsection{Endomorphisms of dimension groups\label{Subsection: Endomorphism
dimension groups}}

Let $\left( G,G^+,u\right) $ be a dimension group. Let us denote by $End_{%
\mathcal{DG}}$ the set of pairs $\left( G,\phi \right) \in \mathcal{DG}%
\times \omega ^{\omega }$ such that $G$ is a dimension group and $\phi $ is
an endomorphism of $G$.

Let $\left( l,w,m\right) $ be a Bratteli diagram, and let

\begin{equation*}
\left( \mathbb{Z}^{l(n) },\left( w_{n}\left( 0\right) ,\ldots ,w_{n}\left(
l(n) -1\right) \right) ,\varphi _{n}\right) _{n\in \omega }
\end{equation*}%
be the inductive system of dimension groups whose inductive limit is the
dimension group $G_{l,w,m}$ associated with $\left( l,w,m\right) $. Fix an
endomorphism $q$ of $\left( l,w,m\right) $, and for $n\in \omega $, define a
positive homomorphism $\psi _{n}\colon \mathbb{Z}^{l(n) }\rightarrow \mathbb{%
Z}^{l\left( n+1\right) }$ by 
\begin{equation*}
\psi _{n}\left( e_{i}^{(l(n))}\right) =\sum_{j\in l(n+1) }q_{n}(i,j)
e^{(l(n+1))}_{j}.
\end{equation*}%
Observe that the sequence $\left( \psi _{n}\right) _{n\in \omega }$ induces
an inductive limit endomorphism $\phi _{((l,w,m),q)}$ of $G_{(l,w,m)}$.

\begin{proposition}
\label{Proposition: Borel endomorphism dimension group}There is a Borel map%
\begin{align*}
End_{\mathcal{DG}} &\rightarrow End_{\mathcal{BD}} \\
\left( G,\phi \right) &\mapsto \left( T^{G},q^{G,\phi }\right),
\end{align*}%
such that the dimension group associated with $T^{G}$ is isomorphic to $G$,
and the endomorphism of the dimension group associated with $T^{G}$
corresponding to $q^{G,\phi }$, is conjugate to $\phi $.
\end{proposition}

\begin{proof}
It is enough to construct, in a Borel way,

\begin{itemize}
\item a Bratteli diagram $\left( l^{G,\phi },w^{G,\phi },m^{G,\phi }\right) $%
,

\item an endomorphism $q^{G,\phi }$ of $\left( l^{G,\phi },w^{G,\phi
},m^{G,\phi }\right) $, and

\item functions $\theta _{n}^{G,\phi }\colon l^{G,\phi }(n)\rightarrow G^{+}$
for $n$ in $\omega $,
\end{itemize}

such that the following conditions hold:

\begin{enumerate}
\item \label{Enumerate item: homomorphism 2}For every $i\in l^{G,\phi }(n)$, 
\begin{equation*}
\theta _{n}^{G,\phi }(i)=\sum_{j\in l^{G}(n+1)}m_{n}^{G,\phi }(i,j)\theta
_{n}^{G,\phi }(j);
\end{equation*}

\item \label{Enumerate item: 1:1 2}For any $k_{0},\ldots ,k_{l(n)-1}\in 
\mathbb{Z}$ such that%
\begin{equation*}
\sum_{i\in l^{G}(n)}k_{i}\theta _{n}^{G,\phi }(i)=0,
\end{equation*}%
we have that 
\begin{equation*}
\sum_{i\in l^{G}(n)}k_{i}m_{n}^{G,\phi }(i,j)=0\text{,}
\end{equation*}%
for every $j\in l^{G,\phi }(n+1)$;

\item \label{Enumerate item: onto 2}For every $x\in G^{+}$, there are $n\in
\omega $ and $i\in l^{G,\phi }(n)$ such that $\theta _{n}^{G,\phi }(i)=x$;

\item \label{Enumerate item: intertwining}$\phi \left( \theta _{n}^{G,\phi
}(i)\right) =q^{G}(i)$ for $n\in \omega $ and $i\in l^{G,\phi }(n)$.
\end{enumerate}

In fact, it is not difficult to see that Conditions (\ref{Enumerate item:
homomorphism 2}), (\ref{Enumerate item: 1:1 2}), (\ref{Enumerate item: onto
2}), and (\ref{Enumerate item: intertwining}) ensure that $\left( \theta
_{n}^{G,\phi }\right) _{n\in \omega }$ defines an isomorphism from the
dimension group associated with $\left( l^{G,\phi },w^{G,\phi },m^{G,\phi
}\right) $ to $G$ that conjugates the endomorphism associated with $%
q^{G,\phi }$ and $\phi $.

We define $l_{n}^{G,\phi },w_{n}^{G,\phi },m_{n}^{G,\phi },q_{n}^{G,\phi }$
and $\theta _{n}^{G,\phi }$ satisfying conditions 1--4 by recursion on $n$.
Define $l^{G,\phi }\left( 0\right) =1$ and $\theta _{0}^{G,\phi }\left(
0\right) =u$. Suppose that $l^{G,\phi }(k)$, $w_{k}^{G,\phi }$, $%
m_{k-1}^{G,\phi }$, and $\theta _{k}^{G,\phi }$ have been defined for $k\leq
n$. Define 
\begin{equation*}
\theta ^{\prime }\colon 2l^{G,\phi }(n)+1=\{0,\ldots ,2l^{G,\phi
}(n)\}\rightarrow \omega
\end{equation*}
by%
\begin{equation*}
\theta ^{\prime }(i)=%
\begin{cases}
\theta ^{G,\phi }(i) & \text{ if }0\leq i<l^{G,\phi }(n)\text{,} \\ 
\phi \left( \theta ^{G,\phi }(i-l^{G,\phi }(n))\right) & \text{ if }%
l^{G,\phi }(n)\leq i<2l^{G,\phi }(n)\text{,} \\ 
n & \text{ if }i=2l^{G,\phi }(n)\mbox{ and }n\in G^{+}, \\ 
u & \text{ otherwise.}%
\end{cases}%
\end{equation*}%
Suppose that the positive integer $N$, and the functions $\Phi \colon
N\rightarrow G^{+}$ and $g\colon (2l^{G,\phi }(n)+1)\times N\rightarrow
\omega $ are obtained via Lemma \ref{Lemma: Shen's condition} from $%
2l^{G,\phi }(n)+1$ and $\theta ^{\prime }$, and let $N^{\prime }\in \omega $%
, $\Phi ^{\prime }\colon N^{\prime }\rightarrow \omega $, and $g^{\prime
}\colon N\times N^{\prime }\rightarrow \omega $ satisfy the conclusion of
Lemma \ref{Lemma: Shen's condition} for the choices $N$ and $\Phi $. Define
now:

\begin{align*}
l^{G,\phi }(n+1)& =N^{\prime }; \\
w_{n}^{G,\phi }(j)& =\sum_{i\in l^{G}(n)}w_{n}^{G,\phi }(i)m_{n}^{G,\phi
}(i,j) \\
m_{n}^{G,\phi }(i,j)& =%
\begin{cases}
\sum_{t\in N}g(i,t)g^{\prime }(t,j) & \text{ if }i\in l^{G,\phi }(n)%
\mbox{
and }j\in l^{G,\phi }(n+1)\text{,} \\ 
0 & \text{ otherwise.}%
\end{cases}
\\
q_{n}^{G,\phi }(i,j)& =\sum_{t\in N}g(2i,t)g^{\prime }(t,j) \\
\theta _{n+1}^{G,\phi }& =\Phi .
\end{align*}%
It is not difficult to check that this recursive construction gives maps
satisfying conditions 1--4.
\end{proof}

\begin{corollary}
\label{Corollary: Borel map from DG to AF} There is a Borel map that assigns
to a dimension group $G$ with a distinguished endomorphism $\phi$, a code
for a unital AF-algebra $A$ and a code for an endomorphism $\rho$ of $A$,
such that the $K_{0}$-group of $A$ is isomorphic to $G$ as dimension groups
with order units, and the endomorphism of the $K_{0}$-group of $A$
corresponding to $\rho$ is conjugate to $\phi$.
\end{corollary}

\begin{proof}
Let $G$ be a dimension group, and let $\phi $ be an endomorphism of $G$.
Using Proposition \ref{Proposition: Borel endomorphism dimension group},
choose in a Borel way a Bratteli diagram $\left( l,m,w\right) $ and an
endomorphism $q$ of $\left( l,m,w\right) $ such that $G$ is isomorphic to
the dimension group associated with $\left( l,m,w\right) $, and $\rho $ is
conjugate to the endomorphism associated with $q$. Use Proposition \ref%
{Proposition: Bratteli and AF} to choose in a Borel way, a unital AF-algebra 
$A$ whose Bratteli diagram is $(l,m,w)$. Apply Proposition \ref{Proposition:
Borel endomorphism AF algebra} to choose in a Borel way an endomorphism $%
\rho $ of $A$ whose induced endomorphism of the Bratteli diagram is $q$. It
is clear from the construction that the $K_{0}$-group of $A$ is isomorphic
to $G$. Moreover, $\phi $ is conjugate to the endomorphism of the $K_{0}$%
-group of $A$ corresponding to $\rho $. Therefore, the result follows from
Proposition \ref{Proposition: Borel endomorphism dimension group} and
Proposition \ref{Proposition: Borel endomorphism AF algebra}.
\end{proof}

\section{Cocycle conjugacy of automorphisms of \texorpdfstring{$%
\mathcal{O}_{2}$}{O2} \label{Chapter: cocycle conjugacy of automorphisms}}

\subsection{Strongly self-absorbing C*-algebras}

Upon studying the literature around Elliott's classification program, it is
clear that certain C*-algebras play a central role in major stages of the
program: UHF-algebras (particularly those of infinite type), the Cuntz
algebras $\mathcal{O}_{2}$ and $\mathcal{O}_{\infty }$ \cite%
{cuntz_simple_1977}, and, more recently, the Jiang-Su algebra $\mathcal{Z}$ 
\cite{jiang_simple_1999}. In \cite{toms_strongly_2007}, Toms and Winter were
able to abstract the property which singles these algebras out among other
similar C*-algebras. The relevant notion is that of strongly self-absorbing
C*-algebras, which we define below; see also \cite[Definition 1.3]%
{toms_strongly_2007}.

\begin{definition}
\label{Definition: ssa}Let $\mathcal{D}$ be a separable, unital, infinite
dimensional C*-algebra. Denote by $\mathcal{D}\otimes \mathcal{D}$ the
completion of the algebraic tensor product $\mathcal{D}\odot \mathcal{D}$
with respect to any compatible C*-norm on $\mathcal{D}\odot \mathcal{D}$. We
say that $\mathcal{D}$ is \emph{strongly self-absorbing} if there exists an
isomorphism $\varphi \colon \mathcal{D}\rightarrow \mathcal{D}\otimes 
\mathcal{D}$ which is approximately unitarily equivalent to the map $%
a\mapsto a\otimes 1_{\mathcal{D}}$.
\end{definition}

It is shown in \cite{toms_strongly_2007} that a C*-algebra $\mathcal{D}$
satisfying Definition \ref{Definition: ssa} is automatically nuclear. In
particular the choice of the tensor product norm on $\mathcal{D}\odot 
\mathcal{D}$ is irrelevant. By \cite[Examples 1.14]{toms_strongly_2007} the
following C*-algebras are strongly-self-absorbing: UHF-algebras of infinite
type, the Cuntz algebras $\mathcal{O}_{2}$ and $\mathcal{O}_{\infty }$, the
tensor product of a UHF-algebra of infinite type and $\mathcal{O}_{\infty }$%
, and the Jiang-Su algebra. No other strongly self-absorbing C*-algebra is
currently known.

\begin{definition}
\label{Definition: absorbing}Suppose that $D$ is a nuclear C*-algebra. A
C*-algebra $A$ absorbs $D$ tensorially --or is $D$-\emph{absorbing}-- if the
tensor product $A\otimes D$ is isomorphic to $A$.
\end{definition}

The particular case of Theorem \ref{Theorem: ssa-absorbing} when $\mathcal{D}
$ is the Jiang-Su algebra $\mathcal{Z}$ has been proved in \cite[Theorem A.1]%
{farah_descriptive_2012}.

\begin{theorem}
\label{Theorem: ssa-absorbing} Suppose that $\mathcal{D}$ is a strongly
self-absorbing C*-algebra. The set of $\gamma \in \Gamma (H)$ such that $%
C^{\ast }(\gamma )$ is a $\mathcal{D}$-absorbing unital C*-algebra is Borel.
\end{theorem}

\begin{proof}
By \cite[Lemma 3.14]{farah_turbulence_2014}, the set $\Gamma _{u}(H)$ of $%
\gamma \in \Gamma (H)$ such that $C^{\ast }(\gamma )$ is unital, is Borel.
Moreover, there is a Borel function $Un\colon \Gamma _{u}(H)\rightarrow B(H)$
such that $Un(\gamma )$ is the unit of $C^{\ast }\left( \gamma \right) $ for
every $\gamma \in \Gamma _{u}(H)$. Let $\left\{ d_{n}\colon n\in \omega
\right\} $ be an enumeration of a dense subset of $\mathcal{D}$ such that $%
d_{0}=1$, and let $\left\{ p_{n}\colon n\in \omega \right\} $ be an
enumeration of $\mathcal{U}$. By \cite[Theorem 2.2]{toms_strongly_2007}, or 
\cite[Theorem 7.2.2]{rordam_classification_2002}, a unital C*-algebra $A$ is 
$\mathcal{D}$-absorbing if and only if for every $n,m\in \mathbb{N}$ and
every finite subset $F$ of $A$, there are $a_{0},a_{1},\ldots ,a_{n}\in A$
such that

\begin{itemize}
\item $a_{0}$ is the unit of $A$,

\item $\left\Vert xa_i-a_ix \right\Vert <\frac{1}{m}$ for every $i\in n$ and 
$x\in F$, and

\item $\left\Vert p_{i}\left( a_{0},\ldots ,a_{n}\right) -p_{i}\left(
d_{0},\ldots ,d_{n}\right) \right\Vert <\frac{1}{m} $ for every $i\in m$.
\end{itemize}

Let $\gamma \in \Gamma (H)$ be such that $C^{\ast }(\gamma )$ is unital.
Then $C^{\ast }(\gamma )$ is $\mathcal{D}$-absorbing if and only if for
every $n,m\in \mathbb{N}$ there are $k_{1},\ldots ,k_{n}\in \omega $ such
that

\begin{itemize}
\item $\left\Vert \gamma _{i}\gamma _{k_{j}}-\gamma _{k_{j}}\gamma _{i}
\right\Vert <\frac{1}{m}$ for $i\in m$ and $1\leq j\leq n$,

\item $\left\Vert p_{i}\left( Un(\gamma ),\gamma _{k_{1}},\ldots ,\gamma
_{k_{n}}\right) -p_{i}\left( d_{0},\ldots ,d_{n}\right) \right\Vert <\frac{1%
}{m}$ for every $i\in m$.
\end{itemize}

This shows that the set of $\gamma \in \Gamma (H)$ such that $C^{\ast
}(\gamma )$ is unital and $\mathcal{D}$-absorbing, is Borel.
\end{proof}

\subsection{Background on Kirchberg algebras and the UCT}

A simple C*-algebra $A$ is called \emph{purely infinite}, if for every $a$
and $b$ in $A$ with $a\neq 0$, there are $x$ and $y$ in $A$ such that $xay=b$
\cite[V.2.3.3]{blackadar_operator_2006}. For a unital C*-algebra $A$ this
equivalent to the assertion that for every $a$ in $A$ with $a\neq 0$, there
are $x$ and $y$ in $A$ such that $xay=1$.

A \emph{Kirchberg algebra }is a simple, separable, nuclear, purely infinite
C*-algebra \cite[Chapter 4]{rordam_classification_2002}.

The following result of Kirchberg, which is \cite[Theorem 3.15]%
{kirchberg_embedding_2000}, is crucial in the study of purely infinite
simple C*-algebras.

\begin{theorem}[Kirchberg, \protect\cite{kirchberg_embedding_2000}]
\label{thm: Kirch OI isom thm}If $A$ is a nuclear simple C*-algebra, then $A$
is purely infinite if and only if $A$ is $\mathcal{O}_{\infty }$-absorbing.
\end{theorem}

We proceed to define what it means for a C*-algebra to satisfy the Universal
Coefficient Theorem, or UCT for short. In addition to the $K_0$ group, one
can associate to a C*-algebra another abelian group, namely the $K_1$-group.
Its definition and its basic properties can be found in \cite[Chapter 8]%
{rordam_introduction_2000}. (We point out that unlike the $K_0$-group, the $%
K_1$-group of a C*-algebra does not carry any natural ordering or positive
cone.) \newline
\indent If $A$ is a C*-algebra, we denote by $K_\ast(A)$ the $\mathbb{Z}_2$%
-graded abelian group $K_\ast(A)=K_0(A)\oplus K_1(A)$. If $A$ and $B$ are
C*-algebras, a group homomorphism $\nu\colon K_\ast(A)\to K_\ast(B)$ is said
to have \emph{degree zero} if $\nu(K_j(A))\subseteq K_j(B)$ for $j=0,1$, and
it is said to have \emph{degree one} if $\nu(K_j(A))\subseteq K_{1-j}(B)$
for $j=0,1$.

\begin{definition}
\label{df: UCT} Let $A$ and $B$ be separable C*-algebras. We say that the
pair $(A,B)$ \emph{satisfies the UCT} if the following conditions are
satisfied:

\begin{enumerate}
\item The natural degree zero map $\tau _{A,B}\colon KK(A,B)\rightarrow 
\mathrm{Hom}(K_{\ast }(A),K_{\ast }(B))$ defined in \cite%
{kasparov_operator_1980}, is surjective.

\item The natural degree one map $\mu _{A,B}\colon \mathrm{ker}(\tau
_{A,B})\rightarrow \mathrm{Ext}(K_{\ast }(A),K_{\ast +1}(B))$ defined in 
\cite{kasparov_operator_1980}, is an isomorphism.
\end{enumerate}

If this is the case, by setting $\varepsilon _{A,B}=\mu _{A,B}^{-1}\colon 
\mathrm{Ext}(K_{\ast }(A),K_{\ast +1}(B))\rightarrow KK(A,B)$, we obtain a
short exact sequence 
\begin{equation*}
\xymatrix@=1.35em{0\ar[r] & \mathrm{Ext}(K_\ast(A),K_{\ast+1}(B))\ar[r]^-{%
\varepsilon_{A,B}}& KK(A,B)\ar[r]^-{\tau_{A,B}}&\mathrm{Hom}%
(K_\ast(A),K_\ast(B))\ar[r]&0,}
\end{equation*}%
which is natural on both variables because so are $\tau _{A,B}$ and $\mu
_{A,B}$.\newline
\indent We further say that $A$ \emph{satisfies the UCT}, if $(A,B)$
satisfies the UCT for every separable C*-algebra $B$.
\end{definition}

We say that two separable nuclear C*-algebras $A$ and $B$ are \emph{$KK$%
-equivalent} if there exist $x\in KK(A,B)$ and $y\in KK(B,A)$ such that $%
xy=1_B$ and $yx=1_A$.\newline
\indent The class of all separable nuclear C*-algebras that satisfy the UCT,
usually called the \emph{bootstrap} class $\mathcal{N}$, can be
characterized as the smallest class of separable nuclear C*-algebras
containing the complex numbers $\mathbb{C}$, and closed under the following
operations:

\begin{itemize}
\item Countable direct sums,

\item Crossed products by $\mathbb{Z}$ and by $\mathbb{R}$,

\item Two out of three in extensions,

\item $KK$-equivalence.
\end{itemize}

It can be shown that a C*-algebra belongs to $\mathcal{N}$ if and only it is 
$KK$-equivalent to a commutative C*-algebra. It is a long standing open
problem whether all separable nuclear C*-algebras belong to $\mathcal{N}$.

The UCT plays a key role in the classification of C*-algebras, and it is a
standard assumption in most classification theorems thus far available. (In
many situations, as is the case for AF-algebras, the UCT is automatically
satisfied.) We recall here Kirchberg-Phillips classification theorem for
(unital) Kirchberg algebras satisfying the UCT.

\begin{theorem}[Kirchberg,Phillips \protect\cite{kirchberg_exact_1995,
phillips_classification_2000}]
\label{Theorem: KP classification} Let $A$ and $B$ be unital Kirchberg
algebras satisfying the UCT. Then $A$ and $B$ are isomorphic if and only if
there is a degree zero group isomorphism 
\begin{equation*}
\varphi_\ast\colon K_\ast(A)\to K_\ast(B)
\end{equation*}
such that $\varphi_0([1_A])=[1_B]$.
\end{theorem}

We point out that if $A$ is a Kirchberg algebra, then the order structure on 
$K_0(A)$ is trivial, this is, $K_0(A)^+=K_0(A)$. Thus, every element in $%
K_0(A)$ is positive.

\subsection{Borel spaces of Kirchberg algebras\label{Subsection: Borel
spaces of Kirchberg algebras}}

We will denote by $\Gamma _{uKir}(H)$ the set of $\gamma \in \Gamma (H)$
such that $C^{\ast }(\gamma )$ is a \emph{unital} Kirchberg algebra.

\begin{proposition}
\label{Proposition: Kirchberg Borel} The set $\Gamma _{uKir}(H)$ is Borel.
\end{proposition}

\begin{proof}
Corollary 7.5 of \cite{farah_turbulence_2014} asserts that the set $\Gamma
_{uns}(H)$ of $\gamma \in \Gamma (H)$ such that $C^{\ast }(\gamma )$ is
unital, nuclear, and simple is Borel. The result then follows from this fact
together with Theorem \ref{Theorem: ssa-absorbing}.
\end{proof}

\begin{definition}
\label{definition: Oinfty standard} Fix a projection $p$ in $\mathcal{O}%
_{\infty }$ such that $[p]=0$ in $K_0(\mathcal{O}_\infty)\cong\mathbb{Z}$.
Define the \emph{standard Cuntz algebra} $\mathcal{O}_\infty^{st}$ to be the
corner $p\mathcal{O}_\infty p$.
\end{definition}

\begin{remark}
The C*-algebra $\mathcal{O}_{\infty }^{st}$ is a unital Kirchberg algebra
that satisfies the UCT, with $K$-theory given by 
\begin{equation*}
(K_{0}(\mathcal{O}_{\infty }^{st}),\left[ 1_{\mathcal{O}_{\infty }^{st}}%
\right] ,K_{1}(\mathcal{O}_{\infty }^{st}))\cong (\mathbb{Z},0,0).
\end{equation*}%
Moreover, it is the unique, up to isomorphism, unital Kirchberg algebra
satisfying the UCT with said $K$-theory. In particular, a different choice
of the projection $p$ in Definition \ref{definition: Oinfty standard} (as
long as its class on $K$-theory is $0$), would yield an isomorphic
C*-algebra.
\end{remark}

We point out that, even though there is an isomorphism $\mathcal{O}_{\infty
}^{st}\otimes \mathcal{O}_{\infty }^{st}\cong \mathcal{O}_{\infty }^{st}$
(see comments on page 262 of \cite{izumi_finite_2004}), the C*-algebra $%
\mathcal{O}_{\infty }^{st}$ is not strongly self-absorbing. Indeed, if $%
\mathcal{D}$ is a strongly self-absorbing C*-algebra, then the infinite
tensor product $\bigotimes_{n=1}^{\infty }\mathcal{D}$ of $\mathcal{D}$ with
itself, is isomorphic to $\mathcal{D}$. However, $\bigotimes_{n=1}^{\infty }%
\mathcal{O}_{\infty }^{st}$ is isomorphic to $\mathcal{O}_{2}$, and thus $%
\mathcal{O}_{\infty }^{st}$ is not strongly self-absorbing.\newline
\indent We proceed to give a $K$-theoretic characterization of those unital
Kirchberg algebras that absorb $\mathcal{O}_{\infty }^{st}$. Our
characterization will be used to show that the set of all $\mathcal{O}%
_{\infty }^{st}$-absorbing unital Kirchberg algebras is Borel.\newline
\indent For use in the proof of the following lemma, we recall here that if $%
A$ and $B$ are nuclear separable C*-algebras, and at least one of them
satisfies the UCT, then the $K$-groups of their tensor product $A\otimes B$
are ``essentially'' determined by the $K$-groups of $A$ and $B$, up to an
extension problem. This is the content of the Künneth formula, which will be
needed in the next proof. We refer the reader to \cite%
{schochet_topological_1982} for the precise statement and its proof; see
also Remark 7.11 in \cite{rosenberg_kunneth_1987}.

\begin{lemma}
\label{Lemma: Ointy-standard} Let $A$ be a unital Kirchberg algebra. Then
the following are equivalent:

\begin{enumerate}
\item $A$ is $\mathcal{O}_{\infty }^{st}$-absorbing.

\item The class $[1_{A}]$ of the unit of $A$ in $K_{0}(A)$ is zero.
\end{enumerate}
\end{lemma}

\begin{proof}
We first show that (1) implies (2). Since $\mathcal{O}_{\infty }^{st}$
satisfies the UCT, the Künneth formula applied to $A\otimes \mathcal{O}%
_{\infty }^{st}$ gives 
\begin{equation*}
K_{0}(A\otimes \mathcal{O}_{\infty }^{st})\cong K_{0}(A)\ \ \mbox{ and }\ \
K_{1}(A\otimes \mathcal{O}_{\infty }^{st})\cong K_{1}(A),
\end{equation*}%
with $\left[ 1_{A\otimes \mathcal{O}_{\infty }^{st}}\right] =0$ as an
element in $K_{0}(A)$. The claim follows since any isomorphism $A\otimes 
\mathcal{O}_{\infty }^{st}\cong A$ must map the unit of $A\otimes \mathcal{O}%
_{\infty }^{st}$ to the unit of $A$.\newline
\indent Let us now show that (2) implies (1). Fix a non-zero projection $p$
in $\mathcal{O}_{\infty }$ such that $[p]=0$ as an element of $K_{0}(%
\mathcal{O}_{\infty })$. Then $1_{A}\otimes 1_{\mathcal{O}_{\infty }}$,
which is an element of $A\otimes \mathcal{O}_{\infty }$, represents the zero
group element in $K_{0}(A\otimes \mathcal{O}_{\infty })$. Likewise, $%
1_{A}\otimes p$ also represents the zero group element in $K_{0}(A\otimes 
\mathcal{O}_{\infty })$. Since any two non-zero projections in a Kirchberg
algebra are Murray-von Neumann equivalent if and only if they determine the
same class in $K$-theory (see \cite{cuntz_k-theory_1981-1}), it follows that
there is an isometry $v$ in $A\otimes \mathcal{O}_{\infty }$ such that 
\begin{equation*}
vv^{\ast }=1_{A}\otimes p.
\end{equation*}%
Define a map 
\begin{equation*}
\varphi _{0}\colon A\otimes _{alg}\mathcal{O}_{\infty }\rightarrow
(1_{A}\otimes p)(A\otimes \mathcal{O}_{\infty })(1_{A}\otimes p)\cong
A\otimes \mathcal{O}_{\infty }^{st}
\end{equation*}%
by $\varphi _{0}(a\otimes b)=v(a\otimes b)v^{\ast }$, for $a$ in $A$ and $b$
in $\mathcal{O}_{\infty }$, and extended linearly. It is straightforward to
check that $\varphi _{0}$ extends to a *-homomorphism $\varphi \colon
A\otimes \mathcal{O}_{\infty }\rightarrow A\otimes \mathcal{O}_{\infty
}^{st} $. We claim that $\varphi $ is an isomorphism. For this, it is enough
to check that the *-homomorphism 
\begin{equation*}
\psi \colon (1_{A}\otimes p)(A\otimes \mathcal{O}_{\infty })(1_{A}\otimes
p)\rightarrow A\otimes \mathcal{O}_{\infty }
\end{equation*}%
given by $\psi (x)=v^{\ast }xv$ for all $x$ in $(1_{A}\otimes p)(A\otimes 
\mathcal{O}_{\infty })(1_{A}\otimes p)$, is an inverse for $\varphi $. This
is immediate since $(1_{A}\otimes p)x(1_{A}\otimes p)=x$ for all $x$ in $%
(1_{A}\otimes p)(A\otimes \mathcal{O}_{\infty })(1_{A}\otimes p)$.\newline
\indent Once we have $A\otimes \mathcal{O}_{\infty }^{st}\cong A\otimes 
\mathcal{O}_{\infty }$, the result follows from the fact that there is an
isomorphism $A\cong A\otimes \mathcal{O}_{\infty }$ by Kirchberg's $\mathcal{%
O}_{\infty }$-isomorphism Theorem (here reproduced as Theorem \ref{thm:
Kirch OI isom thm}). This finishes the proof of the lemma.
\end{proof}

\begin{corollary}
\label{cor: abs of OIst is Borel} The set of all $\gamma \in \Gamma (H)$
such that $C^{\ast }(\gamma )$ is a $\mathcal{O}_{\infty }^{st}$-absorbing
unital Kirchberg algebra, is Borel.
\end{corollary}

\begin{proof}
This follows from Lemma \ref{Lemma: Ointy-standard}, together with the fact
that the $K$-theory of a C*-algebra and the class of its unit in $K_{0}$ can
be computed in a Borel fashion; see \cite[Section 3.3]%
{farah_descriptive_2012}.
\end{proof}

It follows from Theorem \ref{Theorem: ssa-absorbing}, together with the fact
that a the UHF-algebra of infinite type $\mathbb{M}_{n^{\infty }}$ is
strongly self-absorbing, that the set of all $\gamma \in \Gamma (H)$ such
that $C^{\ast }(\gamma )$ is a $\mathbb{M}_{n^{\infty }}$-absorbing unital
Kirchberg algebra is Borel.

\subsection{Isomorphism of $p$-divisible torsion free abelian groups\label%
{Section: Isomorphism of p-divisible}}

\begin{definition}
Let $G$ be an abelian group and let $n$ be a positive integer.

\begin{enumerate}
\item We say that $G$ is \emph{$n$-divisible}, if for every $x$ in $G$ there
exists $y$ in $G$ such that $x=ny$.

\item We say that $G$ is \emph{uniquely $n$-divisible}, if for every $x$ in $%
G$ there exists a unique $y$ in $G$ such that $x=ny$.
\end{enumerate}

Given a set $S$ of positive integers, we say that $G$ is (uniquely) $S $%
-divisible, if $G$ is (uniquely) $n$-divisible for every $n$ in $S$.
\end{definition}

It is clear that if $n$ is a positive integer, then any $n$-divisible
torsion free abelian group is uniquely $n$-divisible.

It is easily checked that the following classes of abelian groups are Borel
subsets of the standard Borel space of groups $\mathcal{G}$:

\begin{itemize}
\item Torsion free groups;

\item $n$-divisible groups, for any positive integer $n$;

\item Uniquely $n$-divisible groups.
\end{itemize}

The main result of \cite{hjorth_isomorphism_2002} asserts that if $\mathcal{C%
}$ is any class of countable structures such that the relation $\cong _{%
\mathcal{C}}$ of isomorphisms of elements of $\mathcal{C}$ is Borel, then $%
\cong _{\mathcal{C}}$ is Borel reducible to the relation $\cong _{TFA}$ of
isomorphism of torsion free abelian groups. Moreover, \cite[Theorem 1.1]%
{downey_isomorphism_2008} asserts that $\cong _{TFA}$ is a complete analytic
set and, in particular, not Borel.

\begin{proposition}
\label{Proposition: complete analytic p-divisible} Suppose that $\mathcal{P}$
is a set of prime numbers which is coinfinite in the set of all primes. If $%
\mathcal{C}$ is any class of countable structures such that the relation $%
\cong _{\mathcal{C}}$ of isomorphism of elements of $\mathcal{C}$ is Borel,
then $\mathcal{C}$ is Borel reducible to the relation of isomorphism of
torsion free $\mathcal{P}$-divisible groups. Moreover, the latter
equivalence relation is a complete analytic set and, in particular, not
Borel.
\end{proposition}

\begin{proof}
A variant of the argument used in the proof of the main result of \cite%
{hjorth_isomorphism_2002} can be used to prove the first assertion. Indeed,
the only modification needed is in the definition of the \emph{group eplag }%
associated with an \emph{excellent prime labeled graph }as in \cite[Section 2%
]{hjorth_isomorphism_2002} (we refer to \cite{hjorth_isomorphism_2002} for
the definitions of these notions). Suppose that $\left( V,E,f\right) $ is an
excellent prime labeled graph such that the range of $f$ is disjoint from $%
\mathcal{P}$. Denote by $\mathbb{Q}^{\left( V\right) }$ the direct sum%
\begin{equation*}
\mathbb{Q}^{(V)}=\bigoplus_{v\in V}\mathbb{Q}
\end{equation*}%
of copies of $\mathbb{Q}$ indexed by $V$, and identify an element $v$ of $V$
with the corresponding copy of $\mathbb{Q}$ in $\mathbb{Q}^{\left( V\right)
} $. We define the $\mathcal{P}$-divisible group eplag $\mathcal{G}_{%
\mathcal{P}}\left( V,E,f\right) $ associated with $\left( V,E,f\right) $, to
be the subgroup of $\mathbb{Q}^{\left( V\right) }$ generated by%
\begin{equation*}
\left\{ \frac{v}{p^{n}f( v) ^{m}}, \frac{v+w}{p^{n}f( \left\{ v,w\right\}) }%
\colon v\in V,\left\{ v,w\right\} \in E, n,m\in\omega, p\in\mathcal{P}%
\right\} \text{.}
\end{equation*}%
It is easy to check that $\mathcal{G}_{\mathcal{P}}\left( V,E,f\right) $ is
indeed a torsion-free $\mathcal{P}$-divisible abelian group. The group eplag 
$\mathcal{G}\left( V,E,f\right) $ as defined in \cite[Section 2]%
{hjorth_isomorphism_2002}, is the particular case of this definition with $%
\mathcal{P}=\varnothing $. The same argument as in \cite%
{hjorth_isomorphism_2002}, where

\begin{enumerate}
\item the group eplag $\mathcal{G}\left( V,E,f\right) $ is replaced
everywhere by $\mathcal{G}_{\mathcal{P}}\left( V,E,f\right) $, and

\item all the primes are chosen from the \emph{complement }of $\mathcal{P}$,
\end{enumerate}

gives a proof of the first claim of this proposition.

The second claim follows by modifying the argument in \cite%
{downey_isomorphism_2008} and, in particular, the construction of the
torsion-free abelian group associated with a tree on $\omega $ as in \cite[%
Theorem 2.1]{downey_isomorphism_2008}. Choose injective enumerations $\left(
p_{n}\right) _{n\in \omega }$ and $\left( q_{n}\right) _{n\in \omega }$ of
disjoint subsets of the complement of $\mathcal{P}$ in the set of all
primes, and let $T$ be a tree on $\omega $. Define the excellent prime
labeled graph $\left( V_{T},E_{T},f_{T}\right) $ as follows. The graph $%
\left( V_{T},E_{T}\right) $ is just the tree $T$, and%
\begin{equation*}
f\colon V_{T}\cup E_{T}\rightarrow \left\{ p_{n},q_{n}\colon n\in \omega
\right\}
\end{equation*}%
is defined by

\begin{equation*}
f(x)=\left\{ 
\begin{array}{lll}
q_{n} & \hbox{if $x$ is a vertex in the $n$-th level of $T$;} &  \\ 
p_{n} & 
\hbox{if $x$ is an edge between the $n$-th and $n+1$-th levels of
$T$.} & 
\end{array}%
\right.
\end{equation*}%
Define the $\mathcal{P}$-divisible torsion free abelian group $G_{\mathcal{P}%
}\left( T\right) $ to be the group eplag $\mathcal{G}_{\mathcal{P}}\left(
V_{T},E_{T},f_{T}\right) $ as defined in the first part of the proof. The
same proof as that of \cite[Theorem 2.1]{downey_isomorphism_2008} shows the
following facts: If $T$ and $T^{\prime }$ are isomorphic trees, then the
groups $G_{\mathcal{P}}\left( T\right) $ and $G_{\mathcal{P}}\left(
T^{\prime }\right) $ are isomorphic. On the other hand, if $T$ is
well-founded and $T^{\prime }$ is ill-founded, then $G_{\mathcal{P}}\left(
T\right) $ and $G_{\mathcal{P}}\left( T^{\prime }\right) $ are not
isomorphic. The second claim of this proposition can now be proved as \cite[%
Theorem 1.1]{downey_isomorphism_2008}.
\end{proof}

\subsection{Automorphisms of $\mathcal{O}_{2}$\label{Subsection:
automorphisms of O2}}

Denote by $\mathrm{\mathrm{Aut}}(\mathcal{O}_{2})$ the Polish group of
automorphisms of $\mathcal{O}_{2}$ with respect to the topology of pointwise
convergence. Given a positive integer $n$, the closed subspace $\mathrm{%
\mathrm{Aut}}_{n}(\mathcal{O}_{2})$, of automorphisms of $\mathcal{O}_{2}$
of order $n$ can be identified with the space of actions of $\mathbb{Z}_{n}$
on $\mathcal{O}_{2}$.

\begin{definition}
(See \cite[Definition 3.6]{izumi_finite_2004}) An automorphism $\alpha $ in $%
\mathrm{\mathrm{Aut}}_{n}(\mathcal{O}_{2})$ is said to be \emph{%
approximately representable} if for every $\varepsilon >0$ and for every
finite subset $F$ of $\mathcal{O}_{2}$, there exists a unitary $u$ of $%
\mathcal{O}_{2}$ such that

\begin{enumerate}
\item $\left\Vert u^{n}-1\right\Vert <\varepsilon $,

\item $\left\Vert \alpha (u)-u\right\Vert <\varepsilon $, and

\item $\left\Vert \alpha (a)-uau^{\ast }\right\Vert <\varepsilon $ for every 
$a\in F$.
\end{enumerate}
\end{definition}

It is clear that the set of approximately representable automorphisms of
order $n$ of $\mathcal{O}_{2}$ is a $G_{\delta }$ subset of $\mbox{Aut}_{n}(%
\mathcal{O}_{2})$.

For a prime number $p$, we introduce below a certain model action of $%
\mathbb{Z}_p$ on $\mathcal{O}_2$. When $p=2$, this is the example
constructed by Izumi on page 262 of \cite{izumi_finite_2004}.

\begin{definition}
Let $p$ be a prime number. Choose projections $q_0,\ldots,q_{p-1}$ in $%
\mathcal{O}_{\infty}^{st}$ with

\begin{enumerate}
\item $\sum_{j=0}^{p-1} q_j=1$

\item $[q_j]=0$ for $j=1,\ldots,p-1$ and $[q_0]=1$ in $K_0(\mathcal{O}%
_{\infty}^{st})\cong\mathbb{Z}$.
\end{enumerate}

Set $u=\sum_{j=0}^{p-1}e^{2\pi ij/p}p_{j},$ which is a unitary of order $p$
in $\mathcal{O}_{\infty }^{st}$. We define an order $p$ automorphism $\nu
_{p}$ of $\mathcal{O}_{2}$ by 
\begin{equation*}
\nu _{p}=\bigotimes_{n=1}^{\infty }\mathrm{Ad}(u).
\end{equation*}
\end{definition}

Recall that the center of a simple unital C*-algebra is trivial, this is, is
equal to the set of scalar multiples of its unit. An action $\alpha $ of $%
\mathbb{Z}_{p}$ on $\mathcal{O}_{2}$ is \emph{pointwise outer }if $\alpha
_{g}$ is not an inner automorphism for every nonzero $g\in \mathbb{Z}_{p}$.

\begin{lemma}
Let $p$ be a prime number, and denote by $\nu_p$ the order $p$ automorphism
of $\mathcal{O}_2$ constructed above. Then $\nu_p$ induces a pointwise outer
action of $\mathbb{Z}_p$ on $\mathcal{O}_2$.
\end{lemma}

\begin{proof}
We claim that it is enough to show that $\nu _{p}$ is not inner. Indeed,
assume $\nu _{p}^{j}$ that is inner for some $j\in \{2,\ldots ,p-1\}$. Since 
$\{j,j^{2},\ldots ,j^{p-1}\}=\mathbb{Z}_{p}$, it follows that $\nu _{p}$ is
inner.\newline
\indent Assume that $\nu _{p}$ is inner, and let $v$ be a unitary in $%
\mathcal{O}_{2}$ such that $\nu _{p}=\mathrm{Ad}(v)$. For $n$ in $\mathbb{N}$%
, set $A_{n}=\bigotimes_{k=1}^{n}\mathcal{O}_{\infty }^{st}$. Given $%
\varepsilon >0$, there are a positive integer $n$ in $\mathbb{N}$ and a
unitary $w$ in $A_{n}$ such that $\Vert v-w\Vert <\varepsilon $. Given $a$
in $\mathcal{O}_{\infty }^{st}$, denote by $b$ its canonical image in the $%
(n+1)$-st tensor factor of $\bigotimes_{k=1}^{\infty }\mathcal{O}_{\infty
}^{st}$, this is, $b=1\otimes \cdots \otimes 1\otimes a\otimes 1\cdots $.
Then 
\begin{align*}
\Vert ua-au\Vert & =\Vert uau^{\ast }-a\Vert \\
& =\Vert \nu (b)-b\Vert \\
& =\Vert vbv^{\ast }-b\Vert \\
& \leq 2\Vert v-w\Vert +\Vert wbw^{\ast }+b\Vert .
\end{align*}%
Since $w$ and $b$ belong to different tensor factors, they commute, so we
conclude that $\Vert ua-au\Vert <2\varepsilon $. Since $\varepsilon $ is
arbitrary and $a$ and $u$ do not depend on it, it follows that $au=ua$.
Since $a$ is also arbitrary, it follows that $u$ is in the center of $%
\mathcal{O}_{\infty }^{st}$, which is trivial. This is a contradiction.
Hence $\nu _{p}$ is not inner, and the result follows.
\end{proof}

We compute the associated crossed product in the following lemma.

\begin{lemma}
Let the notation be as in the previous lemma. Then the crossed product $%
\mathcal{O}_2\rtimes_{\nu_p}\mathbb{Z}_p$ is a unital Kirchberg algebra
satisfying the UCT and with $K$-theory given by 
\begin{equation*}
\left(K_0(\mathcal{O}_2\rtimes_{\nu_p}\mathbb{Z}_p),\left[1_{\mathcal{O}_2
\rtimes_{\nu_p}\mathbb{Z}_p}\right],K_1(\mathcal{O}_2\rtimes_{\nu_p}\mathbb{Z%
}_p) \right)\cong \left(\mathbb{Z}\left[\frac{1}{p}\right],0,\{0\}\right).
\end{equation*}
\end{lemma}

\begin{proof}
\indent The crossed product is well known to be nuclear, unital and
separable. Since $\nu _{p}$ is pointwise outer, \cite[Theorem 3.1]%
{kishimoto_outer_1981} implies that $\mathcal{O}_{2}\rtimes _{\nu _{p}}%
\mathbb{Z}_{p}$ is purely infinite and simple, and thus it is a Kirchberg
algebra.\newline
\indent Given $n$ in $\mathbb{N}$, set $A_{n}=\bigotimes_{k=1}^{n}\mathcal{O}%
_{\infty }^{st}$ and $u_{n}=\bigotimes_{k=1}^{n}u$. Note that $\mathcal{O}%
_{2}=\varinjlim A_{n}$ and $\nu _{p}=\varinjlim \mathrm{Ad}(u_{n})$ with the
obvious inductive limit decomposition and connecting maps. Moreover, there
are natural isomorphism 
\begin{equation*}
\mathcal{O}_{2}\rtimes _{\nu _{p}}\mathbb{Z}_{p}\cong \varinjlim
A_{n}\rtimes _{\mathrm{Ad}(u_{n})}\mathbb{Z}_{p}\cong A_{n}\oplus \cdots
\oplus A_{n}\cong \mathcal{O}_{\infty }^{st}\oplus \cdots \oplus \mathcal{O}%
_{\infty }^{st}.
\end{equation*}%
(There are $p$ direct summands.) Since the UCT passes to inductive limits,
we conclude that $\mathcal{O}_{2}\rtimes _{\nu _{p}}\mathbb{Z}_{p}$
satisfies the UCT.\newline
\indent The $K$-theory is computed analogously as it is done in the proof of 
\cite[Lemma 4.7]{izumi_finite_2004}. The crucial point is showing that the
connecting maps in the inductive limit $\mathbb{Z}^{p}\rightarrow \mathbb{Z}%
^{p}\rightarrow \cdots $ are all constant and given by the matrix 
\begin{equation*}
\left( 
\begin{array}{cccc}
1 & e^{2\pi i/p} & \cdots & e^{2\pi i(p-1)/p} \\ 
e^{2\pi i/p} & e^{2\pi i2/p} & \cdots & 1 \\ 
\vdots & \vdots & \ddots & \vdots \\ 
e^{2\pi i(p-1)/p} & 1 & \cdots & e^{2\pi i(p-2)/p} \\ 
&  &  & 
\end{array}%
\right) .
\end{equation*}%
We omit the details, our case is just notationally more difficult than the
case $p=2$.
\end{proof}

We fix, until the end of this section, a prime number $p$. For a simple
nuclear unital C*-algebra $A$, denote by $\widetilde{\alpha }_{A,p}$ the
order $p$ automorphism of $A\otimes \mathcal{O}_{2}$ defined by $\mbox{id}%
_{A}\otimes \nu _{p}$. By Kirchberg's $\mathcal{O}_{2}$-isomorphism Theorem 
\cite[Theorem 3.8]{kirchberg_embedding_2000}, there is an isomorphism $%
\varphi \colon A\otimes \mathcal{O}_{2}\rightarrow \mathcal{O}_{2}$. Denote
by $\alpha _{A,p}$ the order $p$ automorphism of $\mathcal{O}_{2}$ given by 
\begin{equation*}
\alpha _{A,p}=\varphi \circ \widetilde{\alpha }_{A,p}\circ \varphi ^{-1}.
\end{equation*}%
It is immediate to check that $\alpha _{A,p}$ is approximately
representable, using that $\nu _{p}$ is approximately representable.

%

We will denote by $D_{p}$ the tensor product $\mathcal{O}_{\infty
}^{st}\otimes \mathbb{M}_{p^{\infty }}$. Suppose that $\alpha $ and $\beta $
are actions of a discrete group $G$ on C*-algebras $A$ an $B$. By definition 
$\alpha $ and $\beta $ are \emph{conjugate }if there is an isomorphism $%
\gamma $ from $A$ to $B$ such that $\beta _{g}\circ \gamma =\gamma \circ
\alpha _{g}$ for every $g\in G$. Similarly $\alpha $ and $\beta $ are \emph{%
cocycle conjugate }if $\beta $ is conjugate to a cocycle perturbation of $%
\alpha $.


\begin{proposition}
\label{Proposition: automorphism and crossed products} Let $A$ and $B$ be $%
D_{p}$-absorbing unital Kirchberg algebras. Then the following statements
are equivalent:

\begin{enumerate}
\item $A$ and $B$ are isomorphic;

\item The actions $\alpha_{A,p}$ and $\alpha_{B,p}$ are conjugate;

\item The actions $\alpha_{A,p}$ and $\alpha_{B,p}$ are cocycle conjugate;

\item The crossed products $\mathcal{O}_{2} \rtimes_{\alpha_{A,p}}\mathbb{Z}%
_p$ and $\mathcal{O}_{2} \rtimes_{\alpha_{B,p}}\mathbb{Z}_p$ are isomorphic.
\end{enumerate}
\end{proposition}

\begin{proof}
(1) implies (2). If $\psi \colon A\rightarrow B$ is an isomorphism, then $%
\psi \otimes \mbox{id}_{\mathcal{O}_{2}}\colon A\otimes \mathcal{O}%
_{2}\rightarrow B\otimes \mathcal{O}_{2}$ conjugates $\mbox{id}_{A}\otimes
\nu _{p}$ and $\mbox{id}_{B}\otimes \nu _{p}$, and hence $\alpha _{A,p}$ and 
$\alpha _{B,p}$ are conjugate.\newline
\indent(2) implies (3) and (3) implies (4) hold in full generality.\newline
\indent(4) implies (1). There are isomorphisms 
\begin{align*}
\mathcal{O}_{2}\rtimes _{\alpha _{A,p}}\mathbb{Z}_{p}& \cong (A\otimes 
\mathcal{O}_{2})\rtimes _{\mbox{id}_{A}\otimes \nu _{p}}\mathbb{Z}_{p} \\
& \cong A\otimes \left( \mathcal{O}_{2}\rtimes _{\nu _{p}}\mathbb{Z}%
_{p}\right) \\
& \cong A\otimes D_{p}.
\end{align*}%
Likewise, $\mathcal{O}_{2}\rtimes _{\alpha _{B,p}}\mathbb{Z}_{p}\cong
B\otimes D_{p}$. The result now follows from the fact that $A$ and $B$ are $%
D_{p}$-absorbing.
\end{proof}

\subsection{Borel reduction\label{Subsection: Borel reduction}}

Denote by $\mathcal{A}_{p}$ the set of all elements $\gamma $ in $\Gamma (H)$
such that $C^{\ast }(\gamma )$ is a unital $D_{p}$-absorbing Kirchberg
algebra. Then $\mathcal{A}_{p}$ is Borel by Theorem \ref{Theorem:
ssa-absorbing} and Corollary \ref{cor: abs of OIst is Borel}. One can, in
fact, regard $\mathcal{A}_{p}$ as the standard Borel space parametrizing $%
D_{p}$-absorbing unital Kirchberg algebras. Thus, the equivalence relation $%
E $ on $\mathcal{A}_{p}$ defined by 
\begin{equation*}
\gamma E\gamma ^{\prime }\ \ \mbox{ if and only if }\ \ C^{\ast }(\gamma
)\cong C^{\ast }\left( \gamma ^{\prime }\right) ,
\end{equation*}%
can be identified with the relation of isomorphism of unital $D_{p}$%
-absorbing Kirchberg algebras.

\begin{theorem}
\label{Proposition: reduction isomorphism to conjugacy} There are Borel
reductions:

\begin{enumerate}
\item From the relation of isomorphism of $D_{p}$-absorbing Kirchberg
algebras, to the relation of cocycle conjugacy of approximately
representable automorphisms of $\mathcal{O}_{2}$ of order $p$.

\item From the relation of isomorphism of $D_{p}$-absorbing Kirchberg
algebras, to the relation of conjugacy of approximately representable
automorphisms of $\mathcal{O}_{2}$ of order $p$.
\end{enumerate}
\end{theorem}

\begin{proof}
In view of Proposition \ref{Proposition: automorphism and crossed products},
and Elliott's theorem $\mathcal{O}_{2}\otimes \mathcal{O}_{2}\cong \mathcal{O%
}_{2}$ \cite{rordam_short_1994}, it is enough to show that there is a Borel
function from $\Gamma _{uKir}(H)$ to $\mathrm{\mathrm{Aut}}_{p}\left( 
\mathcal{O}_{2}\otimes \mathcal{O}_{2}\right) $ that assigns to every $%
\gamma \in \Gamma _{uKir}(H)$, an automorphism $\alpha _{\gamma ,p}$ of $%
\mathcal{O}_{2}\otimes \mathcal{O}_{2}$ which is conjugate to $\mbox{id}%
_{C^{\ast }(\gamma )}\otimes \nu _{p}$.

We follow the notation of \cite[Section 6.1]{farah_turbulence_2014}, and
denote by $SA(\mathcal{O}_{2})$ the space of C*-subalgebras of $\mathcal{O}%
_{2}$. Then $SA(\mathcal{O}_{2})$ is a Borel subset of the Effros Borel
space of closed subsets of $\mathcal{O}_{2}$, as defined in \cite[Section
12.C]{kechris_classical_1995}. It follows from \cite[Theorem 6.5]%
{farah_turbulence_2014} that the set $SA_{uKir}\left( \mathcal{O}_{2}\right) 
$ of C*-subalgebras of $\mathcal{O}_{2}$ isomorphic to a unital Kirchberg
algebra is Borel. Moreover, again by \cite[Theorem 6.5]%
{farah_turbulence_2014}, there is a Borel function from $\Gamma
_{uKir}\left( \mathcal{O}_{2}\right) $ to $SA_{uKir}\left( \mathcal{O}%
_{2}\right) $ that assigns to an element $\gamma $ of $\Gamma _{uKir}\left( 
\mathcal{O}_{2}\right) $ a subalgebra of $\mathcal{O}_{2}$ isomorphic to $%
C^{\ast }(\gamma )$. It is therefore enough to show that there is a Borel
function from $SA_{uKir}\left( \mathcal{O}_{2}\right) $ to $\mathrm{\mathrm{%
Aut}}_{p}\left( \mathcal{O}_{2}\otimes \mathcal{O}_{2}\right) $ that assigns
to $A\in SA_{uKir}\left( \mathcal{O}_{2}\right) $ an automorphism $\alpha
_{A,p}$ of $\mathcal{O}_{2}\otimes \mathcal{O}_{2}$ conjugate to $\mbox{id}%
_{A}\otimes \nu _{p}$.

Denote by $\mathrm{End}\left( \mathcal{O}_{2}\otimes \mathcal{O}_{2}\right) $
the space of endomorphism of $\mathcal{O}_{2}\otimes \mathcal{O}_{2}$. By 
\cite[Theorem 7.6]{farah_turbulence_2014}, there is a Borel map from $%
SA_{uKir}\left( \mathcal{O}_{2}\right) $ to $\mathrm{End}\left( \mathcal{O}%
_{2}\otimes \mathcal{O}_{2}\right) $ that assigns to an element $A$ in $%
SA_{uKir}\left( \mathcal{O}_{2}\right) $ a unital injective endomorphism $%
\eta _{A}$ of $\mathcal{O}_{2}\otimes \mathcal{O}_{2}$ with range $A\otimes 
\mathcal{O}_{2}$. In particular, $\eta _{A}$ is an isomorphism between $%
\mathcal{O}_{2}\otimes \mathcal{O}_{2}$ and $A\otimes \mathcal{O}_{2}$. For $%
A$ in $SA_{uKir}\left( \mathcal{O}_{2}\right) $, define 
\begin{equation*}
\alpha _{A,p}=\eta _{A}^{-1}\circ \left( \mbox{id}_{A}\otimes \nu
_{p}\right) \circ \eta _{A},
\end{equation*}%
and note that the map $A\mapsto \alpha _{A,p}$ is Borel.

It is enough to show that for every $x,y\in \mathcal{O}_{2}$ and every $%
\varepsilon >0$, the set of all C*-algebras $A$ in $SA_{uKir}\left( \mathcal{%
O}_{2}\right) $ such that%
\begin{equation*}
\left\Vert \alpha _{A,p}(x)-y\right\Vert <\varepsilon ,
\end{equation*}%
is Borel. Fix $x$ and $y$ in $\mathcal{O}_{2}$.

By \cite[Theorem 12.13]{kechris_classical_1995}, there is a sequence of
Borel functions from $SA_{uKir}\left( \mathcal{O}_{2}\right) $ to $\mathcal{O%
}_{2}$, which we will denote by $A\mapsto a_{n}^{A}$ for $n$ in $\omega $,
such that for $A$ in $SA_{uKir}\left( \mathcal{O}_{2}\right) $, the set $%
\{a_{n}^{A}\colon n\in \omega \}$ is an enumeration of a dense subset of $A$.%
\newline
\indent Fix a countable dense subset $\left\{ b_{n}\colon n\in \omega
\right\} $ of $\mathcal{O}_{2}$. Then%
\begin{equation*}
\left\Vert \alpha _{A,p}(x)-y\right\Vert =\left\Vert \left( \mbox{id}%
_{A}\otimes \nu \right) (\eta _{A}(x))-\eta _{A}(y)\right\Vert ,
\end{equation*}%
and thus $\left\Vert \alpha _{A,p}(x)-y\right\Vert <\varepsilon $ if and
only if there are positive integers $k\in \omega $ and $n_{0},\ldots
,n_{k-1},m_{0},\ldots ,m_{k-1}\in \omega $, and scalars $\lambda _{0},\ldots
,\lambda _{k-1}\in \mathbb{Q}(i)$, such that%
\begin{equation*}
\left\Vert \eta _{A}(x)-\sum_{i\in k}\lambda _{i}a_{n_{i}}^{G}\otimes
b_{m_{i}}\right\Vert <\frac{\varepsilon }{2}\ \ \ \mbox{ and }\ \ \
\left\Vert \sum_{i\in k}\lambda _{i}a_{n_{i}}^{G}\otimes \nu \left(
b_{m_{i}}\right) -\eta _{A}(y)\right\Vert <\frac{\varepsilon }{2}\text{.}
\end{equation*}%
Since the map $A\mapsto \eta _{A}$ is Borel, it follows that the set of all
C*-algebras $A$ in $SA_{uKir}\left( \mathcal{O}_{2}\right) $ such that $%
\left\Vert \alpha _{A,p}(x)-y\right\Vert <\varepsilon $ is Borel. The result
follows.
\end{proof}

\subsection{Constructing Kirchberg algebras with assigned $K_{0}$-group}

The following is the main result of this section.

\begin{theorem}
\label{Theorem: Kirchberg algebras from groups}There is a Borel map from the
Borel space $\mathcal{G}$ of discrete groups to the Borel space $\Gamma
_{uKir}(H)$ parametrizing Kirchberg algebras, which assigns to every
countable discrete abelian group $G$, a code $\gamma $ for a unital
Kirchberg algebra $C^{\ast }(\gamma )$ that satisfies the UCT, and with $K$%
-theory given by 
\begin{equation*}
(K_{0}(C^{\ast }(\gamma )),\left[ 1_{C^{\ast }(\gamma )}\right]
,K_{1}(C^{\ast }(\gamma )))\cong (G,0,\{0\}).
\end{equation*}%
In particular, $C^{\ast }(\gamma )$ is $\mathcal{O}_{\infty }^{st}$%
-absorbing. Moreover, if $p$ is a prime number, then $C^{\ast }(\gamma )$ is 
$D_{p}$-absorbing if and only if $G$ is uniquely $p$-divisible.
\end{theorem}

\begin{proof}
Use Lemma \ref{Lemma: exact sequence} to choose, in a Borel way from $G$, a
torsion free abelian group $H$ and an automorphism $\alpha $ of $H$ such
that 
\begin{equation*}
H/\mbox{Im}(\mbox{id}_{H}-\alpha )\cong G.
\end{equation*}%
Denote by $D$ the dimension group given by%
\begin{equation*}
D=\mathbb{Z}\left[ \frac{1}{2}\right] \oplus H
\end{equation*}%
with positive cone%
\begin{equation*}
D^{+}=\left\{ \left( t,h\right) \in D\colon t>0\right\} \cup \left\{ \left(
0,0\right) \right\} ,
\end{equation*}%
and order unit $\left( 1,0\right) $. Consider the endomorphism $\rho $ of $D$
defined by%
\begin{equation*}
\beta \left( t,h\right) =\left( \frac{t}{2},\alpha \left( h\right) \right) 
\text{{}}
\end{equation*}%
for $(t,h)$ in $D$. It is clear that $D$ and $\beta $ can be computed in a
Borel way from $H$ and $\alpha $. By Corollary \ref{Corollary: Borel map
from DG to AF}, one can obtain in a Borel way from $H$ and $\beta $, a code
for a unital AF-algebra $B$ and a code for an injective corner endomorphism $%
\rho $ of $B$ such that the $K_{0}$-group of $B$ is isomorphic to $D$, and
the endomorphism of the $K_{0}$-group of $B$ induced by $\rho $ is conjugate
to $\beta $. By Corollary \ref{Corollary: Borel crossed product endomorphism}%
, one can obtain in a Borel way a code $\gamma _{G}\in \Gamma (H)$ for the
crossed product $B\rtimes _{\rho }\mathbb{N}$ of $B$ by the endomorphism $%
\rho $. It can be shown, as in the proof of \cite[Theorem 3.6]%
{rordam_classification_1995}, that $C^{\ast }(\gamma _{G})$ is a unital
Kirchberg algebra satisfying the UCT, with trivial $K_{1}$-group, $K_{0}$%
-group isomorphic to $G$, and $\left[ 1_{C^{\ast }(\gamma _{G})}\right] =0$
in $K_{0}\left( C^{\ast }\left( \gamma \right) \right) $. An easy
application of the Pimsner-Voiculescu exact sequence \cite[Theorem 2.4]%
{pimsner_exact_1980} gives the computation of the $K$-theory; see \cite[%
Corollary 2.2]{rordam_classification_1995}. Pure infiniteness of $C^{\ast
}(\gamma _{G}) $ is proved in \cite[Theorem 3.1]{rordam_classification_1995}%
. The map $\mathcal{G}\rightarrow \Gamma _{uKir}(H)$ given by $G\mapsto
\gamma _{G}$ is Borel by construction.
\end{proof}

\begin{corollary}
\label{Corollary: reduction of iso of groups to iso of kirchberg} Let $p$ be
a prime number. There are Borel reductions:

\begin{enumerate}
\item From the relation of isomorphism of infinite countable discrete
abelian groups, to the relation of isomorphism of $\mathcal{O}_{\infty
}^{st} $-absorbing unital Kirchberg algebras (satisfying the UCT and with
trivial $K_{1}$-group).

\item From the relation of isomorphism of uniquely $p$-divisible infinite
countable discrete abelian groups, to the relation of isomorphism of $D_{p}$%
-absorbing unital Kirchberg algebras (satisfying the UCT and with trivial $%
K_{1}$-group).
\end{enumerate}
\end{corollary}

\begin{proof}
Both results follow from Theorem \ref{Theorem: Kirchberg algebras from
groups} above, together with the Kirchberg-Phillips classification theorem
(see \cite{kirchberg_exact_1995} and \cite{phillips_classification_2000}).
\end{proof}

\begin{corollary}
\label{Corollary: reduction of iso of p-divisible}Let $p$ be a prime number,
and $\mathcal{C}$ be any class of countable structures such that the
relation $\cong _{\mathcal{C}}$ of elements of $\mathcal{C}$ is Borel.
Assume that $F$ be any of the following equivalence relations:

\begin{itemize}
\item isomorphism of $D_{p}$-absorbing unital Kirchberg algebras with
trivial $K_{1}$-group,

\item conjugacy of approximately representable automorphisms of $\mathcal{O}%
_{2}$ of order $p$,

\item cocycle conjugacy of approximately representable automorphisms of $%
\mathcal{O}_{2}$ of order $p$.
\end{itemize}

Then $\cong _{\mathcal{C}}$ is Borel reducible to $F$, and moreover $F$ is a
complete analytic set.
\end{corollary}

\begin{corollary}
\label{Corollary: not Borel}The relations of isomorphism of unital Kirchberg
algebras, and conjugacy and cocycle conjugacy of automorphisms of $\mathcal{O%
}_{2}$, are complete analytic sets, and in particular, not Borel.
\end{corollary}

\subsection{Actions of countable discrete groups on $\mathcal{O}_{2}$}

Suppose that $G$ is a countable discrete group. Denote by $\mathrm{Act}(G,A)$
the space of actions of $G$ on $A$ endowed with the topology of pointwise
norm convergence. It is clear that $\mathrm{Act}(G,A)$ is homeomorphic to a $%
G_{\delta }$ subspace of the product of countably many copies of $A$ and, in
particular, is a Polish space.

Let $G$ and $H$ be countable discrete groups, and let $\pi \colon
G\rightarrow H$ be a surjective homomorphism from $G$ to $H$. Define the
Borel map $\pi^*\colon \mathrm{Act}(H,A) \to \mathrm{Act}(G,A)$ by $%
\pi^*(\alpha)=\alpha\circ\pi$ for $\alpha$ in $\mathrm{Act}(H,A)$. It is
easy to check that $\pi^*$ is a Borel reduction from the relation of
conjugacy of actions of $H$ to the relation of conjugacy of actions of $G$.
The following proposition is then an immediate consequence of this
observation together with Corollary \ref{Corollary: reduction of iso of
p-divisible}.

\begin{proposition}
Let $G$ be a countable discrete group with a nontrivial cyclic quotient. If $%
\mathcal{C}$ is any class of countable structures such that the relation $%
\cong _{\mathcal{C}}$ of isomorphism of elements of $\mathcal{C}$ is Borel,
then $\cong _{\mathcal{C}}$ is Borel reducible to the relation of conjugacy
of actions of $G$ on $\mathcal{O}_{2}$. \newline
\indent Moreover, the latter equivalence relation is a complete analytic set
as a subset of $\mathrm{Act}( G,A) \times \mathrm{Act}(G,A) $ and, in
particular, is not Borel.
\end{proposition}

The situation for cocycle conjugacy is not as clear. It is not hard to
verify that if $G=H\times N$ and $\pi \colon G\rightarrow H$ is the
canonical projection, then $\pi ^{\ast }$, as defined before, is a Borel
reduction from the relation of cocycle conjugacy in $\mathrm{Act}(H,A)$ to
the relation of cocycle conjugacy in $\mathrm{Act}(G,A)$. Using this
observation and the structure theorem for finitely generated abelian groups,
one obtains as a consequence of Corollary \ref{Corollary: reduction of iso
of p-divisible} the following fact:

\begin{proposition}
Let $G$ be any finitely generated abelian group. If $\mathcal{C}$ is any
class of countable structures such that the relation $\cong _{\mathcal{C}}$
of isomorphism of elements of $\mathcal{C}$ is Borel, then $\cong _{\mathcal{%
C}}$ is Borel reducible to the relation of conjugacy of actions of $G$ on $%
\mathcal{O}_{2}$. \newline
\indent Moreover, the latter equivalence relation is a complete analytic set
as a subset of $\mathrm{Act}(G,A) \times \mathrm{Act}(G,A) $ and, in
particular, not Borel.
\end{proposition}

\section{Final comments and remarks}

Recall that an automorphism of a C*-algebra $A$ is said to be \emph{%
pointwise outer }(or \emph{aperiodic}) if none of its nonzero powers is
inner. By \cite[Theorem 1]{nakamura_aperiodic_2000}, an automorphism of a
Kirchberg algebra is pointwise outer if and only if it has the so called
Rokhlin property. Moreover, it follows from this fact together with \cite[%
Corollary 5.14]{phillips_tracial_2012} that the set $\mathrm{Rok}(A)$ of
pointwise outer automorphisms of a Kirchberg algebra $A$ is a dense $%
G_{\delta }$ subset of $\mathrm{\mathrm{Aut}}(A)$, which is moreover easily
seen to be invariant by cocycle conjugacy.

It is an immediate consequence of \cite[Theorem 9]{nakamura_aperiodic_2000},
see also \cite[Theorem 5.2]{matui_classification_2008}, that aperiodic
automorphisms of $\mathcal{O}_{2}$ form a single cocycle conjugacy class. In
particular, and despite the fact that the relation of cocycle conjugacy of
automorphisms of $\mathcal{O}_{2}$ is not Borel, its restriction to the
comeager subset $\mathrm{Rok}(\mathcal{O}_{2})$ of $\mathrm{\mathrm{Aut}}(%
\mathcal{O}_{2})$ has only one class and, in particular, is Borel. This can
be compared with the analogous situation for the group of ergodic measure
preserving transformations of the Lebesgue space: The main result of \cite%
{foreman_conjugacy_2011} asserts that the relation of conjugacy of ergodic
measure preserving transformations of the Lebesgue space is a complete
analytic set. On the other hand, the restriction of such relation to the
comeager set of ergodic \emph{rank one }measure preserving transformations
is Borel.

It is conceivable that similar conclusions might hold for the relation of
conjugacy of automorphisms of $\mathcal{O}_{2}$. We therefore suggest the
following problem:

\begin{question}
Consider the relation of conjugacy of automorphisms of $\mathcal{O}_{2}$,
and restrict it to the invariant dense $G_{\delta }$ set of aperiodic
automorphisms. Is this equivalence relation Borel?
\end{question}

By \cite[Theorem 4.5]{kerr_Borel_2014}, the automorphisms of $\mathcal{O}
_{2} $ are not classifiable up to conjugacy by countable structures. This
means that if $\mathcal{C}$ is any class of countable structures, then the
relation of conjugacy of automorphisms of $\mathcal{O}_{2}$ is \emph{not}
Borel reducible to the relation of isomorphisms of structures from $\mathcal{%
C}$. It would be interesting to know if one can draw similar conclusions for
the relation of cocycle conjugacy.

\begin{question}
Is the relation of cocycle conjugacy of automorphisms of $\mathcal{O}_{2}$
classifiable by countable structures?
\end{question}

\cite[Theorem 4.5]{kerr_Borel_2014} in fact shows that the relation of
conjugacy of automorphisms is not classifiable for a large class of
C*-algebras, including all C*-algebras that are classifiable according to
the Elliott classification program \cite[Section 2.2]%
{rordam_classification_2002}. It would be interesting to know if the same
holds for the relation of cocycle conjugacy. More generally, it would be
interesting to have some information about the complexity of the relation of
cocycle conjugacy for automorphisms of other simple C*-algebras. This
problem seems to be currently wide open.

\begin{problem}
\label{Problem: other algebras} Find an example of a simple unital nuclear
separable C*-algebra for which the relation of cocycle conjugacy of
automorphisms is not classifiable by countable structures.
\end{problem}

Recall that an equivalence relation on a standard Borel space is said to be 
\emph{smooth}, or \emph{concretely classifiable}, if it is Borel reducible
to the relation of equality in some standard Borel space. A smooth
equivalence relation is in particular Borel and classifiable by countable
structures. Therefore Corollary \ref{Corollary: not Borel} in particular
shows that the relation of cocycle conjugacy of automorphisms of $\mathcal{O}%
_{2}$ is not smooth.

If $X$ is a compact Hausdorff space, we denote by $C(X)$ the unital
commutative C*-algebra of complex-valued continuous functions on $X$. It is
a classical result of Gelfand and Naimark that any unital commutative
C*-algebra is of this form; see \cite[Theorem II.2.2.4]%
{blackadar_operator_2006}. Moreover, by \cite[II.2.2.5]%
{blackadar_operator_2006}, the group $\mathrm{\mathrm{Aut}}(C(X))$ of
automorphisms of $C(X)$ is isomorphic to the group $\mathrm{Homeo}(X)$ of
homeomorphisms of $X$. It is clear that in this case the relations of
conjugacy and cocycle conjugacy of automorphisms coincide. By \cite[Theorem 5%
]{camerlo_completeness_2001}, if $X$ is the Cantor set, then the relation of
(cocycle) conjugacy of automorphisms of $C(X)$ is not smooth (but
classifiable by countable structures). On the other hand, when $X$ is the
unit square $[0,1]^{2}$, then the relation of cocycle conjugacy of
automorphisms of $C(X)$ is not classifiable by countable structures in view
of \cite[Theorem 4.17]{hjorth_classification_2000}. This addresses Problem %
\ref{Problem: other algebras} in the case of abelian unital C*-algebras. No
similar examples are currently known for simple unital C*-algebras.

It is worth mentioning here that if one considers instead the relation of 
\emph{unitary conjugacy} of automorphisms, then there is a \emph{strong
dichotomy} in the complexity. Recall that two automorphisms $\alpha ,\beta $
of a unital C*-algebra are unitarily conjugate if $\alpha \circ \beta ^{-1}$
is an inner automorphism, this is, implemented by a unitary element of $A$. 
Theorem 1.2 in \cite{lupini_unitary_2013} shows that whenever this relation
is not smooth, then it is even not classifiable by countable structures. The
same phenomenon is shown to hold for unitary conjugacy of \emph{irreducible
representations} in \cite[Theorem 2.8.]{kerr_turbulence_2010}; see also \cite%
[Section 6.8]{pedersen_c-algebras_1979}. It is possible that similar
conclusions might hold for the relation of conjugacy or cocycle conjugacy of
automorphisms of simple C*-algebras.

\begin{question}
Is it true that, whenever the relation of (cocycle) conjugacy of
automorphisms of a simple unital C*-algebra $A$ is not smooth, then it is
not even classifiable by countable structures?
\end{question}

Kirchberg-Phillips classification theorem (Theorem \ref{Theorem: KP
classification}) asserts that Kirchberg algebras satisfying the UCT are
classified up to isomorphism by their $K$-groups. By \cite[Section 3.3]%
{farah_descriptive_2012} the $K$-theory of a C*-algebra can be computed in a
Borel way. It follows that Kirchberg algebras satisfying the UCT are
classifiable up to isomorphism by countable structures. Conversely, by
Corollary \ref{Corollary: reduction of iso of p-divisible} if $\mathcal{C}$
is any class of countable structure \emph{with Borel isomorphism relation},
then the relation of isomorphism of elements of $\mathcal{C}$ is Borel
reducible to the relation of isomorphism of Kirchberg algebras satisfying
the UCT. It is natural to ask whether the same conclusion holds for\emph{\
any }class of countable structures $\mathcal{C}$.

\begin{question}
\label{Question: Kirchberg algebras}Suppose that $\mathcal{C}$ is a class of
countable structures. Is the relation of isomorphism of elements of $%
\mathcal{C}$ Borel reducible to the relation of isomorphism of Kirchberg
algebras with the UCT?
\end{question}

A class $\mathcal{D}$ of countable structures is \emph{Borel complete }if
the following holds:\ For any class of countable structures $\mathcal{C}$
the relation of isomorphism of elements of $\mathcal{C}$ is Borel reducible
to the relation of isomorphism of elements of $\mathcal{D}$. Theorem 1,
Theorem 3, and Theorem 10 of \cite{friedman_borel_1989} assert that the
classes of countable trees, countable linear orders, and countable fields of
any fixed characteristic are Borel complete; Theorem 7 of \cite%
{friedman_borel_1989} shows, using results of Mekler from \cite%
{mekler_stability_1981}, that the relation of isomorphism of countable
groups is Borel complete. A long standing open problem --first suggested in 
\cite{friedman_borel_1989}-- asks whether the class of (torsion-free)
abelian groups is Borel complete. In view of Corollary \ref{Corollary:
reduction of iso of groups to iso of kirchberg}, a positive answer to such
problem would settle Question \ref{Question: Kirchberg algebras}
affirmatively.

\providecommand{\bysame}{\leavevmode\hbox to3em{\hrulefill}\thinspace} %
\providecommand{\MR}{\relax\ifhmode\unskip\space\fi MR } 
\providecommand{\MRhref}[2]{  \href{http://www.ams.org/mathscinet-getitem?mr=#1}{#2}
} \providecommand{\href}[2]{#2}

\end{document}